\input ./style/arxiv-general.cfg
\documentclass[aop,MSNbibl,seceqn,dvips]{arximspdf}
\makeatletter
   \@ifpackageloaded{graphicx}{}{\usepackage{graphicx}}
\makeatother
\usepackage{mathbh}

\doi{10.1214/14-AOP988}% Updated by VTEXPTS2LaTeX.exe, 05.01.2015 09:20
\volume{44}
\issue{1}
\pubyear{2016}
\firstpage{739}
\lastpage{805}
\docsubty{FLA}

\makeatletter
\newcommand{\rrvert}{\vert}
\newcommand{\rrVert}{\Vert}
\newcommand{\llvert}{\vert}
\newcommand{\llVert}{\Vert}
\renewcommand{\mid}{|}
\newcommand{\R}{\mathbb{R}}
\newcommand{\Z}{\mathbb{Z}}
\newcommand{\N}{\mathbb{N}}
\newcommand{\Exp}{\operatorname{Exp}}
\newcommand{\Sfl}{\mathcal{S}_{1}}
\renewcommand{\d}{\mathrm{d}}
\newtheorem{teo}{Theorem}[section]
\newproclaim{defn}[teo]{Definition}
\newtheorem{lem}[teo]{Lemma}
\newtheorem{prop}[teo]{Proposition}
\newtheorem{cor}[teo]{Corollary}
\newproclaim{rem}[teo]{Remark}
\makeatother

\begin{document}
\begin{frontmatter}

\title{Behavior near the extinction time in self-similar fragmentations II: Finite dislocation measures}
\runtitle{Behavior near the extinction time in fragmentations}

\begin{aug}
% Corresponding author: Bénédicte Haas - haas@ceremade.dauphine.fr% Updated by VTEXPTS2LaTeX.exe, 06.01.2015 09:18
%Updated by VTEXPTS2LaTeX.exe, 05.01.2015 09:20
\author[A]{\fnms{Christina}~\snm{Goldschmidt}\thanksref{T1}\ead[label=e1]{goldschm@stats.ox.ac.uk}}
\and
\author[B]{\fnms{B\'en\'edicte}~\snm{Haas}\corref{}\thanksref{T2}\ead[label=e2]{haas@ceremade.dauphine.fr}}
\runauthor{C. Goldschmidt and B. Haas}
\affiliation{University of Oxford and Universit\'e Paris-Dauphine}
%\dedicated{}
\address[A]{Department of Statistics\\
\quad and Lady Margaret Hall\\
University of Oxford\\
1 South Parks Road\\
Oxford\\
OX1 3TG\\
United Kingdom\\
\printead{e1}}
\address[B]{CEREMADE\\
Universit\'e Paris-Dauphine\\
Place du Mar\'echal De Lattre De Tassigny\\
75016 Paris\\
France\\
\printead{e2}}
\end{aug}
\thankstext{T1}{Supported in part by EPSRC Postdoctoral Fellowship
EP/D065755/1.}
\thankstext{T2}{Supported in part by ANR-08-BLAN-0190 and ANR-08-BLAN-0220-01.}

% HISTORY:
%
\received{\smonth{9} \syear{2013}}% Updated by VTEXPTS2LaTeX.exe,
%05.01.2015 09:20
%
\revised{\smonth{9} \syear{2014}}% Updated by VTEXPTS2LaTeX.exe,
%05.01.2015 09:20

% ABSTRACT
%
\begin{abstract}
We study a Markovian model for the random fragmentation of an object.
At each time, the state consists of a collection of blocks. Each block
waits an exponential amount of
time with parameter given by its size to some power~$\alpha$,
independently of the other blocks. Every block then splits randomly
into sub-blocks whose
relative sizes are distributed according to the so-called dislocation
measure. We focus here on
the case where $\alpha< 0$. In this case, small blocks split
intensively, and so the whole
state is reduced to ``dust'' in a finite time, almost surely (we call this
the extinction time). In this paper, we investigate how the
fragmentation process
behaves as it approaches its extinction time. In particular, we prove a
scaling limit for the block sizes which, as a direct consequence, gives
us an expression for an invariant measure for the fragmentation
process. In an earlier paper %[Goldschmidt and Haas
[\textit{Ann. Inst. Henri Poincar\'e Probab. Stat.} \textbf{46} (2010)  338--368], we
considered the same problem for another family of fragmentation
processes, the so-called \emph{stable fragmentations}. The results
here are similar, but we emphasize that the methods used to prove them
are different. Our approach in the present paper is based on Markov
renewal theory and involves a somewhat unusual ``spine'' decomposition
for the fragmentation, which may be of independent interest.
\end{abstract}

% KEYWORDS
\end{frontmatter}

\tableofcontents

%s1 #&#
\section{Introduction and main results}

We consider a Markovian model for the random fragmentation of a
collection of blocks of some material, where the manner in which the
fragmentation occurs is controlled solely by the masses of the blocks.
More specifically, suppose that the current state consists of blocks of
masses $m_1, m_2, \ldots$ which are such that (for definiteness)
$\mathbf{m} = (m_1, m_2, \ldots)$ belongs to the state-space
\[
\mathcal{S}:= \Biggl\{ \mathbf s = (s_1, s_2, \ldots)\dvtx
s_1 \ge s_2 \ge\cdots\ge0, \sum
_{i=1}^{\infty} s_i < \infty\Biggr\},
\]
which is endowed with the $\ell_1$-distance
\[
d\bigl(\mathbf{s}, \mathbf{s}'\bigr) = \bigl\llVert\mathbf{s} -
\mathbf{s}' \bigr\rrVert_1:= \sum
_{i \ge1} \bigl\llvert s_i - s_i'
\bigr\rrvert\qquad\mbox{for } \mathbf s, \mathbf s' \in\mathcal S.
\]
The transition mechanism depends on two parameters: a real number
$\alpha$ and a probability measure $\nu$ on $\mathcal{S}_1:=
\{\mathbf{s} \in\mathcal{S}\dvtx  \llVert \mathbf{s}\rrVert _1 = 1 \}$, and
can be described as follows. The different blocks evolve independently.
For $i \ge1$, block $i$ splits after an exponential time of mean
$m_i^{-\alpha}$ into sub-blocks of masses $m_i \mathbf{S}$, where the
random sequence $\mathbf{S} = (S_1, S_2, \ldots)$ is distributed
according to $\nu$. To avoid ``phantom'' fragmentation events, we will
always assume that $\nu(\mathbf{1}) = 0$, where the state $\mathbf
{1} = (1,0, \ldots)$ consists of a single block of mass 1. We will
then write
\[
F(t) = \bigl(F_1(t), F_2(t), \ldots\bigr) \in\mathcal{S}
\]
for the state of the fragmentation process at time $t$, and \mbox{$\mathbb
{P}_{\mathbf{s}}$} for the law of \mbox{$(F(t), t \ge0)$} started from a
state $\mathbf{s} \in\mathcal{S}$. By default, we will start our
processes from the state $\mathbf{1}$, and we will write $\mathbb{P}$
instead of $\mathbb{P}_{\mathbf{1}}$. Whenever we write $(F(t), t\ge
0)$ without making explicit reference to its law, we implicitly assume
$F(0) = \mathbf{1}$. It is clear that (whatever its starting point)
$(F(t), t \ge0)$ is a transient Markov process with a single absorbing
state at $\mathbf{0} = (0, 0, \ldots)$.

This model described in the previous paragraph is a \emph{self-similar
fragmentation process}, as introduced by Filippov in \cite{Filippov}
and Bertoin in \cite{BertoinHom,BertoinSSF}. We refer to the second
pair of papers for a rigorous construction based on Poisson point
processes. This construction gives a version of the fragmentation which
is c\`adl\`ag for the topology of pointwise convergence. Proposition~1.9 of \cite{BertoinBook} shows, in addition, that the sum of the
masses of the blocks is a continuous function almost surely. Hence,
there exists a c\`adl\`ag version of the fragmentation for the $\ell
_1$-distance, which is the version we will always consider in this
paper. More precisely, $(F(t), t \ge0)$ is a c\`adl\`ag strong Markov
process which possesses the following self-similarity property:
\[
\begin{tabular}{p{306pt}}
$(F(t), t \ge0)$ has the same distribution under
$\mathbb {P}_{m \mathbf{1}}$ as  $(m F(m^{\alpha} t), t
\ge0)$ has under $\mathbb{P}_{\mathbf{1}}$
\end{tabular}
\]
(we will revisit a stronger version of this property in
Proposition~\ref{strongMarkov} below). Consequently, the parameter
$\alpha$ is known as the \emph{index of self-similarity}. The
probability measure $\nu$ is called the \emph{dislocation measure}.
In \cite{BertoinHom,BertoinSSF}, Bertoin constructs a more general
class of processes in which $\nu$ is allowed to be an infinite (but
$\sigma$-finite) measure satisfying a certain integrability condition;
roughly speaking, these processes are allowed to jump at a dense set of
times. He also allows dislocation measures which do not preserve the
original mass, and the possibility of deterministic erosion of the
block masses, but we will not consider any of these variants further here.

Henceforth, we will restrict our attention to the case $\alpha< 0$. In
this case, smaller blocks split (on average) faster than larger ones.
Despite the fact that each splitting event preserves the total mass
present in the system, the fragmentation exhibits the striking
phenomenon of \emph{loss of mass}, whereby splitting events accumulate
in such a way that blocks are reduced in finite time to blocks of mass
0 (known as \emph{dust}). This is reflected by the fact that the total
mass $M(t)=\sum_{i \geq1} F_i(t)$ decreases as time passes [so that
the dust has mass $1-M(t)$]. Moreover, if we define the \emph
{extinction time},
\[
\zeta= \inf\bigl\{t \ge0\dvtx  F(t) = \mathbf{0}\bigr\},
\]
then $\zeta< \infty$ almost surely; see \cite{BertoinAB}. The manner
in which mass is lost has been studied in detail by Bertoin~\cite
{BertoinAB} and Haas \cite{HaasLossMass,HaasRegularity}. Our focus
here is different: we aim to understand the behavior of the
fragmentation process close to its extinction time.

In most of the sequel, we will impose a further condition on the
dislocation measure $\nu$: we will require it to be \emph{nongeometric}. That is, for any $r \in(0,1)$, we have
\[
\nu\bigl( s_i \in r^{\N} \cup\{0\}, \forall i \ge1 \bigr)
< 1
\]
(where $\N:= \{1,2,\ldots\}$). Fragmentations with geometric
dislocation measures behave in a genuinely different way to their
nongeometric counterparts; we will discuss this difference further
below.\vspace*{1pt} For technical reasons, we will also need to \mbox{impose} the condition
that $\int_{\mathcal S_1}s_1^{-1-\rho}\nu(\mathrm d \mathbf
s)<\infty$ for some $\rho>0$. This assumption is not very
restrictive: for example, it is always satisfied for fragmentations
where blocks split into at most $N$ sub-blocks ($N$ being fixed) since
then $s_1+\cdots+s_N=1$, and so the largest mass $s_1$ is bounded
below by $1/N$ $\nu$-a.s.

We consider the usual Skorokhod topology on the space of c\`adl\`ag
functions $f\dvtx  [0,\infty) \to\mathcal{S}$.
By convention, we will set $F(t) = \mathbf{1}$ for $t < 0$. Our
principal result is then the following theorem.

%th1.1 #&#
\begin{teo} \label{teomainabstract}
Suppose that $\nu$ is nongeometric and that $\int_{\Sfl}
s_1^{-1-\rho} \nu(\d\mathbf s) < \infty$ for some $\rho>0$. Then
there exists $C_{\infty}$, a c\`adl\`ag $\mathcal{S}$-valued
self-similar process independent of $\zeta$, such that
\[
\bigl(\varepsilon^{1/\alpha} \bigl(F\bigl( (\zeta- \varepsilon t)-\bigr), t
\geq0
\bigr),\zeta\bigr) \stackrel{\mathit{law}} {\rightarrow} \bigl(
\bigl(C_{\infty}(t), t \geq0 \bigr),\zeta\bigr).
\]
Moreover, $C_{\infty}(0)=\mathbf0$ and $\mathbb P(C_{\infty,i}(1)>0)>0$
for all $i \geq1$.
\end{teo}

In particular, as $\varepsilon\to0$,
\[
\varepsilon^{1/\alpha} F( \zeta- \varepsilon) \stackrel{\mathrm{law}} {
\rightarrow} C_{\infty}(1).
\]
Since $\mathcal S$ is endowed with the $\ell_1$-distance, this entails
that the rescaled total mass
$
\varepsilon^{1/\alpha}M(\zeta-\varepsilon)
$
has a nontrivial limit in distribution as $\varepsilon\to0$.

The self-similarity of the limit process $C_{\infty}$ takes the form
\[
\bigl(a^{1/\alpha} C_{\infty}(a t), t \ge0\bigr) \stackrel{\mathrm
{law}} {=}\bigl(C_{\infty}(t), t \ge0\bigr)
\]
for all $a > 0$.
We will specify the distribution of $C_{\infty}$ more precisely below
once we have established the necessary notation; see Definition~\ref
{defnCinfty}. This process models the evolution of masses that
coalesce, with a regular immigration of infinitesimally small masses,
as illustrated in Figure~\ref{figCinfini}. Reversing time, this gives
a fragmentation process that starts from one infinitely large mass. A
connection with a biased randomized version of $F$ is made in
Proposition~\ref{CF}.

In a first paper \cite{GoldschmidtHaas}, we proved a result of the
same form as Theorem~\ref{teomainabstract} for a different subclass
of self-similar fragmentations with negative index, the stable
fragmentations. The stable fragmentations, which were introduced in
\cite{Miermont}, are qualitatively rather different in that they all
have infinite dislocation measures. They can be represented in terms of
stable L\'evy trees (see \cite{DLG02,DLG05} for a definition), and the
methods used in our earlier paper rely crucially on the excursion
theory available for these trees. The methods used in the present work
are quite different and are dependent on the finiteness of the
dislocation measure. We conjecture, nonetheless, that Theorem~\ref
{teomainabstract} is true for generic nongeometric self-similar
fragmentations with negative index.

The proof of Theorem~\ref{teomainabstract} proceeds in two main
steps. We begin by studying the \textit{last fragment process} $F_*$,
where $F_*(t)$ is the mass of the unique fragment present at time $t$
that dies exactly at time $\zeta$. We construct this process in
Section~\ref{Lastfrag}, where we also discuss some properties of
$\zeta$. We are, of course, interested in the asymptotic behavior of
$F_*$ close to time $\zeta$. A significant difficulty is that
the evolution of the process $F_*$ is not Markovian. To overcome this
difficulty, we introduce the discrete-time process
\[
Z_n=F_*(T_n)^{\alpha}(\zeta-T_n),
\qquad n \geq0,
\]
where $T_n$ denotes the $n$th jump time of the last fragment process
$F_*$. The quantity $Z_n$ can be thought of as an updated notion of the
extinction time seen in the natural timescale of the last fragment at
its $n$th jump time. It turns out that $(Z_n)_{n \ge0}$ is a Markov
chain which converges to a stationary distribution as $n \rightarrow
\infty$. This is proved in Section~\ref{cvZn} using standard
Foster--Lyapunov criteria. Moreover, the Markov chain $(Z_n)_{n \ge0}$
drives a bigger Markov chain which additionally tracks the relative
sizes of the fragments produced by the split at time $T_n$. From this
bigger Markov chain we derive a Markov renewal process in Section~\ref
{MarkovRW}, and we then use a version of renewal theory, developed for
such processes in \cite
{Alsmeyer2,Alsmeyer,AthreyaRen,JacodRen,KestenRen,OreyRen,ShurenkovRen},
to obtain the behavior of $F_*$ near $\zeta$.

The second step of the proof consists of decomposing the fragmentation
process along its \textit{spine} $F_*$, in order to get the behavior
of the whole process near $\zeta$. This is the purpose of
Sections~\ref{Spine} and~\ref{fullfrag}, where we prove a detailed
version of Theorem~\ref{teomainabstract}. Roughly speaking, the
limiting process $C_{\infty}$ is built from a spine, the limit process
of $F_*$ near $\zeta$, by grafting onto it independent fragmentation
processes conditioned to die before specific times. A significant
technical difficulty in this proof is to deal with blocks which
separated from the spine ``a long time in the past'' and have not yet
become extinct, and for this we will need to establish a tightness criterion.

Spine methods are standard in the study of branching processes. In
earlier work on fragmentation processes (e.g., in \cite
{BertoinSSF,BertoinAB}), the so-called \emph{tagged fragment} has
proved to be a very useful tool. This is again a sort of spine but of a
rather different nature to ours (in particular, the tagged fragment is
a Markov process). However, the tagged fragment vanishes at a time
which is strictly smaller than $\zeta$ and, as a consequence, cannot
help us to understand the behavior of the fragmentation near its
extinction time $\zeta$. We believe that the spine decomposition we
develop in the present paper, based on the last fragment process,
should not be particular to the finite dislocation measure case.
However, our results do not immediately extend to the case of infinite
dislocation measures.

As a direct consequence of Theorem~\ref{teomainabstract}, we are able
to construct an invariant measure for the fragmentation process (since
$F$ is transient, this is necessarily an infinite measure).

%th1.2 #&#
\begin{teo}
\label{teoinvariant}
Under the conditions of Theorem~\ref{teomainabstract}, consider the
occupation measure $\lambda$ of $C_{\infty}$, which is defined on
$(\mathcal{S}, \mathcal{B}(\mathcal{S}))$ by
\[
\lambda(A) = \int_0^{\infty} \mathbb{P}
\bigl(C_{\infty}(t) \in A \bigr)\, \mathrm{d} t
\]
for all $A \in\mathcal{B}(\mathcal{S})$.
Then $\lambda$ is a $\sigma$-finite invariant measure for the
transition kernel of the fragmentation process $F$; that is, for all
$u> 0$ and all $A \in\mathcal{B}(\mathcal{S})$,
\[
\lambda(A) = \int_{\mathcal{S}} \mathbb P_{\mathbf s }{\bigl(F(u)
\in A\bigr) } \lambda(\mathrm{d} \mathbf{s}).
\]
\end{teo}

We can interpret $\lambda$ heuristically as the ``law'' of $C_{\infty
}$ ``sampled at a uniform time in $[0, \infty)$.'' To the best of our
knowledge, this is the first time that invariant measures have been
considered for self-similar fragmentation processes. Theorem~\ref
{teoinvariant} is proved in Section~\ref{secinvariant}, where we
will see that it is an easy consequence of the convergence in
distribution of $\varepsilon^{1/\alpha} F(\zeta-\varepsilon)$ to
$C_{\infty}(1)$. In particular, this invariance result also holds for
the stable fragmentations and, more generally, for any fragmentation
process such that $\varepsilon^{1/\alpha} F(\zeta-\varepsilon)$ has a
nontrivial limit in distribution [in $(\mathcal S,d)$] as $\varepsilon
\rightarrow0$.

We conclude the main part of the paper in Section~\ref{geofrag} by
investigating the case of geometric fragmentations. These
fragmentations should not be viewed simply as a degenerate special
case: they can be interpreted in terms of various other models, in
particular discounted branching random walks (introduced by Athreya~\cite{Athreya}) and randomly growing $k$-ary trees (studied by Barlow,
Pemantle and Perkins \cite{BPP}). Theorem~\ref{teomainabstract} is
not valid for geometric fragmentations. Indeed, we will see in
Proposition~\ref{propgeo} that the rescaled sequence $\varepsilon
^{1/\alpha} F (\zeta- \varepsilon)$ does not converge in
distribution in this situation. However, we do obtain convergence along
suitable subsequences, which entails the existence of a continuum set
of distinct invariant measures, indexed by $x \in[0,1)$.

\hyperref[appendix]{Appendix} containing various technical
lemmas. It is split into two sections. The first concerns criteria for
convergence in the space $(\mathcal{S},d)$ and in the Skorokhod
topology on c\`adl\`ag processes taking values in $(\mathcal{S},d)$.
The second section contains the proofs of fine results about stationary
and biased versions of the Markov chain $(Z_n)_{n \ge0}$ which are
necessary for the proof of Theorem~\ref{teomainabstract} but which
are not of much intrinsic interest.

%%%%%%%%%%%%%%%%%%%%%%%%%%%%%%%
%s2 #&#
\section{The last fragment process}\label{Lastfrag}
%%%%%%%%%%%%%%%%%%%%%%%%%%%%%%%

In this section, we gather together some results on the extinction time
$\zeta$ and prove the existence of the last fragment process.
We refer to \cite{BertoinSSF,BertoinBook} for background on
fragmentation processes. In particular, we will use the following
\textit{strong fragmentation property} on several occasions.

%pr2.1 #&#
\begin{prop}[(Bertoin \cite{BertoinSSF})]
\label{strongMarkov}
Let $T$ be a stopping time with respect to the filtration generated by
$F$. Write, for $t \geq T$,
\[
F(t) = \bigl(F^{(1,T)}(t), F^{(2,T)}(t), \ldots\bigr),
\]
where, for each $i \geq1$, $F^{(i,T)}$ is the process evolving in
$\mathcal S$ which has $F^{(i,T)}(T) = F_i(T)$ and, for
$t > T$, tracks the evolution of the fragments coming from the $i$th
block of $F(T)$. Then
\[
F^{(i,T)}(T+t)=F_{i}(T)G^{(i)}\bigl(tF_{i}(T)^{\alpha}
\bigr) \qquad\forall i \geq1,
\]
where the processes $G^{(i)}$ are independent and have the same
distribution as $F$. They are also independent of $T$ and
$F(T)$.
\end{prop}

%s2.1 #&#
\subsection{The extinction time} \label{secextinctiontime}

We begin by establishing some properties of
the
extinction time $\zeta$, which will be useful to us in the sequel.
We will make use of Proposition 14 of \cite{HaasLossMass}, which
states that
\[
\mathbb{E} \bigl[\exp(a\zeta) \bigr] < \infty\qquad\mbox{for all
positive $a$
sufficiently small.}
\]

%le2.2 #&#
\begin{lem} \label{LemmaDensity}
The distribution of $\zeta$ is absolutely continuous with respect
to Lebesgue measure on $(0,\infty)$, and there exists a continuous and strictly
positive version of its density, which we denote
$f_{\zeta}$. Furthermore:
\begin{longlist}[(iii)]
\item[(i)] $f_{\zeta}(x) \leq1$ for all $x \in(0,\infty)$;

\item[(ii)] $f_{\zeta}(x)=o(\exp(-cx))$ as $x \to\infty
$, for some $c>0$;

\item[(iii)] $f_{\zeta}(x)=o(1)$ as $x \to0$ and,
moreover, for each $\beta>\alpha$
such that $\int_{\mathcal S_1}s_1^{-\beta} \nu(\mathrm
d \mathbf s)<\infty$, $\mathbb F_{\zeta}(x):= \mathbb{P}
(\zeta\leq x ) =\mathcal O(x^{1-\beta/\alpha})$.
\end{longlist}
\end{lem}

\begin{pf} Let $T_1:=\inf\{t \geq0\dvtx F(t)\neq(1,0,\ldots) \}$ be the
first splitting time of~$F$. Then $T_1$ is exponentially distributed
with parameter 1, and $F(T_1)$ is distributed according to $\nu$.
Moreover, since $T_1$ is a stopping time with respect to the filtration
generated by $F$, we get from Proposition~\ref{strongMarkov} that
\[
\zeta= T_1+\sup_{i\geq1}\bigl\{F_i(T_1)^{-\alpha}
\zeta^{(i)} \bigr\},
\]
where
$T_1$, $F(T_1)$ and $(\zeta^{(i)},i \geq1)$ are independent, and
$(\zeta^{(i)},i \geq1)$ is a collection of independent random
variables with
the same distribution as $\zeta$.
Since $T_1$ has an exponential
distribution, this implies that $\zeta$ possesses a density, say
$f_{\zeta}$, which in turn implies that $\xi:=\sup_{i\geq
1}\{F_i(T_1)^{-\alpha}\zeta^{(i)}\}$ possesses a density, given by
%
%e2.1 #&#
\begin{equation}
\label{EqZetamax} f_{\xi}(y)=\int_{\mathcal
S_1}\sum
_{i\dvtx s_i>0}f_{\zeta}\bigl(s_i^{\alpha}y
\bigr)s_i^{\alpha}\prod_{j\neq
i}\mathbb F_{\zeta}
\bigl(s_j^{\alpha}y\bigr) \nu(\mathrm d \mathbf s),
\end{equation}
where $\mathbb F_{\zeta}$ is the cumulative distribution function
corresponding to $f_{\zeta}$.
Note that if $\mathbb F_{\zeta}(s_j^{\alpha}y)>0$, for all\vspace*{2pt} $j\neq i$, then
necessarily $\prod_{j\neq i}\mathbb F_{\zeta}(s_j^{\alpha}y)>0$. This is
obvious when the set $\{j\dvtx s_j>0\}$ is finite. When it is infinite,
taking logarithms and using the fact that
\[
\log\bigl(\mathbb F_{\zeta}\bigl(s_j^{\alpha}y\bigr)
\bigr)\sim-\mathbb P \bigl(\zeta>s_j^{\alpha}y\bigr)
\]
as $j \rightarrow\infty$, we see that the above product is null
if and only if the sum $\sum_{j\neq i} \mathbb P
(\zeta>s_j^{\alpha}y)$ is infinite. But this never happens when
$\sum_{j\neq i}s_j\leq1$, since
\[
\mathbb P \bigl(\zeta>s_j^{\alpha}y\bigr) \leq\mathbb E
\bigl[\zeta^{-1/\alpha} \bigr]s_j y^{-1/\alpha}
\]
and $\zeta$ has exponential moments.

Now, choose $f_{\zeta}$ so that
%
%e2.2 #&#
\begin{equation}
\label{Zintermsofxi} f_{\zeta}(x)=\exp(-x)\int_{0}^{x}
\exp(y)f_{\xi}(y)\, \mathrm{d}y\qquad\mbox{for all }x>0.
\end{equation}
Then, $f_{\zeta}$ is continuous and $f_{\zeta}(x)\leq
\mathbb{P}(\xi\leq x) \rightarrow0$ as $x\rightarrow0$. In
particular, we get (i) and the first assertion of (iii). Note also
that if $f_\zeta(x)=0$ for some $x>0$, then $f_{\xi}$ equals $0$
a.e. on $[0,x]$. Hence, using (\ref{EqZetamax}) and the remark
following it, we see that $f_{\zeta}$ equals $0$ on $[0,x']$ for
some $x'>x$. This easily entails that $f_{\zeta}$ equals
$0$ on~$\mathbb R_+$, which is impossible. Hence, $f_{\zeta}(x)>0$
for all $x>0$.

Next, to prove (ii), note that for all $0 \leq a\leq1$,
\begin{eqnarray*}
\exp(ax)f_{\zeta}(x) &\leq&\int_{0}^{x}
\exp(ay) f_{\xi}(y) \,\mathrm d y \leq\mathbb{E} \bigl[ \exp(a\xi)
\bigr] \leq\mathbb{E} \bigl[ \exp(a\zeta) \bigr]
\end{eqnarray*}
since $\zeta=T_1+\xi$. The last expectation is finite for all
positive $a$ sufficiently small, and so $\exp(cx)f_{\zeta
}(x)\rightarrow0$ as $x\rightarrow\infty$ for all $c<a$.

It remains to prove the second assertion of (iii). Let $\Gamma:=\{
\gamma\geq0\mbox{ s.t. }\exists C_{\gamma}<\infty\dvtx  \mathbb F_{\zeta
}(x)\leq
C_{\gamma} x^{\gamma}, \forall x \geq0 \}$. Since $\mathbb F_{\zeta
}$ is
smaller than 1, $\Gamma$ is an interval whose left endpoint is
$0$. Moreover, since $f_{\zeta}(x)\leq1$ for all $x >0$, we have
$[0,1] \subseteq\Gamma$. In particular, we have checked the assertion
for $\beta\leq0$. Now consider $\gamma\in\Gamma$. We have
\begin{eqnarray*}
f_{\zeta}(x)&\leq&\mathbb{P}(\xi\leq x) \leq \int_{\mathcal
S_1}
\mathbb F_{\zeta}\bigl(s_1^{\alpha}x\bigr) \nu(\mathrm
d \mathbf s)\qquad \bigl[\mbox{since }\xi\geq F_1^{-\alpha}(T_1)
\zeta^{(1)}\bigr]
\\
&\leq& C_{\gamma}x^{\gamma} \int_{\mathcal S_1}
s_1^{\alpha\gamma} \nu(\mathrm d \mathbf s),
\end{eqnarray*}
which implies that $\gamma+1$ is in $\Gamma$ provided that
$\int_{\mathcal S_1} s_1^{\alpha\gamma} \nu(\mathrm d
\mathbf s) <\infty$. The second assertion of (iii) is then straightforward.
\end{pf}

%s2.2 #&#
\subsection{Building the last fragment}

For all $t \geq0$ and all $i \in\mathbb N$, denote by
$F^{(i,t)}$ the fragmentation process\vspace*{1pt} starting from
$(F_i(t),0,\ldots.)$ which tracks the evolution of the masses
emanating from $F_i(t)$. Let $Z^{(i,t)}:=\inf\{s\geq0\dvtx
F^{(i,t)}(s)=\mathbf{0}\}$ be the first time at which this process
is reduced to dust.

%le2.3 #&#
\begin{lem} \label{LemmaLastFragment}
Almost\vspace*{2pt} surely, for all $ 0 \leq t <
\zeta$, there exists a unique index $i(t)$ such that
$Z^{(i(t),t)}=\sup_{j\in\mathbb N}Z^{(j,t)}=\zeta-t$.
\end{lem}

\begin{pf}
Fix $t >0$. By Proposition~\ref{strongMarkov},
$Z^{(i,t)}=F_i(t)^{-\alpha}\zeta^{(i,t)}$, where\break
$(\zeta^{(i,t)}, i \geq1) $ is a collection of i.i.d. random variables,
with the same distribution as $\zeta$, independent of $F(t)$.
Hence
\[
\mathbb E \biggl[ \sum_{i\geq1} \bigl(Z^{(i,t)}
\bigr)^{-1/\alpha} \biggr]=\mathbb E \biggl[ \sum_{i\geq1}
F_i(t) \bigl(\zeta^{(i,t)}\bigr)^{-1/\alpha} \biggr] \leq
\mathbb E \bigl[ \zeta^{-1/\alpha} \bigr] <\infty.
\]
In particular, the sum $\sum_{i\geq1}
(Z^{(i,t)})^{-1/\alpha}$ is almost surely finite, which
implies that $Z^{(i,t)} \rightarrow0$ a.s. as $i
\rightarrow\infty$. Hence,\vspace*{1pt} the supremum $\sup_{j\in\mathbb
N}Z^{(j,t)}$ is attained for some $i \in\N$. Conditional on
$t< \zeta$, this index $i$ is necessarily a.s. unique, since
\begin{eqnarray*}
&&\mathbb{P} \bigl( \exists k,j\dvtx F_{k}^{-\alpha}(t)\zeta
^{(k,t)}=F_{j}^{-\alpha}(t)\zeta^{(j,t)},
F_{k}(t)\neq0, F_{j}(t)\neq0 \bigr)=0
\\
&&\qquad \Leftrightarrow\quad \forall k,j,\qquad\mathbb{P} \bigl(F_{k}^{-\alpha}(t)
\zeta^{(k,t)}=F_{j}^{-\alpha
}(t)\zeta^{(j,t)},
F_{k}(t)\neq0, F_{j}(t)\neq0 \bigr)=0,
\end{eqnarray*}
which is clearly satisfied, since $\zeta^{(k,t)}$ and $\zeta
^{(j,t)}$ are absolutely continuous (by Lemma~\ref{LemmaDensity})
and independent of $F(t)$. Hence, conditionally on $t <\zeta$,
there almost surely exists a unique index $i(t)$ such that
$Z^{(i(t),t)}=\sup_{j\in\mathbb N}Z^{(j,t)}$. To
conclude, note that when $i(t)$ exists and is unique, then, for
all $s \leq t$, $i(s)$ is automatically defined as the index of
the ancestor at time $s$ of $F_i(t)$. Therefore, with probability
one, the indices $i(t)$ are well defined for all $0 \leq t <
\zeta$.
\end{pf}

Let $(\Omega,\mathcal F)$ denote the measurable space on which we work.

%de2.4 #&#
\begin{defn}
Let $E:=\{\omega\in\Omega\dvtx  \forall t < \zeta(\omega), \exists!$
$i(t)(\omega)$ s.t. $Z^{(i(t)(\omega),t)}(\omega)=\sup_{j\in\mathbb
N}Z^{(j,t)}(\omega) \}$, and define for all $t \geq0$,
\[
F_{\ast}(t) (\omega)= \cases{ F_{i(t)(\omega)}(t) (\omega), &\quad if
$\omega\in E$ and $t<\zeta(\omega)$,
\cr
0, &\quad otherwise.}
\]
The process $F_{\ast}$ is called the \emph{last fragment
process}. It is nonincreasing, c\`adl\`ag and $\zeta=\inf\{t \ge0\dvtx
F_{\ast}(t)=0\}$ a.s. (by Lemma~\ref{LemmaLastFragment}).
\end{defn}

%re2.5 #&#
\begin{rem}
\label{remfinitejumps}
Almost surely, for all $t \geq0$, $F_*(t)>0$ implies that the number
of jumps of $F_*$ in $[0,t]$ is finite. This is obvious if $\nu(s_1
\leq a)=1$ for some $a<1$. Otherwise, it can be easily seen via the
Poissonian construction of the fragmentation in \cite{BertoinHom,BertoinSSF}.
\end{rem}

In the sequel, we will use the last fragment as a ``spine'' for the
fragmentation process: when blocks separate from the last fragment,
they evolve essentially as independent fragmentation processes which
are conditioned to die before the last fragment. We emphasize that it
is not measurable with respect to the natural filtration of the
fragmentation process.

%%%%%%%%%%%%%%%%%%%%%%%%%%%%%%%%%%%%%%
%s3 #&#
\section{Asymptotics along a subsequence}\label{cvZn}
%%%%%%%%%%%%%%%%%%%%%%%%%%%%%%%%%%%%%%

We now derive a convergent Markov chain from the last fragment process
$F_*$, which demonstrates that $F_*$ restricted to its jump times
behaves as expected near $\zeta$. We prove the Markov property of the
chain in Section~\ref{secMarkovchain} and show that it converges
exponentially fast to its stationary distribution in Section~\ref
{secergodicity}. In Section~\ref{secStatPro}, we consider an eternal
stationary version of the Markov chain. We also introduce a biased
version of this eternal chain, which is an essential building-block for
the process $C_{\infty}$.

%s3.1 #&#
\subsection{A Markov chain}\label{secMarkovchain}

Let $T_{1}<T_{2}< \cdots<T_{n}< \cdots$ be the increasing
sequence of times at which $F_{\ast}$ splits, that is, $T_{1}=\inf
{\{t\geq0\dvtx F_{\ast}(t)<1\}}$ and, for $n \geq2$,
\[
T_{n}=\inf{\bigl\{t\geq T_{n-1}\dvtx F_{\ast}(t)<F_{\ast}(T_{n-1})
\bigr\}}.
\]
For convenience, set $T_0 = 0$. We note that only $T_0$ and $T_1$ are
stopping times with respect to the natural filtration of the
fragmentation process. From Remark~\ref{remfinitejumps} and since
$\zeta=\inf\{t \geq0\dvtx F_*(t)=0\}$, we clearly have that
\[
T_n \to\zeta\qquad\mbox{a.s. as }n\rightarrow\infty.
\]
Define, for $n\geq0$,
%
%e3.1 #&#
\begin{equation}
\label{Markovprocess} Z_n:=\bigl(F_{\ast}(T_{n})
\bigr)^{\alpha}(\zeta-T_{n}),
\end{equation}
and note that $Z_n^{1/\alpha}$ is the value of the process $\varepsilon
^{1/\alpha}F_*(\zeta-\varepsilon)$ at $\varepsilon= \zeta-T_n$.
Intuitively, $Z_n$ is a version of the extinction time updated
according to what we know about the last fragment at time $T_n$.

Note also that $Z_0=\zeta$, and
set $\Theta_0 = 1$, ${\bolds\Delta}_0 = (0,0, \ldots)$. For $n \ge1$,
let $\Theta_n = F_*(T_n)/ F_*(T_{n-1})$, and let ${\bolds\Delta}_n =
(\Delta_{n,1},\Delta_{n,2}, \ldots)$ be the relative sizes of the
other sub-blocks resulting from the split of $F_*$ which occurs at time
$T_n$, ordered so that $\Delta_{n,1} \ge\Delta_{n,2} \ge\cdots\ge
0$. Then
\[
\bigl(F_*(T_{n-1}) \Delta_{n,1}, F_*(T_{n-1})
\Delta_{n,2}, \ldots\bigr)
\]
are the sizes of the blocks which \emph{split off} from the last
fragment at time $T_n$. As a consequence of the fact that $\nu$ is
conservative, we have $\Theta_n + \sum_{i=1}^{\infty} \Delta_{n,i}
= 1$ almost surely. See Figure~\ref{figspinedecomp} for an illustration.

%pr3.1 #&#
\begin{prop} \label{propMarkovprop}
\textup{(a)} The process $(Z_n, \Theta_n, {\bolds\Delta}_n)_{n\geq
0}$ is a time-homogeneous Markov chain. Moreover, conditional on
$\sigma(Z_m, \Theta_m, {\bolds\Delta}_m, m \le n)$, the law of
$(Z_{n+1}, \Theta_{n+1}, {\bolds\Delta}_{n+1})$ depends only on the
value of $Z_n$.

\textup{(b)} The transition densities $P(x, \mathrm d y)$, $x>0$, of
$(Z_n)_{n \ge0}$ are given by
%
%e3.2 #&#
\begin{eqnarray}\label{ProbaTransition}
&& P(x,\d y)
\nonumber\\[-10pt]\\[-12pt]\nonumber
&&\qquad =\frac{e^{-x}}{f_{\zeta}(x)} {f_{\zeta
}(y)} \biggl( \int
_{\mathcal S_1} \sum_{i\dvtx s_i>0}e^{s_{i}^{-\alpha}y}
\prod_{j\neq i}\mathbb{F}_{\zeta}
\bigl(s_{j}^{\alpha}s_{i}^{-\alpha}y\bigr)
\mathbh{1}_{ \{ 0<y<s_{i}^{\alpha}x \} } \nu(\mathrm d \mathbf s) \biggr
) \,\mathrm d y,\hspace*{-20pt}
\end{eqnarray}
where $\mathbb F_{\zeta}$ is the cumulative distribution function of
$\zeta$.
\end{prop}

We refer to $(Z_n)_{n \ge0}$ as the \emph{driving chain} of $(Z_n,
\Theta_n, {\bolds\Delta}_n)_{n\geq0}$.

%re3.2 #&#
\begin{rem}\label{remdensitypositive}
The density
in (\ref{ProbaTransition}) is strictly positive for all $x,y>0$. This
is a consequence of the positivity of $f_{\zeta}$ on $(0, \infty)$
(Lemma~\ref{LemmaDensity}) and of the fact that
$\prod_{j\neq i}\mathbb{F}_{\zeta}(s_{j}^{\alpha}s_{i}^{-\alpha}y)
> 0$ when
$s_{i}^{-\alpha}y>0$ (as explained in the proof of
Lemma~\ref{LemmaDensity}).
\end{rem}

Let $Y_0:=\zeta^{1/\alpha}$, and for $n \geq1$, let
\[
Y_n:= \biggl(\frac{\zeta-T_n}{\zeta-T_{n-1}} \biggr)^{1/\alpha
}=
\frac{Z^{1/\alpha}_n}{Z^{1/\alpha}_{n-1} \Theta_n}.
\]
%

%f1 #&#
\begin{figure}

\includegraphics{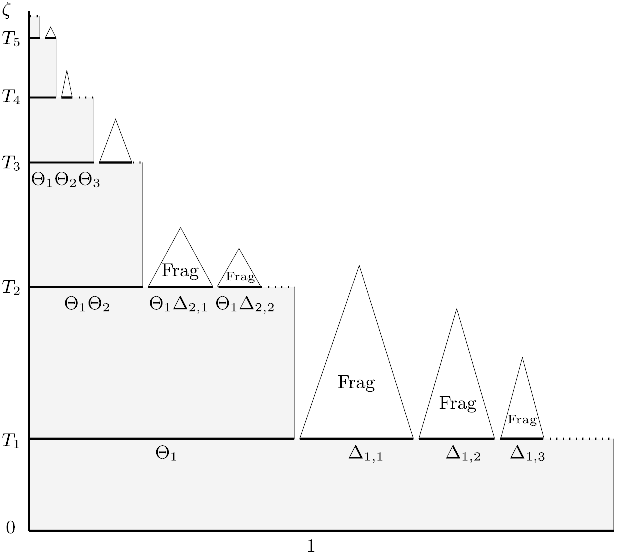}

\caption{The spine decomposition.  Time runs up the page.  The size of the last fragment, $F_{*}$, which
is constant on the intervals $[T_i,T_{i+1}-)$ is shaded. The blocks which split off from $F_{*}$ start their
own fragmentation processes, each conditioned to become extinct before $\zeta$.}\label{figspinedecomp}
\end{figure}

\noindent 
Later on it will turn out to be convenient to work with $Y_n$,
essentially because the times to extinction $\zeta-T_n$ can then be
expressed in the multiplicative form $\zeta\prod_{i=1}^nY_i^{\alpha
}$. To this end, we need the following simple corollary of
Proposition~\ref{propMarkovprop}.

%co3.3 #&#
\begin{cor}
\label{corotaun}
The process $ (Z_n, Y_n, {\bolds\Delta}_n )_{n \geq0}$ is a
time-homogeneous Markov chain with driving chain $(Z_n,n \geq0)$.
\end{cor}

The rest of this section is devoted to the proof of Proposition~\ref
{propMarkovprop}. Recall from Proposition~\ref{strongMarkov} that for
$t \geq0$, $F(T_1+t)$ is the decreasing rearrangement of the terms of
the sequences
\[
F_{1}(T_{1})G^{(1)}\bigl(tF_{1}(T_{1})^{\alpha
}
\bigr),F_{2}(T_{1})G^{(2)}\bigl(tF_{2}(T_{1})^{\alpha}
\bigr),\ldots,
\]
where the processes $G^{(i)}$ are independent fragmentations, all
having the same
distribution as $F$. They are also independent of $T_{1}$ and
$F(T_{1})$. Now let $\zeta^{(i)}=\inf\{t\geq
0\dvtx G^{(i)}(t)=\mathbf0\}$, so that
%
%e3.3 #&#
\begin{equation}
\label{eqzeta} \zeta=T_{1}+\sup_{i\geq1}\bigl
\{F_{i}(T_{1})^{-\alpha}\zeta^{(i)}\bigr\}.
\end{equation}
By Lemma~\ref{LemmaLastFragment}, this supremum is a maximum. Let
$I:=\operatorname{argmax}_{i\geq
1}\{F_{i}(T_{1})^{-\alpha}\zeta^{(i)}\}$, and note that $F_{\ast
}(T_{1})=F_{I}(T_{1})$ and $Z_1=\zeta^{(I)}$. Let
%
%e3.4 #&#
\begin{equation}
\label{Hs} H^{(i, j)} =G^{ (j+ \mathbh{1}_{ \{ j \ge i \}
} )}= \cases{ G^{(j)}, &
\quad if $j < i$,
\cr
G^{(j+1)}, &\quad if $j \ge i$.}
\end{equation}
Finally, for $x > 0$ and suitable test functions $\phi$ and $\psi$,
we write
\[
A(\phi,x) = \mathbb{E} \bigl[\phi(F) \mid\zeta= x \bigr] \quad\mbox
{and}\quad
B(\psi,x) = \mathbb{E} \bigl[\psi(F) \mid\zeta< x \bigr].
\]

%re3.4 #&#
\begin{rem}
The function $A(\phi, \cdot)$ is well defined only up
to a Borel set of Lebesgue measure 0, and is Borel-measurable.
However, when applied to a positive and absolutely continuous random
variable, say $X$, this is enough to define the random variable $A(\phi
,X)$ properly up to a set of probability 0. This remark is also valid
for any forthcoming functions defined as expectations conditional on
$\zeta=x$.
\end{rem}

The following lemma is the key result
needed to prove the Markov property of $(Z_n, \Theta_n, \bolds
{\Delta}_n)_{n\geq0}$.

%le3.5 #&#
\begin{lem} \label{TheLemma}
For all suitable test functions $\phi$ and $\psi_j, j \ge1$,
\begin{eqnarray*}
&& \mathbb E \Biggl[\phi\bigl(G^{(I)} \bigr) \prod
_{j=1}^{\infty} \psi_j
\bigl(H^{(I,j)} \bigr) \bigg| \zeta, \zeta^{(I)},
F_{I}(T_1), \bigl(F_{k}(T_1), k
\ne I\bigr) \Biggr]
\\
&&\qquad = A \bigl(\phi,\zeta^{(I)} \bigr) \prod_{j=1}^{\infty}
B \bigl(\psi_j, F_I^{-\alpha}(T_1)F^{\alpha}
_{j + \mathbh{1}_{ \{ j \ge I
\} }}(T_1) \zeta^{(I)} \bigr).
\end{eqnarray*}
\end{lem}
In particular, conditional on $\zeta^{(I)}$, $G^{(I)}$ is independent
of $\zeta$, $F(T_1)$ and $F_I(T_1)$, and is distributed as a
fragmentation process conditioned to die at time $\zeta^{(I)}$.

\begin{pf*}{Proof of Lemma \ref{TheLemma}}
We will, in fact, prove that
\begin{eqnarray*}
&& \mathbb E \Biggl[\phi\bigl(G^{(I)} \bigr) \prod
_{j=1}^{\infty} \psi_j
\bigl(H^{(I,j)} \bigr) \bigg|\zeta, \zeta^{(I)},
F(T_1), I \Biggr]
\\
&&\qquad = A \bigl(\phi,\zeta^{(I)} \bigr) \prod
_{j=1}^{\infty} B \bigl(\psi_j,
F_I^{-\alpha}(T_1)F^{\alpha}
_{j + \mathbh{1}_{ \{ j \ge I
\} }}(T_1) \zeta^{(I)} \bigr),
\end{eqnarray*}
which\vspace*{2pt} implies the statement of the lemma. Let $\chi$ be another test
function. For $i\neq j$, set $S_{i,j}=\{F_i^{-\alpha}(T_1)\zeta^{(i)}
\geq F_{j}^{-\alpha}(T_1)^{-\alpha} \zeta^{(j)} \}$ and note that $\{
I=i\}=\bigcap_{j\geq1}S_{i,j+ \mathbh{1}_{ \{ j \ge i
\} }}$. We have
\begin{eqnarray*}
&& \mathbb E \Biggl[\phi\bigl(G^{(I)} \bigr) \prod
_{j=1}^{\infty} \psi_j
\bigl(H^{(I,j)} \bigr) \chi\bigl(\zeta,\zeta^{(I)},
F(T_1) \bigr) \mathbh{1}_{ \{ I=i \} } \Biggr]
\\
&&\qquad  = \mathbb E \Biggl[\phi\bigl(G^{(i)} \bigr) \prod
_{j=1}^{\infty} \psi_j
\bigl(H^{(i,j)} \bigr) \chi\bigl(T_1+F_i^{-\alpha}(T_1)
\zeta^{(i)},\zeta^{(i)}, F(T_1) \bigr)
\mathbh{1}_{ \{ I=i \}
} \Biggr]
\\
&&\qquad  = \mathbb{E} \Biggl[ \chi\bigl(T_1+F_i^{-\alpha}(T_1)
\zeta^{(i)},\zeta^{(i)}, F(T_1) \bigr)
\\
&&\hspace*{10pt}\quad\qquad{} \times\mathbb E \Biggl[ \phi\bigl(G^{(i)} \bigr) \prod
_{j =1}^{\infty} \psi_j
\bigl(G^{(j + \mathbh{1}_{ \{ j \ge i \} })} \bigr) \mathbh{1}_{S_{i,j+
\mathbh{1}_{ \{ j \ge i \} }}} \bigg| T_1,F(T_1),
\zeta^{(i)} \Biggr] \Biggr].
\end{eqnarray*}
Since $G^{(j)}, j \ge1$ are independent fragmentations, independent of
$T_{1}$ and $F(T_{1})$,
we see that
\begin{eqnarray*}
& &\mathbb E \Biggl[\phi\bigl(G^{(i)}\bigr) \prod
_{j =1}^{\infty} \psi_j
\bigl(G^{(j + \mathbh{1}_{ \{ j \ge i \} })} \bigr) \mathbh
{1}_{S_{i,j+ \mathbh{1}_{ \{ j \ge i \} }}} \bigg|
T_1,F(T_1),\zeta^{(i)} \Biggr]
\\
&&\qquad = \mathbb{E} \bigl[\phi\bigl(G^{(i)} \bigr) \mid
\zeta^{(i)} \bigr] \prod_{j=1}^{\infty}
\mathbb{E} \bigl[\psi_j \bigl(G^{(j +
\mathbh{1}_{ \{ j \ge i \} })} \bigr)
\mathbh{1}_{S_{i,j+
\mathbh{1}_{ \{ j \ge i \} }}} \mid F(T_1),\zeta^{(i)} \bigr]
\\
&&\qquad = A \bigl(\phi, \zeta^{(i)} \bigr) \prod
_{j=1}^{\infty} B \bigl(\psi_j,
F_i^{-\alpha}(T_1)F^{\alpha}
_{j + \mathbh{1}_{ \{ j \ge i
\} }}(T_1)\zeta^{(i)} \bigr)
\\
&&\hspace*{60pt}\quad\qquad {}\times\mathbb{P} \bigl(\zeta^{(j + \mathbh{1}_{
\{ j \ge i \} })} <
F_i^{-\alpha}(T_1)F^{\alpha}
_{j+
\mathbh{1}_{ \{ j \ge i \} }}(T_1) \zeta^{(i)} \mid F(T_1),
\zeta^{(i)} \bigr)
\\
&&\qquad = A \bigl(\phi, \zeta^{(i)} \bigr) \Biggl[ \prod
_{j=1}^{\infty} B \bigl(\psi_j,
F_i^{-\alpha}(T_1)F^{\alpha}
_{j + \mathbh{1}_{
\{ j \ge i \} }}(T_1) \zeta^{(i)} \bigr) \Biggr] \mathbb{P}
\bigl(I=i \mid F(T_1), \zeta^{(i)} \bigr).
\end{eqnarray*}
Then
\begin{eqnarray*}
&& \mathbb E \Biggl[\phi\bigl(G^{(I)} \bigr) \prod
_{j=1}^{\infty} \psi_j
\bigl(H^{(I,j)} \bigr) \chi\bigl(\zeta,\zeta^{(I)},
F(T_1) \bigr) \mathbh{1}_{ \{ I=i \} } \Biggr]
\\
&&\qquad = \mathbb E \Biggl[A \bigl(\phi, \zeta^{(I)} \bigr) \prod
_{j=1}^{\infty} B \bigl(\psi_j,
F_I^{-\alpha}(T_1)F^{\alpha}
_{j +
\mathbh{1}_{ \{ j \ge I \} }}(T_1) \zeta^{(I)} \bigr)
\\
&&\hspace*{145pt}{}\times  \chi\bigl(\zeta,
\zeta^{(I)}, F(T_1) \bigr) \mathbh{1}_{ \{ I=i
\} }
\Biggr],
\end{eqnarray*}
and the result follows.
\end{pf*}

\begin{pf*}{Proof of Proposition~\ref{propMarkovprop}}
(a) We start by proving that $(Z,\Theta,{\bolds\Delta})$ is a
time-homogeneous Markov chain with driving chain $Z$. To see this, we
will show that for all suitable test functions $f,g_i$ and all $n \geq1$,
{\renewcommand{\theequation}{$\mathfrak{R}_n$}
%e3.5 #&#
\begin{eqnarray}\label{err}
%%%%\tag{$\mathfrak{R}_n$}
&& \mathbb E
\Biggl[f(Z_n,\Theta_n,{\bolds\Delta}_n) \prod
_{i=0}^{n-1}g_i(Z_i,
\Theta_i,{\bolds\Delta}_i) \Biggr]
\nonumber\\[-8pt]\\[-8pt]\nonumber
&&\qquad =\mathbb E \Biggl[
\mathsf{F}_f(Z_{n-1}) \prod_{i=0}^{n-1}g_i(Z_i,
\Theta_i,{\bolds\Delta}_i) \Biggr],
\end{eqnarray}\setcounter{equation}{4}}%
where $\mathsf{F}_f(x)=\mathbb E[f(Z_1,\Theta_1,{\bolds\Delta}_1) \mid
Z_0=x]$. Note that $\mathsf{F}_f(x)$ is well defined for Lebesgue
a.e. $x>0$, since $Z_0=\zeta$ is absolutely continuous. We will prove
by induction on $n$ that (\ref{err}) is valid and that $Z_n$ is
absolutely continuous, so that $\mathsf{F}_f(Z_{n-1})$ is almost
surely well defined. In fact, once (\ref{err}) is proved, the
absolute continuity of $Z_n$ is a direct consequence of the absolute
continuity of $Z_{n-1}$ and of (\ref{err}), taking test
functions $f$ of the form $f=\mathbh{1}_A$ for Borel sets $A$ with
Lebesgue measure 0. So it is enough to focus in the following on the
proof of (\ref{err}) for $n\geq1$.

$(\mathfrak{R}_1)$ is an immediate consequence of the fact that
$\Theta_0$ and ${\bolds\Delta}_0$ are deterministic. Now assume that
(\ref{err}) holds, and
recall that the last fragment process $F_*$ can be written as
%
%e3.5 #&#
\begin{equation}
\label{passageI} F_*(T_1+t)=F_I(T_1)G^{(I)}
\bigl(t F^{\alpha}_I(T_1) \bigr), \qquad t \geq0.
\end{equation}
As for the standard fragmentation process, the last fragment process of
$G^{(I)}$ is well-defined since $G^{(I)}$ is a randomized version of
the fragmentation. We denote it by $(G_*^{(I)}(t),t \geq0)$.
Then for $k \geq1$, let $T_{k}^{(I)}$ be the $k$th time at which
$G_*^{(I)}$ splits, let
\[
\Theta_k^{(I)}:=G_*^{(I)}\bigl(T_{k}^{(I)}
\bigr)/G_*^{(I)}\bigl(T_{k-1}^{(I)}\bigr)
\]
and let ${\bolds\Delta}^{(I)}_k$ be the relative sizes of the other
sub-blocks resulting from the split of $G_*^{(I)}$ at time
$T_{k}^{(I)}$. From (\ref{passageI}), we get that
$T_{k+1}=T_1+F^{-\alpha}_I(T_1)T_{k}^{(I)}$, $\Theta_{k+1}=\Theta
_k^{(I)}$, ${\bolds\Delta}_{k+1}={\bolds\Delta}^{(I)}_k$
and
\[
Z_{k+1} = \bigl( G^{(I)}\bigl( T_{k}^{(I)}
\bigr) \bigr) ^{\alpha} \bigl( Z_1-T_{k}^{(I)}
\bigr):=Z^{(I)}_k.
\]
Therefore,
\begin{eqnarray*}
&& \mathbb E \Biggl[f(Z_{n+1},\Theta_{n+1},{\bolds
\Delta}_{n+1}) \prod_{i=0}^{n}g_i(Z_i,
\Theta_i,{\bolds\Delta}_i) \Biggr]
\\
&&\qquad =\mathbb E \Biggl[f\bigl(Z^{(I)}_n,\Theta^{(I)}_n,{
\bolds\Delta}^{(I)}_n\bigr)g_0(Z_0,
\Theta_0,{\bolds\Delta}_0)g_1(Z_1,
\Theta_1,{\bolds\Delta}_1)
\\
&&\hspace*{141pt}{}\times\prod
_{i=1}^{n-1}g_{i+1} \bigl(Z^{(I)}_i,
\Theta^{(I)}_i,{\bolds\Delta}^{(I)}_i
\bigr) \Biggr]
\\
&&\qquad = \mathbb E \Biggl[g_0(Z_0,\Theta_0,{\bolds
\Delta}_0)g(Z_1,\Theta_1,{\bolds
\Delta}_1)
\\
&&\hspace*{44pt}{}\times  \mathbb E \Biggl[f\bigl(Z^{(I)}_n,
\Theta^{(I)}_n,{\bolds\Delta}^{(I)}_n
\bigr)
\\
&&\hspace*{69pt}{}\times \prod_{i=1}^{n-1}g_{i+1}
\bigl(Z^{(I)}_i,\Theta^{(I)}_i,{\bolds
\Delta}^{(I)}_i\bigr) \bigg| Z_0,
Z_1,F(T_1), F_I(T_1) \Biggr]
\Biggr].
\end{eqnarray*}
Similarly,
\begin{eqnarray*}
&& \mathbb E \Biggl[\mathsf{F}_f(Z_{n}) \prod
_{i=0}^{n}g_i(Z_i,\Theta
_i,{\bolds\Delta}_i) \Biggr]
\\
&&\qquad = \mathbb E \Biggl[g_0(Z_0,\Theta_0,{\bolds
\Delta}_0)g(Z_1,\Theta_1,{\bolds
\Delta}_1)
\\
&&\hspace*{43pt}{}\times  \mathbb E \Biggl[\mathsf{F}_f
\bigl(Z^{(I)}_{n-1}\bigr) \prod_{i=1}^{n-1}g_{i+1}
\bigl(Z^{(I)}_i,\Theta^{(I)}_i,{\bolds
\Delta}^{(I)}_i \bigr) \bigg| Z_0,
Z_1,F(T_1),F_I(T_1) \Biggr]
\Biggr].
\end{eqnarray*}
Then by Lemma~\ref{TheLemma} (recall that $Z_0=\zeta$, $Z_1=\zeta
^{(I)}$) applied to the functions $\psi_j \equiv1, \forall j \in
\mathbb N$ and $\phi(G^{(I)})=f(Z^{(I)}_n,\Theta^{(I)}_n,{\bolds\Delta
}^{(I)}_n ) \prod_{i=1}^{n-1}g_{i+1}(Z^{(I)}_i,\Theta^{(I)}_i,{\bolds
\Delta}^{(I)}_i)$,
\begin{eqnarray*}
&& \mathbb E \Biggl[f\bigl(Z^{(I)}_n,\Theta^{(I)}_n,{
\bolds\Delta}^{(I)}_n \bigr) \prod_{i=1}^{n-1}g_{i+1}
\bigl(Z^{(I)}_i,\Theta^{(I)}_i,{\bolds
\Delta}^{(I)}_i\bigr) \bigg| Z_0,
Z_1,F(T_1),F_I(T_1)
\Biggr]
\\
&&\qquad =u(Z_1),
\end{eqnarray*}
where
\[
u(x)=\mathbb E \Biggl[f(Z_n,\Theta_n,{\bolds
\Delta}_n ) \prod_{i=1}^{n-1}g_{i+1}(Z_i,
\Theta_i,{\bolds\Delta}_i) \bigg| \zeta=x \Biggr],
\]
and similarly
\[
\mathbb E \Biggl[\mathsf{F}_f\bigl(Z^{(I)}_{n-1}
\bigr) \prod_{i=1}^{n-1}g_{i+1}
\bigl(Z^{(I)}_i,\Theta^{(I)}_i,{\bolds
\Delta}^{(I)}_i\bigr) \bigg|  Z_0,
Z_1,F(T_1),F_I(T_1)
\Biggr]=v(Z_1),
\]
where
\[
v(x)=\mathbb E \Biggl[\mathsf{F}_f(Z_{n-1}) \prod
_{i=1}^{n-1}g_{i+1}(Z_i,
\Theta_i,{\bolds\Delta}_i) \bigg| \zeta=x \Biggr].
\]
To get $(\mathfrak{R}_{n+1})$, it remains to prove that $u(x)=v(x)$
for Lebesgue-a.e. $x>0$. For this we use the induction hypothesis
(\ref{err}) which implies that the random variables
\[
f(Z_n,\Theta_n,{\bolds\Delta}_n ) \prod
_{i=1}^{n-1}g_{i+1}(Z_i,
\Theta_i,{\bolds\Delta}_i)\quad\mbox{and}\quad
\mathsf{F}_f(Z_{n-1}) \prod_{i=1}^{n-1}g_{i+1}(Z_i,
\Theta_i,{\bolds\Delta}_i)
\]
have the same expectation conditional on $\zeta$ since
\begin{eqnarray*}
&& \mathbb E \Biggl[ h(\zeta)f(Z_n,\Theta_n,{\bolds
\Delta}_n ) \prod_{i=1}^{n-1}g_{i+1}(Z_i,
\Theta_i,{\bolds\Delta}_i) \Biggr]
\\
&&\qquad = \mathbb E \Biggl[ h(
\zeta)\mathsf{F}_f(Z_{n-1}) \prod
_{i=1}^{n-1}g_{i+1}(Z_i,
\Theta_i,{\bolds\Delta}_i) \Biggr] %
\end{eqnarray*}
for all bounded measurable functions $h$. The result follows by induction.

(b) It remains to prove that the transition densities of the chain
$(Z_n)_{n\geq0}$ are given by identity (\ref{ProbaTransition}). To
get this, we compute the joint density of $(Z_0,Z_1)$. The first step
is to use the independence
of $T_1,F(T_1)$ and $(\zeta^{(j)},j \geq1)$ [defined in~(\ref{eqzeta})] and the fact that $F(T_1)$ is distributed according to $\nu
$, to get that, for any test
function~$\chi$,
\begin{eqnarray*}\label{law3uplet}
&& \mathbb{E} \bigl[\chi\bigl(F_I(T_1),
\zeta,T_1\bigr) \bigr]
\\
&&\qquad =\sum_{i=1}^{\infty}
\mathbb{E} \bigl[\chi\bigl(F_i(T_1),T_1+F_i(T_1)^{-\alpha}
\zeta^{(i)},T_1\bigr)\mathbh{1}_{ \{ I=i \} } \bigr]
\\
&&\qquad  =\int_{\mathcal S_1}\sum_{i\dvtx s_i>0}
\int_0^{\infty} \mathbb{E} \bigl[\chi
\bigl(s_i,t+s_i^{-\alpha}\zeta^{(i)},t\bigr)
\mathbh{1}_{ \{ s_i^{-\alpha} \zeta^{(i)} \geq\max_{j \neq i}
s_j^{-\alpha} \zeta^{(j)} \} } \bigr] e^{-t} \,\mathrm d t\, \nu(\mathrm d
\mathbf s)
\\
&&\qquad  = \int_{\mathcal S_1} \sum_{i\dvtx s_i>0}
\int_0^{\infty}\!\!\int_0^{\infty}
\chi\bigl(s_i,t+s_i^{-\alpha}z,t
\bigr)f_{\zeta}(z) \prod_{j \neq i} \mathbb
F_{\zeta}\bigl(s_j^{\alpha}s_i^{-\alpha}z
\bigr) e^{-t}\,\mathrm d t\, \mathrm d z\,\nu(\mathrm d \mathbf s).
\end{eqnarray*}
In the inner integral, let $x=t+s_i^{-\alpha}z$ [then $z=s_i^{\alpha}
(x-t)$] to get that this last is equal to
\[
\int_{\mathcal S_1}\sum_{i\dvtx s_i>0}\int
_0^{\infty}\!\!\int_0^{x}s_i^{\alpha}
\chi(s_i,x,t)f_{\zeta}\bigl(s_i^{\alpha}(x-t)
\bigr) \prod_{j
\neq i} \mathbb F_{\zeta}
\bigl(s_j^{\alpha}(x-t)\bigr)e^{-t}\, \mathrm d t \,
\mathrm d x\, \nu(\mathrm d \mathbf s).
\]
Taking $\chi(F_I(T_1),\zeta,T_1)=\phi(\zeta,F_I(T_1)^{\alpha
}(\zeta-T_1))$, we obtain
\begin{eqnarray*}
&&  \mathbb{E} \bigl[\phi(Z_0,Z_1) \bigr]
\\
&&\qquad = \mathbb{E}
\bigl[\phi\bigl(\zeta,F_I(T_1)^{\alpha}(
\zeta-T_1)\bigr) \bigr]
\\
&&\qquad  = \int_{\mathcal S_1}\sum_{i\dvtx s_i>0}\int
_0^{\infty}\!\!\int_0^x
s_i^{\alpha}\phi\bigl(x,s_i^{\alpha}(x-t)
\bigr) f_{\zeta}\bigl(s_i^{\alpha}(x-t)\bigr)
\\
&&\hspace*{106pt}{}\times \prod
_{j \neq i}\mathbb F_{\zeta}\bigl(s_j^{\alpha}(x-t)
\bigr) e^{-t}\, \mathrm d t \,\mathrm d x\, \nu(\mathrm d \mathbf s)
\\
&&\qquad  = \int_{\mathcal S_1} \sum_{i\dvtx s_i>0}\int
_0^{\infty}\!\!\int_0^{s_i^{\alpha}x}
e^{s_i^{-\alpha}y-x}\phi(x,y)f_{\zeta}(y) \prod_{j \neq i}
\mathbb F_{\zeta}\bigl(s_j^{\alpha}s_i^{-\alpha}y
\bigr)\, \mathrm d y \,\mathrm d x\, \nu(\mathrm d \mathbf s),
\end{eqnarray*}
where we have used the change of variable
$y=s_i^{\alpha}(x-t)$ in the inner integral, so that $t =
x-s_i^{-\alpha} y$. It follows that
the joint density of $(Z_0,Z_1)$ is given by
%
%e3.6 #&#
\begin{eqnarray}
f_{Z_0,Z_1}(x,y)=e^{-x}f_{\zeta}(y) \int
_{\mathcal S_1} \sum_{i\dvtx s_i>0}
e^{s_i^{-\alpha}y}\mathbh{1}_{ \{ y <
s_i^{\alpha}x \} } \prod_{j \neq i}
\mathbb F_{\zeta}\bigl(s_j^{\alpha}s_i^{-\alpha}y
\bigr) \nu(\mathrm d \mathbf s),\nonumber
\\[-2pt]
\eqntext{x,y >0.}
\end{eqnarray}
In particular, the density of $Z_1$ conditioned on $Z_0=x$ is given by
$f_{Z_0,Z_1}(x,y)/\break f_{\zeta}(x)$, as desired.
\end{pf*}

%s3.2 #&#
\subsection{Geometric ergodicity of the driving chain} \label{secergodicity}

In view of the role of $(Z_n)_{n \ge0}$ as driving chain, it will
suffice to study its ergodic properties in order to deduce those of
$(Z_n, \Theta_n, {\bolds\Delta}_n)_{n \ge0}$. This section is devoted
to the proof of the following result.

%th3.6 #&#
\begin{teo} \label{Theorem1}
Suppose that $\int_{{\mathcal S}_1}s_1^{-1}\nu(\mathrm d
\mathbf{s})<\infty$. Then the Markov chain $(Z_n)_{n\geq0}$ is
positive Harris recurrent and possesses a unique stationary
distribution on $(0,\infty)$, $\pi_{\mathrm{stat}}$. This stationary
distribution is absolutely continuous (with respect to
Lebesgue measure) and its density, which (with a slight abuse of
notation) we also denote by $\pi_{\mathrm{stat}}$, is the unique
solution to the equation
%
%e3.7 #&#
\begin{equation}
\label{Eqstat} \qquad\pi(x)=f_{\zeta}(x)\int_{{\mathcal
S}_1} \Biggl(
\sum_{i=1}^{\infty}e^{s_{i}^{-\alpha}x}\prod
_{j\neq
i}\mathbb{F}_{\zeta}\bigl(s_{j}^{\alpha}s_{i}^{-\alpha}x
\bigr) \biggl(\int_{s_{i}^{-\alpha
}x}^{\infty}\frac{e^{-y}\pi(y)}{f_{\zeta}(y)}\,
\mathrm d y \biggr) \Biggr)\nu(\mathrm d \mathbf s).
\end{equation}
Moreover, the distribution $\mathcal{L}(Z_n)$ of $Z_n$ converges to
$\pi_{\mathrm{stat}}$ exponentially fast;
more precisely, there exists a constant $r>1$ such that
%
%e3.8 #&#
\begin{equation}
\label{expcv} \sum_{n \geq1}r^n \bigl
\llVert\mathcal{L}(Z_n)-\pi_{\mathrm{stat}} \bigr\rrVert
_{\mathrm{TV}}<\infty,
\end{equation}
where $\llVert \cdot\rrVert _{\mathrm{TV}}$ denotes the total variation norm.
\end{teo}

We have not been able to extract an explicit expression for $\pi
_{\mathrm{stat}}$ from (\ref{Eqstat}). (However, Lemmas~\ref
{lemqualitatif} and~\ref{logZfiniteexpect} in the \hyperref[appendix]{Appendix} give some
qualitative information about it.) Note also that (\ref{Eqstat})
implies that $\pi_{\mathrm{stat}}(x)>0$ for $x>0$.

To prove Theorem~\ref{Theorem1}, we use the geometric
ergodic theorem of Meyn and Tweedie \cite{MeynTweedie}, Theorem 15.0.1, which is based on a Foster--Lyapounov drift
criterion; see (\ref{eqdrift}) below. To understand the meaning of
this criterion, we first need to introduce the concept of a small set.
With this in hand, all we will require in order to obtain Theorem~\ref
{Theorem1} from the geometric ergodic theorem are the forthcoming
Lemmas~\ref{lemIrr} and~\ref{lemFosterLyap}.
In the following, for each integer $n$, $P^n$ denotes the $n$-step
transition probability kernel of the chain $(Z_n)_{n \ge0}$.\vspace*{1pt}

Following page 109 of Meyn and Tweedie \cite{MeynTweedie}, a \textit
{small set} $C$
is a Borel subset of
$\R^{*}_+$, for which there exist an integer $m_C>0$ and a nontrivial
measure $\mu_C$ such that
%
%e3.9 #&#
\begin{equation}
\label{SmallSets} \qquad P^{m_C}(x,B)\geq\mu_C(B)\qquad\mbox{for
all Borel sets }B\subseteq(0,\infty)\mbox{ and all }x \in C.
\end{equation}
In our case, subsets of a compact subset of $(0,\infty)$ are clearly
small sets. Indeed, let $C\subseteq[a,b]$, $0<a<b$, and recall from Lemma
\ref{LemmaDensity} that $f_{\zeta}(x) \leq1$ for all $x >0$. It is
then easy to see that for all Borel sets $B \subseteq(0,\infty)$
and all $x\in C$,
\[
P(x,B)\geq e^{-b}\mu_C(B),
\]
where the measure $\mu_{C}$ is defined for all $B$ by
%
%e3.10 #&#
\begin{equation}
\label{eqnboundingmeasure} \qquad\mu_{C}(B)=\int_B{f_{\zeta}(y)}
\biggl( \int_{\mathcal
S_1}\sum_{i\dvtx s_i>0}e^{s_{i}^{-\alpha}y}
\prod_{j\neq i}\mathbb{F}_{\zeta}
\bigl(s_{j}^{\alpha}s_{i}^{-\alpha}y\bigr)
\mathbf1_{\{
0<y<s_{i}^{\alpha
}a\}} \nu(\mathrm d \mathbf s) \biggr) \,\mathrm d y.
\end{equation}
The Markov chain $(Z_n,n\geq0)$ is
\emph{Lebesgue-irreducible} if, for all Borel sets $B \subseteq
(0,\infty)$
with strictly positive Lebesgue measure and all $x>0$, there exists
an integer $n$ with $P^n(x,B)>0$. It is said to be \emph{strong
aperiodic} if there exists a small set $C$ with $m_C=1$ and $\mu_C(C)>0$.

%le3.7 #&#
\begin{lem}
\label{lemIrr}
$(Z_n, n \ge0)$ is both Lebesgue-irreducible and strong aperiodic.
\end{lem}

(In fact, the geometric ergodic theorem is valid if we replace strong
aperiodicity by aperiodicity, but the definition of strong aperiodicity
is easier to write down and easy to check in our context.)

\begin{pf*}{Proof of Lemma \ref{lemIrr}}
By (\ref{ProbaTransition}) and Remark~\ref{remdensitypositive} we have $P(x,B)>0$ for all $x>0$ and all Borel
sets $B$ with strictly positive Lebesgue measure;
Lebesgue-irreducibility follows.
Strong aperiodicity follows directly from the above proof that subsets
of compact subsets of $(0,\infty)$ are small.
\end{pf*}

%le3.8 #&#
\begin{lem}[(Foster--Lyapounov drift criterion)]
\label{lemFosterLyap}
Assume that\break $\int_{\mathcal S_1} s_1^{-1} \nu(\mathrm d \mathbf
s)<\infty$.
Then there exists a small
set $C$, a function $V\dvtx (0,\infty) \rightarrow[1,\infty)$ and
constants $b<\infty$ and $\beta>0$ satisfying
%
%e3.11 #&#
\begin{equation}
\label{eqdrift} \mathbb P V(x)-V(x) \leq-\beta V(x)+b \mathbh{1}_C(x)
\qquad\forall x >0,
\end{equation}
where $\mathbb P V(x):=\int_{0}^{\infty} V(y) P(x,\mathrm
d y)$.
Moreover, $\int_{0}^{\infty} V(x)f_{\zeta}(x)\,\mathrm d x < \infty$.
\end{lem}

Note that in Theorem 15.0.1 of \cite{MeynTweedie}, the words \textit
{small sets} are replaced by \textit{petite sets}. However, small
implies petite, and so we lose nothing here by using the former notion.

\begin{pf*}{Proof of Lemma \ref{lemFosterLyap}}
Let
\[
\label{DefV} V(x):=\frac{\exp(-cx)}{f_{\zeta}(x)}, \qquad x>0,
\]
where $c \in(0,1/2)$ is such that $\exp(cx)f_{\zeta}(x) \rightarrow
0$ as $x
\rightarrow\infty$; such a $c$ exists by Lemma
\ref{LemmaDensity}. Hence, $V(x) \to\infty$ as $x
\rightarrow\infty$ and, still by Lemma
\ref{LemmaDensity}, it
is continuous and $V(x) \to\infty$ as $x \rightarrow0$. In
particular, it possesses a strictly positive minimum on
$(0,\infty)$, which, up to normalization, may be supposed to be 1.

For the remainder of the proof, we proceed in three steps. The goal of
the first two steps is to check that $\mathbb PV(x)<\infty$ for all
$x>0$ and that
\[
\frac{\mathbb PV(x)}{V(x)}=f_{\zeta}(x)\exp(cx)\mathbb P V(x)\rightarrow
0\qquad\mbox{as } x \rightarrow0\mbox{ or } x \rightarrow\infty.
\]
To this end, write $\mathbb P V(x)=\mathbb P_{1}
V(x)+\mathbb P_{2} V(x)$ where
\begin{eqnarray*}
\hspace*{-2pt}&& \mathbb P_{1} V(x)
\\
\hspace*{-2pt}&&\qquad := \frac{e^{-x}}{f_{\zeta}(x)}
\\
\hspace*{-2pt}&&\quad\qquad{}\times \int_{0}^{\infty}
V(y){f_{\zeta}(y)} \biggl( \int_{\mathcal S_1}\sum
_{i\dvtx s_i>c_1}e^{s_{i}^{-\alpha}y}
\prod_{j\neq i}
\mathbb{F}_{\zeta
}\bigl(s_{j}^{\alpha}s_{i}^{-\alpha}y
\bigr)\mathbh{1}_{ \{
0<y<s_{i}^{\alpha}x \} } \nu(\mathrm d \mathbf s) \biggr)\, \mathrm d y,
\end{eqnarray*}
with $c_1 \in(0,c^{-1/\alpha})$.

\begin{longlist}[\textit{Step} 2.]
\item[\textit{Step} 1.] We prove that the quantity $f_{\zeta}(x)\exp
(cx)\mathbb P_{1} V(x)$ is finite for
all $x>0$ and converges to 0 as $x$ tends to 0 or $\infty$.
To see this, note first that $s_i \leq i^{-1}$, $\forall i \geq1$, for
$\nu$-a.e. sequence $\mathbf s$, and, therefore, that the
sum involved in $\mathbb P_{1} V(x)$ only concerns indices
$i<c_1^{-1}$. Since this set of indices is finite, it is sufficient to
check that for all $i<c_1^{-1}$,
\[
e^{(c-1)x}\int_{\mathcal S_1} \mathbh{1}_{ \{ s_i >c_1
\} }
\biggl(\int_0^{s_{i}^{\alpha}x} V(y){f_{\zeta}(y)}
e^{s_{i}^{-\alpha
}y}\prod_{j\neq i}\mathbb{F}_{\zeta}
\bigl(s_{j}^{\alpha}s_{i}^{-\alpha}y\bigr)\,
\mathrm d y \biggr) \nu(\mathrm d \mathbf s)
\]
is finite and converges to 0 as $x$ tends to 0 or to $\infty$. This
term is bounded above by
%
%e3.12 #&#
\begin{equation}
\label{majorant} e^{(c-1)x}\int_{\mathcal S_1}
\mathbh{1}_{ \{ s_i >c_1
\} } \biggl(\int_0^{s_{i}^{\alpha}x}
e^{-cy + s_{i}^{-\alpha}y} \,\mathrm d y \biggr) \nu(\mathrm d \mathbf s)
\end{equation}
which is clearly finite and converges to 0 as $x \rightarrow0$. To get
a similar result when $x \rightarrow\infty$, recall that $c<1/2$, and
note that
\begin{eqnarray*}
&& e^{(c-1)x} \int_0^{s_{i}^{\alpha}x} e^{-cy + s_{i}^{-\alpha}y}\,
\mathrm d y
\\
&&\qquad \leq\cases{ e^{(c-1)x}s_i^{\alpha}x, &\quad
if $c_1^{-\alpha}<s_i^{-\alpha} \leq c$,
\vspace*{3pt}\cr
e^{(1-s_i^{\alpha})cx}s_i^{\alpha}x, &\quad if $c<s_i^{-\alpha}
\leq\frac{1}{2}$,
\vspace*{3pt}\cr
\bigl(e^{(1-s_i^{\alpha})cx}-e^{(c-1)x} \bigr)
\bigl(s_i^{-\alpha
}-c\bigr)^{-1}, &\quad if
$s_i^{-\alpha} >\frac{1}{2}$.} %
\end{eqnarray*}
In all three cases, the upper bound converges to 0 (since $s_i<1$) $\nu
$-a.e. as $x \rightarrow\infty$ and is bounded above by a finite
constant independent both of $x \geq1$ and of $s_i$ in the interval
under consideration. Hence, by dominated convergence, term (\ref
{majorant}) tends to 0 as $x \rightarrow\infty$.

\item[\textit{Step} 2.] We now prove a similar result to the one proved in
step~1, but for $\mathbb P_{2} V$. Here we use the hypothesis
$\int_{\mathcal S_1}s_1^{-1} \nu(\mathrm d
\mathbf s)<\infty$.
It will be sufficient to show that
%
%e3.13 #&#
\begin{equation}
\label{Eqsmalls} \int_{\mathcal S_1}\sum_{i\dvtx s_i \leq c_1}
\biggl( \int_0^{\infty} e^{(s_{i}^{-\alpha}-c)y} \prod
_{j\neq i}\mathbb{F}_{\zeta}\bigl(s_{j}^{\alpha}s_{i}^{-\alpha}y
\bigr)\, \mathrm d y \biggr) \nu(\mathrm d \mathbf s) <\infty,
\end{equation}
using the fact that $\exp{(c-1)x} \rightarrow
0$ as $x \rightarrow\infty$ and monotone convergence near $0$. To get
(\ref{Eqsmalls}), we use the existence of some finite constant $m$
[see Lemma~\ref{LemmaDensity}, and note that $\int_{\mathcal
S_1}s_1^{-\alpha-1} \nu(\mathrm d \mathbf s)\leq
\int_{\mathcal S_1}s_1^{-1} \nu(\mathrm d \mathbf
s)<\infty$] such that $\mathbb{F}_{\zeta}(s_{1}^{\alpha
}s_{i}^{-\alpha}y) \leq m s_{1}^{-1
}s_{i}y^{-1/\alpha}$, for all $y>0$. Hence, the double integral in
(\ref{Eqsmalls}) is
bounded above by
\begin{eqnarray*}
&& \int_{\mathcal S_1} \mathbh{1}_{ \{ s_1 \leq c_1 \} } \biggl(\int
_0^{\infty} e^{(c_1^{-\alpha}-c)y} \prod
_{j \geq2}\mathbb{F}_{\zeta}\bigl(s_{j}^{\alpha}s_{1}^{-\alpha}y
\bigr)\, \mathrm d y \biggr) \nu(\mathrm d \mathbf s)
\\
&&\quad{} + m \int_{\mathcal S_1} \sum_{i \geq2\dvtx s_i \leq c_1}
s_{1}^{-1}s_{i} \biggl(\int
_0^{\infty} e^{(c_1^{-\alpha}-c)y}y^{-1/\alpha}\, \mathrm
d y \biggr)\nu(\mathrm d \mathbf s)
\\
&&\qquad \leq\bigl(c-c_1^{-\alpha}\bigr)^{-1} \int
_{\mathcal S_1} \mathbh{1}_{
\{ s_1 \leq c_1 \} } \nu(\mathrm d \mathbf s) +
m'\int_{\mathcal S_1} s_{1}^{-1 }
\nu(\mathrm d \mathbf s) < \infty.
\end{eqnarray*}

\item[\textit{Step} 3.] From expression (\ref{ProbaTransition}) for the
transition density and from the fact that $f_{\zeta}$ is continuous,
we see that the function $x \mapsto\mathbb P V(x)$ is
continuous on $(0,\infty)$. Let $0<\beta<1$, and introduce the set
$C:=\{x>0\dvtx  \mathbb P
V(x)-(1-\beta)V(x) \geq0\}$. The continuity of $\mathbb P V/V$ on
$(0,\infty)$,
together with steps 1~and~2, imply that $C$ is a compact subset of
$(0,\infty)$, and so it is a small set.
Moreover $b:=\sup_{x \in C}(\mathbb P
V(x)-(1-\beta)V(x))<\infty$, since $\mathbb PV-(1-\beta)V$ is
continuous on $(0,\infty)$.
Finally, for all $x>0$,
\[
\mathbb PV(x) \leq(1-\beta) V(x)+ b \mathbh{1}_C(x), %
\]
which is the required drift criterion.

Finally, note that $\int_{0}^{\infty} V(x) f_{\zeta}(x) \,\mathrm
dx<\infty$ since $V(x) f_{\zeta}(x)=\exp(-cx), x>0$ for some $c>0$.\quad\qed
\end{longlist}\noqed
\end{pf*}

Theorem~\ref{Theorem1} now follows from the geometric ergodic theorem.

%s3.3 #&#
\subsection{The stationary and biased Markov chains}\label{secStatPro}

In order to construct the limit process $C_{\infty}$ appearing in
Theorem~\ref{teomainabstract}, we need to introduce an eternal
stationary version of $(Z_n,Y_n, {\bolds\Delta}_n)_{n \geq1}$ and then
a biased version of this stationary version; see the forthcoming
Definition~\ref{defnCinfty} of $C_{\infty}$. This biased version
will appear in the limit when using the techniques of Markov renewal
theory to pass from the convergence of $(Z_n)$ to the asymptotic
behavior of the continuous-time processes $F_{\ast}$ and $F$ near
their extinction time.

First, we can construct a stationary version of $(Z_n,Y_n, {\bolds\Delta
}_n)_{n \geq1}$ from a fragmentation process conditioned to have an
extinction time distributed according to $\pi_{\mathrm{stat}}$.
Formally, the Markov chain $ ((Z^{\mathrm{stat}}_n, Y^{\mathrm
{stat}}_n, {\bolds\Delta}^{\mathrm{stat}}_n)_{n \geq1}, Z^{\mathrm
{stat}}_0 )$ is defined by
\begin{eqnarray*}
%&& \mathbb E \bigl[f\bigl( \bigl(Z^{\mathrm{stat}}_n,
%Y^{\mathrm{stat}}_n, {\bolds\Delta}^{\mathrm{stat}}_n
%\bigr)_{n \geq1}\bigr),Z^{\mathrm
%{stat}}_0  \bigr]
&&\mathbb{E}\bigl[f \bigl( \bigl(Z_n^{\mathrm{stat}}, Y_n^{\mathrm{stat}}, \Delta_n^{\mathrm{stat}} \bigr)_{n \ge
1}, Z_0^{\mathrm{stat}} \bigr)\bigr]
\\
&&\qquad =\int
_{0}^{\infty}\mathbb E \bigl[f\bigl((Z_n,
Y_n, {\bolds\Delta}_n)_{n \geq1}, Z_0
\bigr) \mid\zeta=x \bigr]\pi_{\mathrm
{stat}}(\mathrm dx) %
\end{eqnarray*}
for suitable test functions $f$. Since $Z_0=\zeta$, the chain is then
stationary: $Z^{\mathrm{stat}}_n$ is distributed according to $\pi
_{\mathrm{stat}}$ for all $n \geq0$ and
\[
\bigl(Z^{\mathrm{stat}}_n, Y^{\mathrm{stat}}_n, {\bolds
\Delta}^{\mathrm
{stat}}_n\bigr) \stackrel{\mathrm{law}}=
\bigl(Z^{\mathrm{stat}}_1, Y^{\mathrm
{stat}}_1, {\bolds
\Delta}^{\mathrm{stat}}_1\bigr) \qquad\mbox{for } n \geq1, %
\]
since $(Z^{\mathrm{stat}}_n,n \geq0)$ is the driving chain of the
Markov chain $(Z^{\mathrm{stat}}_n, Y^{\mathrm{stat}}_n,\break {\bolds\Delta
}^{\mathrm{stat}}_n)_{n \geq1}$.

Now let
%
%e3.14 #&#
\begin{equation}
\label{defstat} \bigl(Z^{\mathrm{stat}}_n,Y^{\mathrm{stat}}_n,{
\bolds\Delta}^{\mathrm
{stat}}_n\bigr)_{n \in\mathbb Z}
\end{equation}
be an \textit{eternal} stationary version of $(Z_n, Y_n,{\bolds\Delta
}_n)_{n \geq1}$. Recall that such process always exists: for all
positive integers $k$, the distribution of the chain $(Z^{\mathrm
{stat}}_n,Y^{\mathrm{stat}}_n,{\bolds\Delta}^{\mathrm{stat}}_n)_{n
\geq-k}$ is defined to be that of $(Z^{\mathrm{stat}}_n,Y^{\mathrm
{stat}}_n,{\bolds\Delta}^{\mathrm{stat}}_n)_{n \geq1}$ and so,\vspace*{1pt} by
Kolmogorov's consistency theorem, the full process $(Z^{\mathrm
{stat}}_n,Y^{\mathrm{stat}}_n,{\bolds\Delta}^{\mathrm{stat}}_n)_{n \in
\mathbb Z}$ is well defined.

Observe that
\[
\int_{0}^{\infty} \mathbb{P} (Y_1 \leq1
\mid Z_0 = x ) f_{\zeta}(x)\, \mathrm dx = \mathbb{P}
(Y_1 \leq1 )=0
\]
and that, by Lemma~\ref{LemmaDensity}, $f_{\zeta}(x)>0$ for all $x$.
It follows that $\mathbb{P} (Y_1 > 1 \mid Z_0 = x ) = 1$ for
Lebesgue-a.e. $x$, and so we also have $\mathbb{P}
(Y_1^{\mathrm{stat}} > 1 ) = 1$. The following lemma is a
consequence of Lemma~\ref{mu} in the \hyperref[appendix]{Appendix}.

%le3.9 #&#
\begin{lem} \label{lemmufinite}
Suppose that $\int_{\Sfl}s_1^{-1} \nu(\mathrm d \mathbf s)<\infty$. Let
\[
\mu= \mathbb{E} \bigl[\log\bigl(Y_1^{\mathrm{stat}}\bigr) \bigr].
\]
Then $\mu\in(0, \infty)$.
\end{lem}

The biased version $(Z^{\mathrm{bias}}, {Y}^{\mathrm{bias}}, {\bolds
\Delta}^{\mathrm{bias}})$ of the eternal stationary Markov chain
constructed just above is then defined by
\[
\mathbb{E} \bigl[g \bigl(\bigl({Z}^{\mathrm{bias}}_n,
{Y}^{\mathrm
{bias}}_n, {\bolds\Delta}^{\mathrm{bias}}_n
\bigr)_{n \in\Z} \bigr) \bigr] = \frac{1}{\mu} \mathbb E \bigl[ \log
\bigl({Y}^{\mathrm{stat}}_1\bigr) g \bigl(\bigl({Z}^{\mathrm{stat}}_n,
{Y}^{\mathrm{stat}}_n, {\bolds\Delta}^{\mathrm{stat}}_n
\bigr)_{n \in\Z} \bigr) \bigr].
\]
Note that the eternal process $(Z^{\mathrm{bias}}_n, {Y}^{\mathrm
{bias}}_n, {\bolds\Delta}^{\mathrm{bias}}_n)_{n \in\Z}$ is a
time-inhomogeneous Markov chain. However, if we restrict to times $n
\ge1$, it is time-homogeneous, with the same transition kernel as the
stationary and standard versions (although a different initial
distribution). As in the standard case, we set
\[
\Theta_n^{\mathrm{bias}}:=\frac{(Z^{\mathrm{bias}}_n)^{1/\alpha
}}{(Z^{\mathrm{bias}}_{n-1})^{1/\alpha} Y^{\mathrm{bias}}_n}\qquad\mbox
{for } n \in
\mathbb Z.
\]

In Appendix~\ref{secappendix2} we will prove various
technical results about the stationary and biased Markov chains, which
will be used in the main body of the paper.

%%%%%%%%%%%%%%%%%%%%%%%%%%%%%%%%%%%%%%%
%s4 #&#
\section{Asymptotics of the last fragment}\label{MarkovRW}
%%%%%%%%%%%%%%%%%%%%%%%%%%%%%%%%%%%%%%%

We will now determine the asymptotics of $\varepsilon^{1/\alpha}
F_{*}(\zeta- \varepsilon)$ as $\varepsilon\rightarrow0$, and then of the
whole process $t \in\mathbb R_+\mapsto\varepsilon^{1/\alpha}
F_{*}(\zeta- \varepsilon t)$. The key point in our approach is the
ergodicity of the driving chain proved in the previous section.

From the biased Markov chain introduced in Section~\ref{secStatPro},
we can now define what will\vspace*{1pt} be the \emph{limit process}, which is
denoted by $(C_{\infty,*}(t),t \geq0)$.
Let $U$ be uniformly distributed on $[0,1]$, independently of
$({Z}^{\mathrm{bias}}, {Y}^{\mathrm{bias}}, {{\bolds\Delta}}^{\mathrm
{bias}})$.
Let
\[
R(k) = \cases{
\displaystyle\bigl(Y_1^{\mathrm{bias}}\bigr)^{-\alpha U} \prod_{i=1}^k \bigl(Y_i^{\mathrm{bias}}
\bigr)^{\alpha}, &\quad if $k \ge1$,
\vspace*{3pt}\cr
\displaystyle\bigl(Y_1^{\mathrm{bias}}
\bigr)^{-\alpha U}, &\quad if $k=0$,
\vspace*{3pt}\cr
\displaystyle\bigl(Y_1^{\mathrm{bias}}
\bigr)^{-\alpha U} \prod_{i=k+1}^0
\bigl(Y_i^{\mathrm
{bias}}\bigr)^{-\alpha}, &\quad if $k \le-1$,}
\]
so that $R(k)$ is a decreasing function of $k \in\Z$. Note the
multiplicative relation $R(k+1)=R(k)(Y_{k+1}^{\mathrm{bias}})^{\alpha
}, \forall k \in\mathbb Z$.
The following result follows from Lemma~\ref{lemLLN} in the \hyperref[appendix]{Appendix}.

%le4.1 #&#
\begin{lem} \label{lemras}
We have $R(k)\rightarrow0$ as $k \rightarrow\infty$ and
$R(k)\rightarrow\infty$ as $k \rightarrow-\infty$ almost surely.
\end{lem}

The process $C_{\infty,*}$ is then a nondecreasing piecewise constant
right-continuous process, which is defined by
$C_{\infty,*}(0)=0$ and, for $t>0$,
\[
C_{\infty,*}(t) = \bigl(Z^{\mathrm{bias}}_{k}
\bigr)^{1/\alpha} \bigl(R(k)\bigr)^{-1/\alpha} \qquad\mbox{if } t \in
\bigl[R(k+1),R(k) \bigr). %
\]
See Figure~\ref{figCinfinistar} for an illustration. The monotonicity
of $C_{\infty,*}$ comes from the identity
\[
\bigl(Z^{\mathrm{bias}}_{k}\bigr)^{1/\alpha}\prod
_{i=1}^{k} \bigl(Y_i^{\mathrm
{bias}}
\bigr)^{-1}=\bigl(Z^{\mathrm{bias}}_{0}\bigr)^{1/\alpha}
\prod_{i=1}^{k} \Theta_i^{\mathrm{bias}},
\qquad k \geq1 %
\]
and from the fact that the random variables $\Theta_i^{\mathrm
{bias}}$ lie in $(0,1)$ a.s. A similar equality holds\vspace*{1pt} for negative $k$.
Note that $R(1)<1<R(0)$ a.s. and so $C_{\infty,*}(1)=(Y^{\mathrm
{bias}}_1)^U (Z^{\mathrm{bias}}_0)^{1/\alpha}$.

%f2 #&#
\begin{figure}[b]

\includegraphics{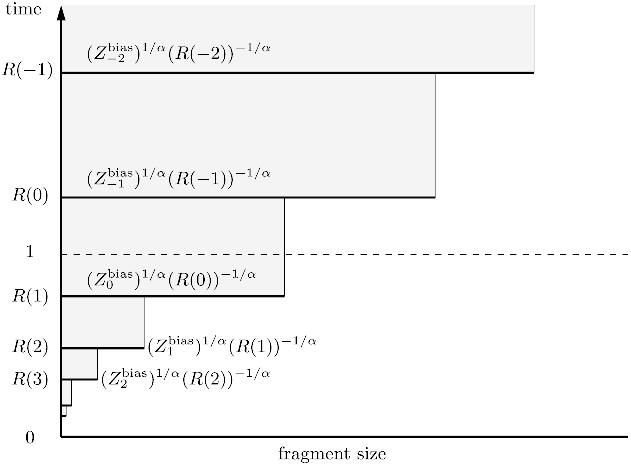}

\caption{The limit $(C_{\infty,*}(t), t \ge 0)$ of the last fragment.  The process is piecewise constant
between the jumps which are indicated. Compare to Figure~\protect\ref{figspinedecomp}: here time has been
reversed.}\label{figCinfinistar}
\end{figure}

%th4.2 #&#
\begin{teo} \label{teolastfrag}
Suppose that $\int_{\Sfl}s_1^{-1} \nu(\mathrm d \mathbf s)<\infty$
and that $\nu$ is nongeometric. Then, as $\varepsilon\to0$,
\[
\bigl( \bigl(\varepsilon^{1/\alpha} F_*\bigl((\zeta-\varepsilon t)-\bigr), t
\geq
0 \bigr), \zeta\bigr) \stackrel{\mathit{law}} {\rightarrow} \bigl(
\bigl(C_{\infty,*}(t), t \geq0 \bigr),\zeta\bigr),
\]
where $\zeta$ and $C_{\infty,*}$ are independent in the limit.
In particular,
\[
\varepsilon^{1/\alpha} F_*(\zeta-\varepsilon) \stackrel{\mathit{law}} {
\rightarrow}\bigl(Y^{\mathrm{bias}}_1\bigr)^U
\bigl(Z^{\mathrm
{bias}}_0\bigr)^{1/\alpha}.
\]
\end{teo}

The proof of this result is based on the convergence in distribution of
the driving chain $(Z_n)_{n\geq0}$, proved in the previous section,
and uses results from Markov renewal theory, which are gathered in
Section~\ref{backgroundMRT} below. In Section~\ref{Onedim}, we prove
the convergence of the one-dimensional marginal distributions of the
rescaled last fragment process. The full functional convergence is then
proved in Section~\ref{Wholepro}.

%s4.1 #&#
\subsection{Background on Markov renewal theory}\label{backgroundMRT}

Let $S_0=0$, and for $n \geq1$,
\[
S_n:= \sum_{i=1}^n
\log{Y_i}.
\]
As a consequence of Corollary~\ref{corotaun}, $(Z_n, S_n)_{n \ge0}$
is a \textit{Markov renewal process}
in the terminology of \cite
{Alsmeyer2,Alsmeyer,AthreyaRen,JacodRen,KestenRen,OreyRen,ShurenkovRen}.
We refer to Alsmeyer's paper \cite{Alsmeyer} for background on this
topic and results about asymptotic behaviors.
As in standard renewal theory, these results depend on hypotheses of
nonarithmeticity/arithmeticity for the support of the process. In our
context, this is formulated as follows: the process is called \emph
{$d$-arithmetic} if $d \geq0$ is the largest number for which there
exists a measurable function $\gamma\dvtx (0,\infty) \rightarrow[0,d)$
such that
%
%e4.1 #&#
\begin{equation}
\label{defarthm} \mathbb P \bigl(\log Y_1 \in\gamma(Z_0)-
\gamma(Z_1)+d \mathbb Z \bigr)=1.
\end{equation}
The process is \emph{nonarithmetic} if no such $d$ exists. The
condition for nonarithmeticity in our setting is unsurprising.

%le4.3 #&#
\begin{lem}
\label{lemarithmetic}
The process $(Z_n,S_n)_{n\geq0}$ is nonarithmetic if and only if the
dislocation measure $\nu$ is nongeometric.
\end{lem}

\begin{pf}
Recall that $Y_1=((\zeta-T_1)/\zeta)^{1/\alpha}$ and $\zeta
=T_1+\Theta_1^{-\alpha} Z_1$, with $T_1$ independent of $(\Theta
_1,Z_1)$, and $Z_0=\zeta$. If $\nu$ is $r$-geometric for some $r \in
(0,1)$, then $\Theta_1 \in r^{\mathbb N}$ a.s. and, consequently,
$\log Y_1 \in\alpha^{-1} ( \log Z_1-\log Z_0 )+(-\log
r)\mathbb N$ a.s. The arithmeticity of $(Z_n,S_n)_{n\geq0}$ follows.

Conversely, assume that (\ref{defarthm}) holds for some $d \geq0$
and some measurable function $\gamma$. This is equivalent to
\[
\mathbb P \bigl(\log\Theta_1 \in\overline\gamma\bigl(T_1+
\Theta_1^{-\alpha}Z_1\bigr)-\overline
\gamma(Z_1)+d \mathbb Z \bigr)=1 %
\]
for some suitable function $\overline\gamma$.
Since $\Theta_1^{-\alpha}Z_1$ has a strictly positive density on
$(0,\infty)$ [see the discussion around (\ref{EqZetamax})], and since
$T_1$ is independent of $(\Theta_1,Z_1)$, this implies that for
Lebesgue a.e. $a>0$,
there exists a real number $b_a$ such that
$
\mathbb P (\overline\gamma(T_1+a)\in b_a + d \mathbb Z )=1$.
But $T_1$ is exponentially distributed, and so $\overline\gamma(u+a)
\in b_a + d \mathbb Z$ for Lebesgue-a.e. $u>0$. This implies that
\[
\mathbb P \bigl(\overline\gamma(Z_0)-\overline\gamma(Z_1)
\in d\mathbb Z \mid Z_0 >a, Z_1>a \bigr)=1\qquad\mbox{for
Lebesgue a.e. }a >0. %
\]
Hence, $\mathbb P (\overline\gamma(Z_0)-\overline\gamma(Z_1)
\in d\mathbb Z )=1$, and so $\mathbb P(\log\Theta_1 \in
d\mathbb Z)=1$. Note that this implies that $d>0$. To conclude, assume
that $\nu$ is\vspace*{2pt} nongeometric; that is, that for all $r \in(0,1)$,
there exists some $i_r \in\mathbb N$ such that $\nu(s_{i_r} \notin
r^{\mathbb N}, s_{i_r}>0)>0$. Then
\[
\mathbb P\bigl(\log\Theta_1 \notin(\log r) \mathbb N\bigr) \geq
\mathbb P\bigl(\Theta_1=F_i(T_1),
F_i(T_1)\notin r ^{\mathbb N}\bigr)\qquad\forall i \in
\mathbb N. %
\]
Since $\mathbb P(\Theta_1=F_i(T_1) \mid F_i(T_1))>0$ when $F_i(T_1)>0$
[this is due to the fact that $\prod_{j\neq i} \mathbb F_{\zeta
}(s_j^{\alpha}x) > 0$, for $\mathbf s \in\mathcal S_1$, when $x>0$,
as explained in the proof of Lemma~\ref{LemmaDensity}] and, since
$\mathbb P(F_{i_r}(T_1)\notin r ^{\mathbb N}\cup\{0\})>0$ by
assumption, we have that $\mathbb P(\log\Theta_1 \notin(\log r)
\mathbb N)>0$ for all $r \in(0,1)$, which contradicts the fact that
$\mathbb P(\log\Theta_1 \in d\mathbb Z)=1$ for some $d>0$. Hence,
$\nu$ is geometric when (\ref{defarthm}) holds.
\end{pf}

Theorem 1 of Alsmeyer~\cite{Alsmeyer} applied to $(Z_n,S_n)_{n \ge0}$
yields the following result, with $\mu=\mathbb E [\log(Y_1^{\mathrm
{stat}})] \in(0,\infty)$; see Lemma~\ref{lemmufinite}.

%th4.4 #&#
\begin{teo}\label{teoAlsmeyer}
Suppose that the dislocation measure $\nu$ is nongeometric and such
that $\int_{\mathcal S_1}s_1^{-1} \nu(\mathrm d \mathbf
s)<\infty$. Suppose that $g\dvtx  \R_+ \times\R_+ \to\R$ is a
measurable function which is such that \textup{(a)}~$g(x,\cdot)$ is
Lebesgue-almost\break everywhere continuous for Lebesgue-almost all $x \in\R
_+$ and \textup{(b)}\break $\int_0^{\infty} \sum_{n \in\Z_+} \sup_{n
\rho\le y < (n+1)\rho} \llvert g(x,y)\rrvert\* \pi_{\mathrm
{stat}}(\mathrm d x) <
\infty$ for some $\rho> 0$. Then as $r \to\infty$,
\[
\mathbb{E} \biggl[\sum_{n \ge0} g(Z_n, r -
S_n) \Big| Z_0 = z \biggr] \to\frac{1}{\mu} \int
_{\R_+} \int_{\R_+} g(x,y)\, \mathrm d y\,
\pi_{\mathrm{stat}}(\mathrm d x),
\]
for Lebesgue-almost all $z \in\R_+$.
\end{teo}

In terms of the biased process introduced in Section~\ref{secStatPro},
Corollary 1 of \cite{Alsmeyer} reads as follows.

%co4.5 #&#
\begin{cor}\label{CoroAlsmeyer} Suppose that the dislocation measure
$\nu$ is nongeometric and such that $\int_{\mathcal
S_1}s_1^{-1} \nu(\mathrm d \mathbf s)<\infty$. Let $h\dvtx \R_+\times\R
_+ \rightarrow\R$ be a measurable function such that $g\dvtx  \R_+\times
\R_+ \to\R$ defined by
$g(x,y)=h(x,y)\mathbb{P} (\log(Y_1)>y \mid Z_0=x )$
satisfies conditions \textup{(a)} and \textup{(b)} of Theorem~\ref{teoAlsmeyer}. Let
\[
J(r)=\sup\{n\geq0\dvtx  S_n \leq r \}, %
\]
and assume that $J(r)<\infty$ for all $r \in\R_+$.
Then for Lebesgue-almost all $z \in\R_+$, as $r \rightarrow\infty$,
\[
\mathbb{E} \bigl[h (Z_{J(r)}, r-S_{J(r)} ) \mid
Z_0=z \bigr] \rightarrow\mathbb{E} \bigl[h \bigl(Z^{\mathrm{bias}}_0,
U \log\bigl(Y^{\mathrm{bias}}_1\bigr) \bigr) \bigr], %
\]
where $U$ is uniformly distributed on $[0,1]$ and independent of
$(Z^{\mathrm{bias}}_0, \log(Y^{\mathrm{bias}}_1))$.
\end{cor}

%re4.6 #&#
\begin{rem} \label{remals}
We have replaced all the ``for $\pi_{\mathrm{stat}}$-almost all $x$''
in\break Alsmeyer's results by ``for Lebesgue-almost all $x$'' since $\pi
_{\mathrm{stat}}$ is equivalent to Lebesgue measure on $\mathbb R_+$.
Note also that a bounded measurable function $h\dvtx \mathbb R_+\times
\mathbb R_+ \rightarrow\mathbb R$ which is such that $h(x,\cdot)$ is
Lebesgue-almost everywhere continuous for Lebesgue-almost all $x \in\R
_+$, satisfies the conditions of Corollary~\ref{CoroAlsmeyer}. Indeed,
the measurability and condition (a) are obvious. For condition (b),
take $\rho=1$, set $\llVert h\rrVert _{\infty}=\sup_{x \geq0}\llvert
h(x)\rrvert $ and note that
\begin{eqnarray*}
&&\int_{\mathbb R_+} \sum_{n \in\Z_+} \sup
_{n \le y < n+1} \bigl\llvert h(x,y)\bigr\rrvert\mathbb P \bigl(
\log(Y_1)>y \mid Z_0=x \bigr) \pi_{\mathrm
{stat}}(\mathrm d
x)
\\
&&\qquad \leq\llVert h\rrVert_{\infty}\!\!\int_{\mathbb R_+} \sum
_{n \in\Z_+} \mathbb P \bigl(\log(Y_1)>n \mid
Z_0=x \bigr) \pi_{\mathrm{stat}}(\mathrm d x)
\\
&&\qquad \leq\llVert h\rrVert_{\infty} (1+\mu)<\infty,
\end{eqnarray*}
since $\mathbb E[Z]+1 \geq\sum_{n \in\mathbb Z_+} \mathbb P(Z > n)$
for any positive random variable $Z$.
\end{rem}

%s4.2 #&#
\subsection{One-dimensional convergence}\label{Onedim}

We use Corollary~\ref{CoroAlsmeyer} to obtain the convergence in
distribution of the rescaled last fragment at time $\zeta- \varepsilon$
as $\varepsilon\rightarrow0$,
%
%e4.2 #&#
\begin{equation}
\label{cvfd} \bigl(\varepsilon^{1/\alpha} F_*(\zeta-\varepsilon),\zeta\bigr)
\stackrel{\mathrm{law}} {\rightarrow} \bigl(\bigl(Z^{\mathrm
{bias}}_0
\bigr)^{1/\alpha}\bigl(Y^{\mathrm{bias}}_1\bigr)^U,
\zeta\bigr),
\end{equation}
with $\zeta$ independent of $(Z^{\mathrm{bias}}_0)^{1/\alpha
}(Y^{\mathrm{bias}}_1)^U$ in the limit.
In fact, this result will be an immediate consequence of the proof of
Theorem~\ref{teolastfrag} in the next section. However, its proof is
instructive and so, by way of a brief warm-up, we give the details here.

Let
%
%e4.3 #&#
\begin{eqnarray}
\label{defNeps} N_{\varepsilon} & =& \sup\{n \ge0\dvtx  \zeta- \varepsilon\ge
T_n\} = \sup\Biggl\{n \geq0\dvtx  \prod_{i=0}^n
Y_i^{\alpha} \ge\varepsilon\Biggr\}
\nonumber\\[-8pt]\\[-8pt]\nonumber
& =& \sup\Biggl\{n \geq0\dvtx  \sum_{i=0}^n
\log{Y_i} \leq\frac
{1}{\alpha}\log\varepsilon\Biggr\},
\end{eqnarray}
with the convention that $\sup\varnothing=-\infty$.
Note that for all $\varepsilon> 0$, since $T_n \to\zeta$ almost surely,
$N_\varepsilon<\infty$ almost surely. Also,
\[
\mathbb{P} (N_\varepsilon\neq-\infty) = \mathbb{P} (\zeta\geq\varepsilon)
\rightarrow1\qquad\mbox{as } \varepsilon\rightarrow0.
\]
Therefore,
\[
F_*(\zeta- \varepsilon) = F_*(T_{N_{\varepsilon}}) \mathbh{1}_{ \{
N_{\varepsilon} \neq-\infty \} } +
\mathbh{1}_{ \{
N_{\varepsilon} = -\infty \} } = \prod_{i=0}^{N_{\varepsilon}}
\Theta_{i} \mathbh{1}_{ \{ N_{\varepsilon} \neq-\infty
\} } + \mathbh{1}_{ \{ N_{\varepsilon} = -\infty \} }.
\]
Hence, since $\prod_{i=0}^{N_{\varepsilon}} \Theta_{i}=Z_{N_{\varepsilon
}}^{1/\alpha}\prod_{i=0}^{N_{\varepsilon}} Y^{-1}_{i}$,
\begin{eqnarray*}
&& \varepsilon^{1/\alpha} F_*(\zeta-\varepsilon)
\\
&&\qquad = Z_{N_{\varepsilon
}}^{1/\alpha}
\exp\biggl( \frac{1}{\alpha} \log\varepsilon- S_{N_{\varepsilon}}-\frac
{1}{\alpha}
\log\zeta\biggr) \mathbh{1}_{
\{ N_{\varepsilon} \neq-\infty \} } + \varepsilon^{1/\alpha
}\mathbh{1}_{ \{ N_{\varepsilon} = -\infty \} }.
\end{eqnarray*}

Next, let $f\dvtx \mathbb R_+ \rightarrow\mathbb R$ be a bounded continuous
test function. To obtain (\ref{cvfd}), it is sufficient to prove that
for Lebesgue-almost all $z>0$,
\[
\mathbb E \bigl[f \bigl( \varepsilon^{1/\alpha} F_*(\zeta-\varepsilon) \bigr)
\mid\zeta=z \bigr] \rightarrow\mathbb E \bigl[ f \bigl(\bigl
(Z_0^{\mathrm{bias}}
\bigr)^{1/\alpha} \bigl({Y}_1^{\mathrm
{bias}}\bigr)^U
\bigr) \bigr]. %
\]
So let $z>0$ and note that, conditional on $\zeta=z$, $N_{\varepsilon
}\neq-\infty$ for all $\varepsilon\leq z$. Hence, for $\varepsilon\leq
z$, since $Z_0=\zeta$,
\begin{eqnarray*}
&& \mathbb E \bigl[f \bigl( \varepsilon^{1/\alpha} F_*(\zeta-\varepsilon)
\bigr)
\mid\zeta=z \bigr]
\\
&&\qquad = \mathbb E \biggl[f \biggl( Z_{N_{\varepsilon}}^{1/\alpha}\exp\biggl(
\frac{1}{\alpha} \log\varepsilon- S_{N_{\varepsilon}}-\frac{1}{\alpha
}\log z \biggr)
\biggr) \Big|  Z_0=z \biggr]
\\
&&\qquad = \mathbb E \biggl[f \biggl( Z_{J(\alpha^{-1} \log(\varepsilon
/z))}^{1/\alpha}\exp\biggl(
\frac{1}{\alpha} \log(\varepsilon/z) - S_{J(\alpha^{-1} \log(\varepsilon/z))}
\biggr) \biggr) \Big|
Z_0=z \biggr],
\end{eqnarray*}
where $J$ is defined in Corollary~\ref{CoroAlsmeyer}.
The last expectation converges to\break $\mathbb E [ f
((Z_0^{\mathrm{bias}})^{1/\alpha} ({Y}_1^{\mathrm{bias}})^U )
]$ as $\varepsilon\rightarrow0$, by Corollary~\ref
{CoroAlsmeyer}, since the function $h$ defined on $(0,\infty) \times
[0,\infty)$ by
\[
h(x,y)=f\bigl(x^{1/\alpha}\exp(y)\bigr) %
\]
and by, say, $h(0,y)=0$ for $y \in\mathbb R_+$, satisfies the
conditions of Corollary~\ref{CoroAlsmeyer}; see Remark~\ref{remals}.

%s4.3 #&#
\subsection{Functional convergence}\label{Wholepro}

We take as a convention (for the standard version of our Markov chain,
started from $Z_0 = \zeta$) that $Z_i=Y_i=0$ and $\bolds{\Delta
}_i=\bolds 0$ for $i<0$.

%le4.7 #&#
\begin{lem}
\label{cvpoints}
Endow $\mathbb(\mathbb R^2_+\times\mathcal S_1)^{{\mathbb Z}} \times
\mathbb R_+$ with the product topology. Then
for Lebesgue a.e. $z>0$, conditional on $\zeta=z$, we have
\begin{eqnarray*}
&& \biggl( (Z_{N_{\varepsilon}+n},Y_{N_{\varepsilon} +n},{\bolds{ \Delta
}}_{N_{\varepsilon} +n}
)_{n \in\mathbb Z},\frac{1}{\alpha}\log(\varepsilon/\zeta)-S_{N_{\varepsilon}}
\biggr)
\\
&&\qquad \stackrel{\mathit{law}} {\rightarrow} \bigl( \bigl(Z^{\mathrm
{bias}}_n,Y^{\mathrm
{bias}}_n,{
\bolds{\Delta}}^{\mathrm{bias}}_n \bigr)_{n \in\mathbb
Z},U\log
\bigl(Y^{\mathrm{bias}}_1\bigr) \bigr)
\end{eqnarray*}
as $\varepsilon\to0$, where $U$ is independent of the process $
(Z^{\mathrm{bias}}_n,Y^{\mathrm{bias}}_n,{\bolds{\Delta}}^{\mathrm
{bias}}_n )_{n \in\mathbb Z}$.
\end{lem}

\begin{pf}
It is sufficient to prove that for all $k \geq1$ and Lebesgue a.e.
$z>0$, conditional on $\zeta=z$,
\begin{eqnarray*}
&& \biggl( (Z_{N_{\varepsilon}+n},Y_{N_{\varepsilon}+n},{\bolds{\Delta
}}_{N_{\varepsilon}+n}
)_{n \geq-k},\frac{1}{\alpha}\log(\varepsilon/\zeta)-S_{N_{\varepsilon}}
\biggr)
\\
&&\qquad \stackrel{\mathrm{law}} {\rightarrow} \bigl( \bigl(Z^{\mathrm
{bias}}_n,Y^{\mathrm
{bias}}_n,{
\bolds{\Delta}}^{\mathrm{bias}}_n \bigr)_{n \geq-k},U\log
\bigl(Y^{\mathrm{bias}}_1\bigr) \bigr).
\end{eqnarray*}
So, in the following, we fix $k \geq1$.

Recall that, conditionally on $\zeta=Z_0=z$, $N_\varepsilon\neq-\infty
$ for all $\varepsilon\leq z$. Moreover, $N_\varepsilon\rightarrow\infty$
as $\varepsilon\rightarrow0$ almost surely. It is therefore sufficient
to show that for Lebesgue a.e. $z>0$ and all bounded continuous
functions $f\dvtx (\mathbb R^2_+\times\mathcal S_1)^{{\mathbb Z}}\times
\mathbb R_+ \rightarrow\mathbb R$,
\begin{eqnarray*}
&&\mathbb E \biggl[f \biggl( (Z_{N_{\varepsilon}+n},Y_{N_{\varepsilon
}+n},{\bolds{
\Delta}}_{N_{\varepsilon}+n} )_{n \geq-k},\frac
{1}{\alpha}\log(
\varepsilon/z)-S_{N_{\varepsilon}} \biggr)\mathbh{1}_{
\{ N_{\varepsilon} \geq k+1 \} } \Big| Z_0=z \biggr]
\\
&&\qquad \rightarrow\mathbb E \bigl[f \bigl( \bigl(Z^{\mathrm
{bias}}_n,Y^{\mathrm{bias}}_n,{
\bolds\Delta}^{\mathrm{bias}}_n \bigr)_{n \geq-k},U\log
\bigl(Y^{\mathrm{bias}}_1\bigr) \bigr) \bigr].
\end{eqnarray*}
To show this, note that for $\varepsilon\leq z$,
\begin{eqnarray*}
&& \mathbb E \biggl[f \biggl( (Z_{N_{\varepsilon}+n},Y_{N_{\varepsilon
}+n},{\bolds
\Delta}_{N_{\varepsilon}+n} )_{n \geq-k},\frac
{1}{\alpha}\log(
\varepsilon/z)-S_{N_{\varepsilon}} \biggr)\mathbh{1}_{
\{ N_{\varepsilon} \geq k+1 \} } \Big|
Z_0=z \biggr]
\\
&&\qquad  = \sum_{i=1}^{\infty} \mathbb E \biggl[f
\biggl( (Z_{i+k+n},Y_{i+k+n},{\bolds\Delta}_{i+k+n}
)_{n \geq-k},\frac
{1}{\alpha}\log(\varepsilon/z)-S_{i+k} \biggr)
\\
&&\hspace*{158pt}{}\times\mathbh{1}_{ \{ S_{i+k}\leq {1}/{\alpha}\log
(\varepsilon/z)< S_{i+k+1} \} } \Big| Z_0=z \biggr]
\\
&&\qquad = \sum_{i=0}^{\infty} \mathbb E \biggl[ g
\biggl(Z_i, \frac
{1}{\alpha}\log(\varepsilon/z)-S_{i}
\biggr) \Big| Z_0=z \biggr],
\end{eqnarray*}
where
\begin{eqnarray*}
g(x,y)
&=& \mathbb E \Biggl[ f \Biggl( (Z_{k+n+1},Y_{k+n+1},{\bolds\Delta
}_{k+n+1} )_{n \geq-k},y-\sum_{j=1}^{k+1}
\log Y_j \Biggr)
\\
&&\hspace*{82pt}{}\times \mathbh{1}_{ \{ \sum_{j=1}^{k+1} \log Y_j \leq y< \sum
_{j=1}^{k+2} \log Y_j \} } \bigg| Z_0 =
x \Biggr],
\end{eqnarray*}
the first equality being a consequence of the definition of $N_\varepsilon
$ and the second of
the Markov property of the process. Note that $g$ satisfies the
assumptions of Theorem~\ref{teoAlsmeyer}; see Remark~\ref{remals}.
Consequently, as $\varepsilon\rightarrow0$, for Lebesgue a.e. $z>0$,
\begin{eqnarray*}
&&\mathbb{E} \biggl[f \biggl( (Z_{N_{\varepsilon}+n},Y_{N_{\varepsilon
}+n},{\bolds
\Delta}_{N_{\varepsilon}+n} )_{n \geq-k},\frac
{1}{\alpha}\log(
\varepsilon/z)-S_{N_{\varepsilon}} \biggr)\mathbh{1}_{
\{ N_{\varepsilon} \geq k+1 \} }
\Big|
Z_0=z \biggr]
\\
&&\qquad \rightarrow\frac{1}{\mu} \int_{\mathbb R_+} \int
_{\mathbb R_+} g(x,y)\, \mathrm dy\, \pi_{\mathrm{stat}}(\mathrm dx).
\end{eqnarray*}
Using the change of variables $u=(y-\sum_{j=1}^{k+1} \log Y_j )/\log
(Y_{k+2})$, we get, for $U$ uniform on $[0,1]$ and independent of the
process $(X,Y,{\bolds\Delta})$, that this limit can be written as
\begin{eqnarray*}
&& \frac{1}{\mu} \int_{\R_+} \mathbb{E} \bigl[
\log(Y_{k+2}) f \bigl( (Z_{k+n+1},Y_{k+n+1},{\bolds
\Delta}_{k+n+1} )_{n \geq
-k},U \log Y_{k+2} \bigr) \mid
Z_0 = x \bigr]
\\
&&\hspace*{26pt}{}\times  \pi_{\mathrm
{stat}}(\mathrm dx)
\\
&&\qquad = \frac{1}{\mu} \mathbb E \bigl[\log\bigl(Y^{\mathrm{stat}}_{k+2}
\bigr) f \bigl( \bigl(Z^{\mathrm{stat}}_{k+n+1},Y^{\mathrm{stat}}_{k+n+1},{
\bolds\Delta}^{\mathrm{stat}}_{k+n+1} \bigr)_{n \geq-k},U \log
Y^{\mathrm
{stat}}_{k+2} \bigr) \bigr]
\\
&&\qquad = \frac{1}{\mu} \mathbb E \bigl[\log\bigl(Y^{\mathrm{stat}}_{1}
\bigr) f \bigl( \bigl(Z^{\mathrm{stat}}_{n},Y^{\mathrm{stat}}_{n},{
\bolds\Delta}^{\mathrm{stat}}_{n} \bigr)_{n \geq-k},U \log
Y^{\mathrm
{stat}}_{1} \bigr) \bigr],
\end{eqnarray*}
by stationarity of the process $(Z^{\mathrm{stat}},Y^{\mathrm
{stat}},{\bolds\Delta}^{\mathrm{stat}})$.
\end{pf}

\begin{pf*}{Proof of Theorem~\ref{teolastfrag}}
Let $\varepsilon\leq\zeta$. Recall that for $0 < t \leq\zeta/\varepsilon$,
\[
N_{\varepsilon t} = \sup\Biggl\{n \ge0\dvtx  \prod_{i=0}^n
Y_i^{\alpha} \ge\varepsilon t \Biggr\} \neq-\infty
\]
and
\[
\varepsilon^{1/\alpha} F_*(\zeta- \varepsilon t) = \varepsilon^{1/\alpha}
Z_{N_{\varepsilon t}}^{1/\alpha} \prod_{i=0}^{N_{\varepsilon t}}
Y_i^{-1}.
\]
We will want to re-center all times around $N_{\varepsilon}$ (which is
$\neq-\infty$ since $\varepsilon\leq\zeta$). To this end, let
\[
R_\varepsilon(k) = \varepsilon^{-1} \prod
_{i=0}^{N_\varepsilon+k} Y_i^{\alpha}, \qquad k
\geq-N_{\varepsilon}
\]
so that $R_\varepsilon(k)$ is strictly decreasing in $k \geq-N_{\varepsilon
}$ and
\[
N_{\varepsilon t} = N_\varepsilon+ \sup\bigl\{k \ge-N_\varepsilon\dvtx
R_\varepsilon(k) \ge t \bigr\}.
\]
Note that $R_\varepsilon(k)=\varepsilon^{-1}(\zeta-T_{N_{\varepsilon}+k})$
and therefore that $(R_\varepsilon(k), k \geq-N_{\varepsilon}+1)$ is the
(decreasing) sequence of jump times of the process $(\varepsilon
^{1/\alpha}F_*(\zeta-\varepsilon t), t \geq0)$.
Re-centering times around $N_{\varepsilon}$, we obtain that $R_{\varepsilon
}(k)$ may be written as
%
%e4.4 #&#
\begin{equation}\label{defRepsk}
\qquad R_{\varepsilon}(k) = \cases{ \displaystyle\exp\bigl(\alpha
S_{N_{\varepsilon}}
- \log(\varepsilon/\zeta)\bigr) \prod_{i=1}^k
Y_{N_{\varepsilon}+i}^{\alpha}, &\quad if $k \ge1$,
\cr
\displaystyle\exp\bigl(\alpha
S_{N_{\varepsilon}} - \log(\varepsilon/\zeta)\bigr), &\quad if $k = 0$,
\cr
\displaystyle\exp\bigl(
\alpha S_{N_{\varepsilon}} - \log(\varepsilon/\zeta)\bigr) \prod
_{i=k+1}^0 Y_{N_{\varepsilon}+i}^{-\alpha}, &\quad
if $-N_{\varepsilon} \le k \le-1$.}
\end{equation}
Similar to the construction of $C_{\infty,*}$, the process
$(\varepsilon^{1/\alpha}F_*(\zeta-\varepsilon t), t \ge0)$ is piecewise
constant and may be constructed from $(Z_n,Y_n)_{n \geq0}$ as follows:
for $0<t \leq\zeta/\varepsilon$,
%
%e4.5 #&#
\begin{eqnarray}
\varepsilon^{1/\alpha} F_*(\zeta- \varepsilon t) = (Z_{N_{\varepsilon
}+k})^{1/\alpha}
\bigl(R_{\varepsilon}(k)\bigr)^{-1/\alpha}\nonumber
\\
\eqntext{\mbox{when } t \in
\bigl(R_{\varepsilon}(k+1),R_{\varepsilon}(k) \bigr]\mbox{ }\forall k \geq
-N_{\varepsilon}.}
\end{eqnarray}

Next, by Lemma~\ref{cvpoints} and the Skorokhod representation theorem
[the space $\mathbb(\mathbb R^2_+\times\mathcal S_1)^{{\mathbb Z}}
\times\mathbb R_+$ is Polish], for Lebesgue a.e. $z>0$, there exists
for all $\varepsilon>0$ a version of
\[
\biggl( (Z_{N_{\varepsilon}+n},Y_{N_{\varepsilon}+n},{\bolds{ \Delta
}}_{N_{\varepsilon}+n}
)_{n \in\mathbb Z},\frac{1}{\alpha}\log(\varepsilon/\zeta)-S_{N_{\varepsilon}}
\biggr) \Big|\zeta=z
\]
that converges almost surely as $\varepsilon\rightarrow0$ to a
version of $ ( (Z^{\mathrm{bias}},Y^{\mathrm{bias}},{\bolds
{\Delta}}^{\mathrm{bias}} ),\break U\log(Y^{\mathrm{bias}}_1)
)$. Then for all $t>0$ and all $\varepsilon\leq z$, construct from this
new version a process $\varepsilon^{1/\alpha}(\tilde F_*(\tilde\zeta-
\varepsilon t),t \geq0)$ (with $\tilde\zeta=z)$, exactly as $\varepsilon
^{1/\alpha}(F_*(\zeta-\varepsilon t),t \geq0)$ is constructed above from
\[
\biggl( (Z_{N_{\varepsilon}+n},Y_{N_{\varepsilon}+n},{\bolds{ \Delta
}}_{N_{\varepsilon}+n}
)_{n \in\mathbb Z},\frac{1}{\alpha}\log(\varepsilon/\zeta)-S_{N_{\varepsilon}}
\biggr).
\]
By Lemma~\ref{Skocv} in the \hyperref[appendix]{Appendix}, the c\`adl\`ag process $\varepsilon
^{1/\alpha}(\tilde F_*((\tilde\zeta- \varepsilon t)-), t \geq0 )$ then
converges almost surely as $\varepsilon\rightarrow0$ to a process
which is distributed as~$C_{\infty,*}$.
\end{pf*}

%%%%%%%%%%%%%%%%%%%%%%%%%%%%%%%%%%%%%
%s5 #&#
\section{The spine decomposition for the fragmentation} \label{Spine}
%%%%%%%%%%%%%%%%%%%%%%%%%%%%%%%%%%%%%

We are now ready to introduce our spine decomposition for a
fragmentation process. It may help the reader to refer to Figure~\ref
{figspinedecomp}. We need a little notation. Write $\bar{F}^{(x)}$ to
denote the (left-continuous) time-reversal of a fragmentation process
$F$ conditioned to become extinct before time $x$, that is, $\bar
{F}^{(x)}(0) = \mathbf{0}$, $\bar{F}^{(x)}(x) = \mathbf{1}$, $\bar
{F}^{(x)}$ is l\`adc\`ag on $\mathbb R_+$, and for any suitable test
function $f$,
\[
\mathbb{E} \bigl[f \bigl( \bar{F}^{(x)}(t), t \ge0 \bigr) \bigr] =
\mathbb{E} \bigl[f \bigl(F(x - t) \bigr)\mathbh{1}_{ \{ 0
\le t \le x \} }+f(\mathbf0)
\mathbh{1}_{ \{ t>x
\} }\mid\zeta< x \bigr].
\]
[We emphasize that $\bar{F}^{(x)}(t) = \mathbf{0}$ for $t > x$.]
Note that since $\bar{F}^{(x)}$ is l\`adc\`ag, the process $(\bar
{F}^{(x)}(t +),t \geq0)$ is c\`adl\`ag. Moreover, the probability that
a fragmentation process jumps at a fixed deterministic time $t$ is 0.
(This can be seen as a consequence of its Poissonian construction \cite
{BertoinHom,BertoinSSF}. Equivalently, and in a more elementary way
since the dislocation measure is finite here, this can be seen using
the genealogical description of the fragmentation\vspace*{1pt} developed in
Chapter~1.2 of \cite{BertoinBook}.)
It is clear from its definition that $\bar F^{(x)}$ inherits this
property on $(0,x)$.

Recall the definitions of $N_{\varepsilon}$ and $R_{\varepsilon}(k)$ from
(\ref{defNeps}) and (\ref{defRepsk}), respectively.

%pr5.1 #&#
\begin{prop}[(Spine decomposition)] \label{propspine}
On the event $\{N_{\varepsilon}\neq-\infty\}=\{\varepsilon\leq\zeta\}$,
the process $(F(\zeta- \varepsilon t), 0 < t \le\zeta/\varepsilon)$ can
be rewritten in the form
\begin{eqnarray*}
\hspace*{-5pt}&& \Biggl( \Biggl\{ \prod_{j=0}^{N_{\varepsilon}+K_{\varepsilon}(t)} \Theta_j,
\\
\hspace*{-5pt}&&\qquad{} \Biggl( \prod_{j=0}^{N_{\varepsilon}+i-1}
\Theta_j \Biggr) \Delta_{N_{\varepsilon}+i,m} \bar{F}_{i,m}^{(\Delta
_{N_{\varepsilon} +
i,m}^{\alpha} \Theta_{N_{\varepsilon} + i}^{-\alpha} Z_{N_{\varepsilon}
+ i})}
\Biggl(\varepsilon t \Biggl(\prod_{j=0}^{N_{\varepsilon} +i-1}
\Theta_j \Biggr)^{\alpha} \Delta_{N_{\varepsilon} + i,m}^{\alpha
}\Biggr)\dvtx
\\
\hspace*{-5pt}&&\hspace*{153pt} m \ge1, -N_{\varepsilon}+1 \le i \le K_{\varepsilon}(t) \Biggr
\}^{\downarrow}, 0 < t \le\zeta/\varepsilon\Biggr),
\end{eqnarray*}
where $K_{\varepsilon}(t)$ is the unique integer $k \ge-N_{\varepsilon}$
such that $R_{\varepsilon}(k) \ge t > R_{\varepsilon}(k+1)$, and $\bar
{F}_{i,m}^{(\Delta_{N_{\varepsilon} + i,m}^{\alpha} \Theta
_{N_{\varepsilon} + i}^{-\alpha} Z_{N_{\varepsilon} + i})}, i \in\Z, m
\ge1$ is a collection of conditioned fragmentation processes which are
independent for distinct $i$ and $m$, conditionally on $(Z_n, \Theta
_n, {\bolds\Delta}_n)_{n \ge0}$.
\end{prop}

Although this expression may seem a little intimidating, the idea
behind it is simple: the decreasing sequence $F(\zeta-\varepsilon t)$ is
composed of $F_*(\zeta-\varepsilon t)$ (the spine term) and the masses of
fragments coming from the fragmentation of all blocks that detached
from the spine $F_*$ before time $\zeta-\varepsilon t$.

\begin{pf*}{Proof of Proposition \ref{propspine}}
Consider the state of the fragmentation at some time $\zeta- \varepsilon
t$. Each block present is either the last fragment, or descends from a
block which split off from the last fragment at time $T_n$ for some $1
\le n \le N_{\varepsilon t}$ (this ensures that $T_n \le\zeta- {\varepsilon
t}$). In other words, the current state may be written as the
decreasing rearrangement of the blocks of
%
%e5.1 #&#
\begin{equation}
\label{state} \bigl( F_*(\zeta-\varepsilon t), \tilde{F}_{n,m}(\zeta-
\varepsilon t - T_n), m \ge1, 1 \le n \le N_{\varepsilon t} \bigr),
\end{equation}
where $\tilde{F}_{n,m}(s)$ represents the collection of blocks present
at time $T_n + s$ which are descended from the $m$th-largest block to
split off from $F_*$ at time $T_n$. Note that the process $\tilde
{F}_{n,m}$ must itself have extinction time at most $\zeta- T_n$
(since it must die before the last fragment), that is, $\tilde
{F}_{n,m}(s) = \mathbf{0}$ for some $s < \zeta- T_n$.

By construction,
\[
F_*(T_n) = \prod_{j=0}^n
\Theta_j.
\]
For $K_{\varepsilon}(t)$ defined to be the unique integer $k \ge
-N_{\varepsilon}$ such that $R_{\varepsilon}(k) \ge t > R_{\varepsilon
}(k+1)$, we have $N_{\varepsilon t} = N_{\varepsilon} + K_{\varepsilon}(t)$
and so
\[
F_*(\zeta-\varepsilon t) = F_*(T_{N_{\varepsilon t}}) = \prod
_{j=0}^{N_{\varepsilon t}} \Theta_j = \prod
_{j=0}^{N_{\varepsilon} +
K_{\varepsilon}(t)} \Theta_j.
\]
For $1 \le n \le N_{\varepsilon} + K_{\varepsilon}(t)$, the blocks
descending from the last fragment at time $T_{n-1}$ which split off
from the last fragment at time $T_n$ have sizes $\{F_*(T_{n-1}) \Delta
_{n,m}, m \ge1\}$, that is, $\tilde{F}_{n,m}(0) = F_*(T_{n-1})
\Delta_{n,m}$. Note that
\[
F_*(T_{n-1}) \Delta_{n,m} = \Biggl(\prod
_{j=0}^{n-1} \Theta_j \Biggr)
\Delta_{n,m}.
\]
Let us write $H_{n,m}(s) = (\tilde{F}_{n,m}(0))^{-1} \tilde
{F}_{n,m} ( \tilde{F}_{n,m}(0)^{-\alpha} s )$ for $\tilde
{F}_{n,m}$ with its natural time- and space-rescaling, in order that we
may later exploit the scaling property. We can then rewrite (\ref
{state}) as
\begin{eqnarray*}
&& \Biggl( \prod_{j=0}^{N_{\varepsilon} +K_{\varepsilon}(t)}
\Theta_j, \Biggl(\prod_{j=0}^{n-1}
\Theta_j \Biggr) \Delta_{n,m} H_{n,m} \Biggl( (
\zeta- \varepsilon t - T_n) \Biggl(\prod_{j=0}^{n-1}
\Theta_j \Delta_{n,m} \Biggr)^{\alpha} \Biggr),
\\
&&\hspace*{185pt} m \ge1, 1 \le n \le N_{\varepsilon} + K_{\varepsilon}(t)
\Biggr).
\end{eqnarray*}
Now observe that
\[
(\zeta- T_n) \Biggl(\prod_{j=0}^{n-1}
\Theta_j \Delta_{n,m} \Biggr)^{\alpha} =
Z_n \Theta_n^{-\alpha} \Delta_{n,m}^{\alpha},
\]
so that we in fact have
%
%e5.2 #&#
\begin{eqnarray}\label{state2}
&& \Biggl(\prod_{j=0}^{N_{\varepsilon} +K_{\varepsilon}(t)}
\Theta_j, \Biggl(\prod_{j=0}^{n-1}
\Theta_j \Biggr) \Delta_{n,m} H_{n,m} \Biggl(
Z_n \Theta_n^{-\alpha} \Delta_{n,m}^{\alpha}
- \varepsilon t \Biggl(\prod_{j=0}^{n-1}
\Theta_j \Biggr)^{\alpha} \Delta_{n,m}^{\alpha}
\Biggr),
\nonumber\\[-8pt]\\[-8pt]
&&\hspace*{203pt} m \ge1, 1 \le n \le N_{\varepsilon} + K_{\varepsilon}(t)
\Biggr).\hspace*{-20pt}\nonumber
\end{eqnarray}
So far, we know that $H_{n,m}$ is some sort of fragmentation process
which is started from $H_{n,m}(0) = \mathbf{1}$ and becomes extinct
before time $Z_n \Theta_n^{-\alpha} \Delta_{n,m}^{\alpha}$.

Suppose, temporarily, that we are on the event $\{N_{\varepsilon} +
K_{\varepsilon}(t) = 1\}$; in other words, by time $\zeta- \varepsilon t$,
the last fragment has split exactly once. Then, in the notation
introduced just before Lemma~\ref{TheLemma}, $H_{1,m} = H^{(I,m)}$,
and Lemma~\ref{TheLemma} entails that, \mbox{conditionally} on $(\Theta_0,
\Theta_1, Z_1, \Delta_{1,m}, m \ge1)$, $H_{1,m}$ is distributed as a
standard fragmentation process conditioned to become extinct before
time $\Delta_{1,m}^{\alpha} \Theta_1^{-\alpha} Z_1$, independently
for different $m \ge1$.
%(Of course, some of these fragmentation processes may already have
%been reduced to dust at the time we observe them, in which case we
%do not include their blocks.)
It follows that, in this case, (\ref{state2}) is distributed as
\[
\bigl(\Theta_0\Theta_1, \Theta_0
\Delta_{1,m} \bar{F}_{1,m}^{(\Delta_{1,m}^{\alpha} \Theta_1^{-\alpha}
Z_1)} \bigl(\varepsilon t
\Theta_0^{\alpha} \Delta_{1,m}^{\alpha} \bigr)
\bigr).
\]

To get to the result for general $N_{\varepsilon}$ and $K_{\varepsilon
}(t)$, note that Lemma~\ref{TheLemma} also tells us that,
conditionally on $(\Theta_0, \Theta_1, Z_1, \Delta_{1,m}, m \ge1)$,
the evolution of the last fragment after its first split (suitably
rescaled) is independent of the evolution of $H_{1,m}$ for $m \ge1$.
So we may apply Lemma~\ref{TheLemma} inductively, just as we did in
the proof of Proposition~\ref{propMarkovprop}, to obtain that (\ref
{state2}) has the same distribution as
\begin{eqnarray*}
&& \Biggl( \prod_{j=0}^{N_{\varepsilon} +K_{\varepsilon}(t)}
\Theta_j, \Biggl(\prod_{j=0}^{n-1}
\Theta_j \Biggr)\Delta_{n,m}
\bar{F}_{n,m}^{(\Delta_{n,m}^{\alpha} \Theta
_n^{-\alpha}Z_n)} \Biggl(\varepsilon t \Biggl(\prod
_{j=0}^{n-1} \Theta_j
\Biggr)^{\alpha} \Delta_{n,m}^{\alpha}
\Biggr),
\\
&& \hspace*{170pt} m \ge1, 1 \le n \le N_{\varepsilon}+K_{\varepsilon}(t) \Biggr).
\end{eqnarray*}
Finally, we will find it convenient to index the split times in such a
way that index $N_{\varepsilon}$ becomes 0. So we simply shift the
indices down by $N_{\varepsilon}$ (i.e., set $n = N_{\varepsilon} + i$).
Now notice that everything we have done here is consistent as we vary
$t$ in $\R_+$, and so we obtain the desired result.
\end{pf*}

So far, we have mainly thought of the spine decomposition in terms of
the forward direction of time for the fragmentation $(F(t), 0 \le t \le
\zeta)$, with blocks gradually detaching from the spine and then
further fragmenting until such a time as they are reduced to dust. We
now adopt the opposite perspective and view $\varepsilon^{1/\alpha
}F(\zeta- \varepsilon\cdot)$ as being composed of a spine plus other
blocks which immigrate into the system and gradually coalesce with one
another, before eventually coalescing with the spine. We group the
nonspine blocks together into sub-collections formed of those which
will attach to the spine at the same time. To this end, for $i \ge
-N_{\varepsilon}+1$, $m \ge1$ and $t \ge0$, define
\begin{eqnarray*}
H^{\varepsilon}_{i,m}(t) &=& \frac{\Delta_{N_{\varepsilon}+
i,m}Z_{N_{\varepsilon_n} + i-1}^{1/\alpha}}{(R_{\varepsilon
}(i-1))^{1/\alpha}}
\\
&&{}\times \bar{F}_{i,m}^{(Z_{N_{\varepsilon}+i-1}
Y_{N_{\varepsilon}+i}^{\alpha} \Delta^{\alpha}_{N_{\varepsilon}+i,m} )}
\bigl( t \Delta_{N_{\varepsilon}+i,m}^{\alpha}Z_{N_{\varepsilon}+i-1}
\bigl(R_{\varepsilon}(i-1)\bigr)^{-1} + \bigr),
\end{eqnarray*}
where\vspace*{1pt} $\bar{F}_{i,m}^{(Z_{N_{\varepsilon}+i-1} Y_{N_{\varepsilon
}+i}^{\alpha} \Delta^{\alpha}_{N_{\varepsilon}+i,m} )}$, $i \in
\mathbb Z, m \geq1$ is a collection of conditioned time-reversed
fragmentation processes which are independent for distinct $i$ and $m$,
conditionally on $(Z_n, Y_n, {\bolds\Delta}_n)_{n \ge0}$.
Let $H^{\varepsilon,\downarrow}_i(t)$ be the decreasing rearrangement of
all terms involved in the sequences $H^{\varepsilon}_{i,m}(t)$, $m \geq
1$. [Note that $H^{\varepsilon,\downarrow}_i(t) \in\mathcal S$ since
$\sum_{m \geq1}\Delta_{N_{\varepsilon}+ i,m} \leq1$.] Thus,
$H^{\varepsilon,\downarrow}_i$ tracks the evolution of the collection of
blocks which attach to the spine at time $R_{\varepsilon}(i)$. The spine
coalesces with other blocks only at times $R_{\varepsilon}(k), k \ge
-N_{\varepsilon}+1$.

Using this new notation, we can rewrite the expression for the spine
decomposition in Proposition~\ref{propspine} in a form more adapted
to our purposes.

%co5.2 #&#
\begin{cor} \label{corspine}
Suppose that $t \in[R_{\varepsilon}(k+1), R_{\varepsilon}(k)
)$ for some $k \geq-N_{\varepsilon}$. Then $\varepsilon^{1/\alpha}
F((\zeta- \varepsilon t)-)$ is the decreasing rearrangement of the masses
which make up:
\begin{itemize}
\item $Z_{N_{\varepsilon} + k}^{1/\alpha} (R_{\varepsilon
}(k))^{-1/\alpha}$;
\item $H^{\varepsilon, \downarrow}_{i}(t)$,
$-N_{\varepsilon}+1 \le i \le k$.
\end{itemize}
\end{cor}

By Lemma~\ref{cvpoints}, it is then, more or less, clear what the
limit process should be. Recall that $(Z^{\mathrm{bias}}_n, \Theta
^{\mathrm{bias}}_n, {\bolds\Delta}_n^{\mathrm{bias}})_{n \in\mathbb
Z}$ is the biased Markov chain introduced in Section~\ref{secStatPro}.
Let
\[
H_{i,m}(t)= \frac{\Delta^{\mathrm{bias}}_{i,m}(Z^{\mathrm
{bias}}_{i-1})^{1/\alpha}}{(R(i-1))^{1/\alpha}} \bar
{F}_{i,m}^{(Z^{\mathrm{bias}}_{i-1} (Y^{\mathrm{bias}}_{i})^{\alpha}
(\Delta_{i,m}^{\mathrm{bias}})^{\alpha} )}
\bigl( t \bigl(\Delta^{\mathrm{bias}}_{i,m}\bigr)^{\alpha}
Z_{i-1} ^{\mathrm{bias}} \bigl(R(i-1)\bigr)^{-1}+ \bigr),
\]
where $\bar{F}_{i,m}^{( Z_{i-1}^{\mathrm{bias}}(Y_i^{\mathrm
{bias}})^{\alpha} (\Delta_{i,m}^{\mathrm{bias}})^{\alpha})}$, $i
\in\mathbb Z, m \geq1$ is a collection of conditioned time-reversed
fragmentation processes which are independent for distinct $i$ and $m$,
conditionally on the chain $(Z^{\mathrm{bias}}_n, Y^{\mathrm
{bias}}_n, {\bolds\Delta}_n^{\mathrm{bias}})_{n \in\mathbb Z}$. Let
$H^{\downarrow}_i(t)$ be the decreasing \mbox{rearrangement} of all terms
involved in the sequences $H_{i,m}(t)$, $m \geq1$.

%de5.3 #&#
\begin{defn} \label{defnCinfty}
Let $C_{\infty}(0) = \mathbf{0}$.
For all $k \in\mathbb Z$ and all $t \in[R(k+1),R(k) )$,
let $C_{\infty}(t)$ be the decreasing rearrangement of the masses
which make up:
\begin{itemize}
\item $(Z_k^{\mathrm{bias}}) ^{1/\alpha}
(R(k))^{-1/\alpha}$;\vspace*{1pt}

\item $H^{\downarrow}_{i}(t)$,
$i \le k$.
\end{itemize}
\end{defn}

%f3 #&#
\begin{figure}

\includegraphics{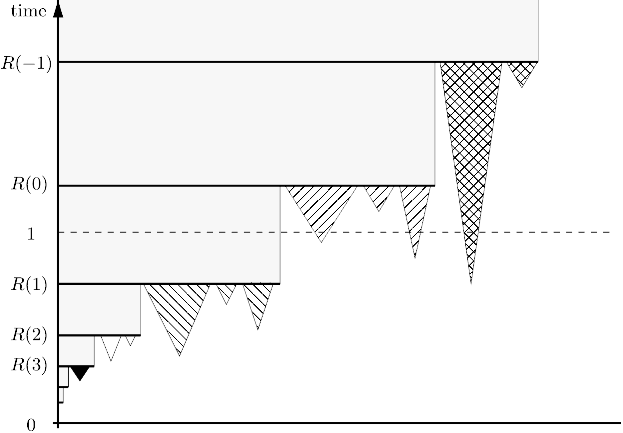}

\caption{The process $C_{\infty}$. The mass of the spine, which is piecewise constant, is shaded in
gray.  For each $k$, the mass of the spine at time $R(k)$ is the coagulation of the mass of the spine at
time $R(k)-$ together with the masses $H_k^{\downarrow}(R(k)-)$ of some other fragments present at
time $R(k)-$. The collections of patterned triangles represent, from left to right, the processes
$H_3^{\downarrow}$, $H_2^{\downarrow}$, $H_1 ^{\downarrow}$, $H_0^{\downarrow}$ and $H_{-
1}^{\downarrow}$ respectively.}\label{figCinfini}
\end{figure}

See Figure~\ref{figCinfini} for an illustration. In a rough sense,
the process $C_{\infty}$ models the evolution of masses that coalesce,
with a regular immigration of infinitesimally small masses. It turns
out that reversing time,
the distribution of $C_{\infty}$ can be related to the distribution of
a transformed biased fragmentation process in the following way. For
all $a$, recall that $C_{\infty,*}(a)$ denotes the mass at time $a$ of
the spine. For $0\leq t \leq a$, let $C_{\infty}^{a}(t)$ denote the
subsequence of $C_{\infty}(t)$ composed of all of the blocks which
will contribute to the mass $C_{\infty,*}(a)$ at time $a$. In other
words, we are looking at the coagulation history of $C_{\infty,*}(a)$.
Note that, for a fixed time $t$, each block of $C_{\infty}(t)$ belongs
to a sequence $C_{\infty}^{a}(t)$ for some $a$ sufficiently large. We
are interested in the distribution of the $(C_{\infty}^{a}(t),0\leq t
\leq a)$ process. By\vspace*{1pt} self-similarity it has the same distribution as
$(a^{1/\alpha}C_{\infty}^{1}(at),0\leq t \leq1)$, so we can focus on
the $C_{\infty}^{1}$ process. The proposition below connects the
distribution of this process to that of a biased fragmentation process.
We need the following elements:
\begin{itemize}
\item Let $Z_0^{\mathrm{stat}}$ be distributed according to $\pi
_{\mathrm{stat}}$, and independently, let $F$ be a fragmentation process.
\item Let $F_{\mathrm{stat}}$ be distributed as the process
$(Z_0^{\mathrm{stat}})^{1/\alpha} F (Z_0^{\mathrm{stat}} \cdot)$
conditioned to die at time $1$. Let $T_{\mathrm{stat},1}$ be the first
jump time of $F_{\mathrm{stat}}$.
\item Independently, let $U$ be uniformly distributed on $[0,1]$.
\end{itemize}

%pr5.4 #&#
\begin{prop} For all test functions $\phi$,
\label{CF}
%
%e5.3 #&#
\begin{eqnarray}\label{idCinfinity}
&& \mathbb E \bigl[\phi\bigl( C^1_{\infty}(t),
0 \leq t \leq1 \bigr) \bigr]\nonumber
\\
&& \qquad=\mathbb E \bigl[\log(1-T_{\mathrm{stat},1})
\nonumber\\[-8pt]\\[-8pt]\nonumber
&&\hspace*{42pt}{}\times \phi
\bigl((1-T_{\mathrm{stat},1})^{U/\alpha}F_{\mathrm{stat}}
\bigl(1- (1-T_{\mathrm{stat},1})^{U}t \bigr), 0 \leq t \leq1
\bigr) \bigr]
\\
&&\quad\qquad{} \times \bigl(\mathbb E \bigl[ \log(1-T_{\mathrm{stat},1}) \bigr]\bigr)^{-1}.\nonumber
\end{eqnarray}
\end{prop}

\begin{pf} First note that
%
%e5.4 #&#
\begin{equation}
\label{eqstat} \mathbb E \bigl[\phi(F_{\mathrm{stat}} ) \bigr] = \int
_{\mathbb R_+} \mathbb E \bigl[ \phi\bigl(x^{1/\alpha}F(x \cdot)
\bigr) \mid\zeta=x \bigr] \pi_{\mathrm{stat}}(\mathrm dx).
\end{equation}
Recall that the fragmentation $F$ can be constructed from the Markov
chain $(Z_n,Y_n,\Delta_n)_{n \geq0}$ and a collection of conditioned
fragmentation processes $\bar F_{i,m}$:  roughly, $F$ is then composed of
a spine $(F_*(T_n), n\geq1)$, where for $n \geq1$
\[
T_n=Z_0-Z_0 \prod
_{i=1}^n Y_i^{\alpha}, \qquad
F_*(T_n)=\prod_{i=1}^n
\Theta_i=\frac{Z_n^{1/\alpha}}{Z_0^{1/\alpha}\prod_{i=1}^n Y _i}, %
\]
from which, at each time $T_{n+1}$, blocks split off to give rise to
conditioned fragmentation processes
\[
\frac{Z_n^{1/\alpha}\Delta_{n+1,m}}{Z_0^{1/\alpha} \prod_{i=1}^n
Y_i} \bar F_{n+1,m}^{(\Delta_{n+1,m}^{\alpha}Z_nY_{n+1}^{\alpha})}
\biggl(
\Delta_{n+1,m}^{\alpha}Z_n \biggl(Y_{n+1}^{\alpha}
-\frac
{(\cdot-T_{n+1})}{Z_0 \prod_{i=1}^n Y^{\alpha}_i} \biggr) \biggr).
\]
These conditioned processes are independent given $(Z_n,Y_n,\Delta
_n)_{n \geq0}$.
From (\ref{eqstat}), we see that $F_{\mathrm{stat}}$ is constructed
similarly from $(Z^{\mathrm{stat}}_n,Y^{\mathrm{stat}}_n,\Delta
^{\mathrm{stat}}_n)_{n \geq0}$, a stationary version of
$(Z_n,Y_n,\Delta_n)_{n \geq0}$, and a collection of conditioned
fragmentation processes as follows: $F_{\mathrm{stat}}$ is composed of
a spine $(F_{\mathrm{stat},*}(T_{\mathrm{stat},n}), n\geq1)$, where
for $n \geq1$
\[
T_{\mathrm{stat},n}=1- \prod_{i=1}^n
\bigl(Y^{\mathrm{stat}}_i\bigr)^{\alpha
}, \qquad
F_*(T_{\mathrm{stat},n})=\frac{(Z^{\mathrm
{stat}}_n)^{1/\alpha}}{\prod_{i=1}^n Y^{\mathrm{stat}}_i}, %
\]
and from this spine, blocks split off at times $T_{\mathrm
{stat},{n+1}}$ to give rise to conditioned fragmentation processes
\begin{eqnarray*}
&& \frac{(Z^{\mathrm{stat}}_n)^{1/\alpha}\Delta^{\mathrm
{stat}}_{n+1,m}}{ \prod_{i=1}^n Y^{\mathrm{stat}}_i}
\\
&&\qquad{} \times\bar
F_{n+1,m}^{((\Delta^{\mathrm{stat}}_{n+1,m})^{\alpha}Z^{\mathrm
{stat}}_n(Y^{\mathrm{stat}}_{n+1})^{\alpha})} \biggl(\bigl(\Delta
^{\mathrm{stat}}_{n+1,m}\bigr)^{\alpha} Z^{\mathrm{stat}}_n
\biggl(\bigl(Y^{\mathrm{stat}}_{n+1}\bigr)^{\alpha}-
\frac{ (\cdot-T_{\mathrm
{stat},n+1} )}{ \prod_{i=1}^n (Y^{\mathrm{stat}}_i)^{\alpha}} \biggr)
\biggr).
\end{eqnarray*}
To finish, multiply $F_{\mathrm{stat}}$ by $(1-T_{\mathrm
{stat},1})^{U/\alpha}$, perform the time change $t \mapsto
1-(1-T_{\mathrm{stat},1})^{U} t$ and
note that $1-T_{\mathrm{stat},1}=(Y_1^{\mathrm{stat}})^{\alpha}$. In
order to obtain the expression in (\ref{idCinfinity}), we must now
take a biased version of this stationary construction. It suffices to
compare this biased, scaled and time-changed version of $F_{\mathrm
{stat}}$ with Definition~\ref{defnCinfty} to conclude the argument.
\end{pf}

%%%%%%%%%%%%%%%%%%%%%%%%%%
%s6 #&#
\section{Convergence of the full fragmentation}\label{fullfrag}
%%%%%%%%%%%%%%%%%%%%%%%%%%

The aim of this section is to prove Theorem~\ref{teomainabstract}.
Throughout, we will assume that $\nu$ is nongeometric and that $\int
_{\mathcal S}s_1^{-1-\rho}\nu(\mathrm d \mathbf s)<\infty$ for some
$\rho>0$.
We start by establishing several preliminary lemmas.

%s6.1 #&#
\subsection{Preliminary lemmas}

We first deal with an important redundancy in our expression for
$(C_{\infty}(t), t \ge0)$: for each time $t$, most of the
$H_{i,m}(t)$ do not contribute.

%le6.1 #&#
\begin{lem} \label{lemfinite}
Consider the expression for $(C_{\infty}(t), t \ge0)$ given in
Definition~\ref{defnCinfty}. Then almost surely for all $t>0$,
$t\notin\{R(k),k \in\mathbb Z\}$, only finitely many indices $i$ and
$m$ contribute nonzero blocks to $C_{\infty}(t)$.
\end{lem}

\begin{pf} We start by proving that only finitely many indices $i$ and
$m$ contribute nonzero blocks to the state a.s. for a \emph{fixed}
$t > 0$. By the self-similarity of $C_{\infty}$, it suffices to prove
that this holds for $t=1$. Recall, moreover, that $R(1) < 1 < R(0)$ a.s.
By the first Borel--Cantelli lemma, it suffices to check that the
following sum is almost surely finite:
%
%e6.1 #&#
\begin{eqnarray}\label{eqnsum}
&& \sum_{i \le0} \sum_{m=1}^{\infty}
\mathbb{P} \bigl(H_{i,m}(1) \neq\mathbf{0} \mid Z^{\mathrm
{bias}},Y^{\mathrm{bias}},
\Delta^{\mathrm{bias}} \bigr)
\nonumber
\\
&&\qquad  = \sum_{i\le0} \sum_{m =1}^{\infty}
\mathbb{P} \bigl(\bar{F}_{i,m}^{(Z_{i-1}^{\mathrm{bias}}
(Y_i^{\mathrm
{bias}})^{\alpha}
(\Delta_{i,m}^{\mathrm{bias}})^{\alpha} )}
\bigl( \bigl(\Delta
_{i,m}^{\mathrm{bias}}\bigr)^{\alpha}Z_{i-1}^{\mathrm{bias}}
\bigl(R(i-1)\bigr)^{-1} + \bigr) \neq\mathbf{0} \mid
\\
&&\hspace*{244pt} Z^{\mathrm
{bias}},Y^{\mathrm{bias}}, \Delta^{\mathrm{bias}} \bigr).\hspace*{-20pt} \nonumber
\end{eqnarray}
For any $x > 0$ and any $0 \le u \le x$,
\[
\mathbb{P} \bigl(\bar{F}^{(x)}(u) \neq\mathbf{0} \bigr) = \mathbb{P} (
\zeta> x-u \mid\zeta< x ) = \frac{\mathbb
{F}_{\zeta}(x) - \mathbb{F}_{\zeta}(x-u)}{\mathbb{F}_{\zeta}(x)}.
\]
Using Lemma~\ref{LemmaDensity}(ii), we see that
\[
\mathbb{P} \bigl(\bar{F}^{(x)}(u) \neq\mathbf{0} \bigr) \le d
\frac{u \exp(-c(x-u))}{\mathbb{F}_{\zeta}(x)}
\]
for some constants $c,d > 0$. Hence (\ref{eqnsum}) is bounded above by
\[
d\sum_{i\le0} \sum_{m =1}^{\infty}
Z_{i-1}^{\mathrm{bias}} \bigl(\Delta_{i,m}^{\mathrm{bias}}
\bigr)^{\alpha} \bigl(R(i-1)\bigr)^{-1}\frac{\exp( -c
Z_{i-1}^{\mathrm{bias}} (Y_i^{\mathrm{bias}})^{\alpha} (\Delta
_{i,m}^{\mathrm{bias}})^{\alpha} (1 - (R(i))^{-1} )
)}{
\mathbb{F}_{\zeta} (Z_{i-1}^{\mathrm{bias}} (Y_i^{\mathrm
{bias}})^{\alpha} (\Delta_{i,m}^{\mathrm{bias}})^{\alpha} ) }
\]
[note that $R(i) > 1 $ for all $i \leq0$].
Using the monotonicity of $\mathbb{F}_{\zeta}$, we see that we only
require the finiteness of
%
%e6.2 #&#
\begin{eqnarray}
\label{eqwantfinite} && \sum_{i\le0} \frac{(R(i))^{-1} } {\mathbb{F}_{\zeta}
(Z_{i-1}^{\mathrm{bias}} (Y_i^{\mathrm{bias}})^{\alpha} )}
\sum
_{m =1}^{\infty} Z_{i-1}^{\mathrm{bias}}
\bigl(Y_i^{\mathrm{bias}}\bigr)^{\alpha} \bigl(
\Delta_{i,m}^{\mathrm{bias}}\bigr)^{\alpha}
\nonumber\\[-8pt]\\[-8pt]\nonumber
&& \hspace*{114pt}{}\times \exp\bigl( -c
Z_{i-1}^{\mathrm{bias}} \bigl(Y_i^{\mathrm{bias}}
\bigr)^{\alpha} \bigl(\Delta_{i,m}^{\mathrm{bias}}
\bigr)^{\alpha} \bigl(1 - \bigl(R(i)\bigr)^{-1} \bigr) \bigr).\hspace*{-15pt}
\end{eqnarray}
Since $\sum_{m =1}^{\infty}\Delta^{\mathrm{bias}}_{i,m} < 1$, we
have $\Delta^{\mathrm{bias}}_{i,m} \le m^{-1}$, and so the inner sum
in $m$ is bounded above by
%
%e6.3 #&#
\begin{equation}
\label{eqinner} C \bigl(1 - \bigl(R(i)\bigr)^{-1} \bigr)^{-1}
\sum_{m =1}^{\infty} \exp\bigl(
-c' Z_{i-1}^{\mathrm{bias}} \bigl(Y_i^{\mathrm{bias}}
\bigr)^{\alpha} m^{-\alpha} \bigl(1 - \bigl(R(i)\bigr)^{-1}
\bigr) \bigr),
\end{equation}
for some constants $C > 0$ and $0 < c' < c$. Now observe that for
$\theta> 0$,
\[
\sum_{m=1}^{\infty} \exp\bigl(-\theta
m^{-\alpha}\bigr) \le\int_0^{\infty} \exp\bigl(-
\theta x^{-\alpha}\bigr)\, \mathrm{d} x = \Gamma\biggl(1-\frac
{1}{\alpha}
\biggr) \theta^{1/\alpha},
\]
and so (\ref{eqinner}) is bounded above by
\[
C'\bigl(Z_{i-1}^{\mathrm{bias}}\bigr)^{1/\alpha}
Y_i^{\mathrm{bias}} \bigl(1 - \bigl(R(i)\bigr)^{-1}
\bigr)^{-1 +1/\alpha}
\]
for some constant $C' > 0$.
Since by Lemma~\ref{lemras} we have that $R(i) \rightarrow\infty$
as $i \rightarrow-\infty$ almost surely, there exists $i_0<0$ such
that for all $i\leq i_0$, $(1 - (R(i))^{-1})^{-1+1/\alpha}$ is bounded
above, say by 2.
For $i \le i_0$, let
\[
B_i = \frac{2(R(i))^{-1}} {\mathbb{F}_{\zeta}
(Z_{i-1}^{\mathrm{bias}} (Y_i^{\mathrm{bias}})^{\alpha} )} \bigl
(Z_{i-1}^{\mathrm{bias}}
\bigr)^{1/\alpha} Y_{i}^{\mathrm{bias}}.
\]
Then by Lemma~\ref{lemLLN},
\begin{eqnarray*}
\lim_{n \to\infty} \frac{1}{n} \log(B_{-n}) & =&
\lim_{n \to\infty} \frac{\alpha}{n} \sum
_{j=-n+1}^0 \log\bigl(Y_j^{\mathrm{bias}}
\bigr) - \lim_{n \to\infty} \frac{1}{n} \log\bigl(
\mathbb{F}_{\zeta}\bigl(Z_{-n-1}^{\mathrm{bias}}
\bigl(Y_{-n}^{\mathrm
{bias}}\bigr)^{\alpha}\bigr)\bigr)
\\
&&{}+ \lim_{n \to\infty} \frac{1}{\alpha n} \log
\bigl(Z_{-n-1}^{\mathrm{bias}}\bigr)+ \lim_{n \to\infty}
\frac{1}{n} \log\bigl(Y_{-n}^{\mathrm{bias}}\bigr)
\\
& =& \alpha\mu< 0.
\end{eqnarray*}
Hence, by Cauchy's root test, (\ref{eqwantfinite}) is almost surely finite.

The statement of the lemma now follows easily: we know that
almost surely for all rational numbers $q \in\mathbb Q \cap(0,\infty
)$, only finitely many indices $i$ and $m$ contribute nonzero blocks
to the state $C_{\infty}(q)$. On this event of probability one, for
each positive time $t \notin\{R(k),k \in\mathbb Z\}$, say $t \in
(R(k+1),R(k))$, consider a rational number $q \in(t,R(k))$. Since all
indices $i,m$ that contribute to the state $C_{\infty}(t)$ also
contribute to the state $C_{\infty}(q)$, the statement follows.
\end{pf}

%le6.2 #&#
\begin{lem}
\label{lemCinS}
$C_{\infty}$ is almost surely a c\`adl\`ag process taking values in
$(\mathcal{S},d)$.
\end{lem}

\begin{pf}
We first prove that, with probability one, $C_{\infty}(t) \in\mathcal
{S}$ for all $t \geq0$. By Lemma~\ref{lemfinite}, with probability
one, for all $t \notin\{R(k),k \in\mathbb Z \}$, $t >0$, $\llVert
C_{\infty
}(t)\rrVert _1<\infty$. If $t=0$, $C_{\infty}(t)=\mathbf{0}$. Finally, for
$t=R(k)$ for some $k$, we can argue via monotonicity. Let $u \in
(R(k),R(k-1))$. Then $\llVert C_{\infty}(t)\rrVert _1 \leq\llVert
C_{\infty}(u)\rrVert _1
<\infty$ on the event of probability one we just considered.

We now turn to the continuity properties. We first show that $\llVert
C_{\infty}(t) \rrVert _1\to0$ as $t \downarrow0$. First, recall that
$C_{\infty,*}(t) \to0$ as $t \downarrow0$ (this was noted at the
beginning of Section~\ref{MarkovRW}, as a consequence of Lemma~\ref
{lemras}). Now fix $\varepsilon> 0$. Then we can find $t_{\varepsilon} >
0$ such that $C_{\infty,*}(t_{\varepsilon}) < \varepsilon/2$. Moreover, we
can always assume that $t_{\varepsilon}$ is not one of the $R(k)$ and,
therefore, that there are only finitely many indices $i$ and $m$ which
contribute to the state of $C_{\infty}(t_{\varepsilon})$. Since the
total mass in each of these fragmentations is decreasing to 0, it
follows that there exists some time $t_{\varepsilon}' \in(0,t_{\varepsilon
})$ such that $\llVert C_{\infty}(t'_{\varepsilon}) \rrVert _1< \varepsilon$.

Now consider a fixed time $t \in(0,\infty)$, and suppose that $t \in
[R(k+1), R(k))$ for some $k \in\Z$.
Take $(t_n)_{n \in\N}$ to be such that $t_1<R(k)$ and $t_n \downarrow
t$. Then $C_{\infty}(t_n)$ is the decreasing rearrangement of
$C_{\infty,*}(R(k+1))$ together with the blocks of $H_{i,m}(t_n)$ for
$m \ge1$, $i \le k$.
There are only finitely many indices $i$ and $m$ which contribute to
the nonzero blocks of $C_{\infty}(t_1)$, and blocks can only
disappear as $t_n$ decreases in $(R(k+1), R(k))$. Hence
\[
\sum_{i=-\infty}^k \sum
_{m=1}^{\infty} \bigl\llVert H_{i,m}(t_n)
- H_{i,m}(t_{\infty}) \bigr\rrVert_1
\]
is\vspace*{1pt} a sum with only finitely many nonzero terms. Since $\bar
{F}_{i,m}^{( Z_{i-1}^{\mathrm{bias}}(Y_i^{\mathrm{bias}})^{\alpha}
(\Delta_{i,m}^{\mathrm{bias}})^{\alpha})}(\cdot+ )$ is c\`adl\`ag
for each $i,m$, each term converges to 0, and so the whole sum
converges to~0. Using Lemma~\ref{lemdecroissance}, we deduce that
$\llVert
C_{\infty}(t_n) - C_{\infty}(t)\rrVert _1 \to0$.

The existence of a left limit at time $t \in(0,\infty)$ such that $t
\in(R(k+1), R(k))$ follows similarly, because again the same finite
collection of indices $i,m$ are involved for all $t' \in(t-\varepsilon,t)$
for sufficiently small $\varepsilon> 0$. Finally, for times $t$ such
that $t = R(k)$ for some $k \in\Z$, there is a slight difference
since the number of indices in the set $\{(k,m),m\geq1\}$ that are
involved may be infinite. However, the result still holds by Lemma~\ref
{lemdecroissance}, since
\[
\sum_{m \geq m_{\eta}} \Delta^{\mathrm{bias}}_{k,m}
\bigl(Z^{\mathrm
{bias}}_{k-1}\bigr)^{1/\alpha}\bigl(R(k-1)
\bigr)^{-1/\alpha} \leq\eta
\]
for some finite $m_{\eta}$ and all $\eta>0$.
\end{pf}

We now turn to an important tightness result, which will allow us to
ignore, in the proof of Theorem~\ref{teomainabstract}, the
possibility that there exist blocks in the system at time $R_{\varepsilon
}(k)$ which persist for a very long time before coalescing with the spine.
From now on, we use its spine decomposition, as discussed in the
previous section. For each $\varepsilon>0$ and each $k \in\mathbb Z$,
let $I_\varepsilon(k)$ be the largest positive integer $i$ such that at
least one nonspine block present at time $R_{\varepsilon}(k)$ attaches
to the spine at time $R_\varepsilon(k-i)$.
Formally, when $k \geq-N_{\varepsilon}+1$,
\[
I_\varepsilon(k)=\sup\bigl\{ 1 \le i \le k+N_{\varepsilon} - 1\dvtx
H_{k-i}^{\varepsilon,\downarrow
}\bigl(R_{\varepsilon}(k)\bigr) \neq\mathbf0
\bigr\}, %
\]
with the convention that $I_\varepsilon(k)=0$ if this is the supremum of
an empty set. We also set $I_\varepsilon(k)=0$ when $k<-N_{\varepsilon}+1$.
Our goal is prove that with a high probability $ I_\varepsilon(k)$ is not
too large, simultaneously for all $\varepsilon$ small enough.

%le6.3 #&#
\begin{lem}[(Tightness)]
\label{lemtight1}
Let $z>0$ be fixed and such that the convergence in distribution of
Lemma~\ref{cvpoints} holds. Consider a sequence $(\varepsilon_n)_{n \in
\mathbb N}$ of strictly positive real numbers converging to 0. Then
there exists a family of positive integers $(j_{\eta}(k))$ indexed by
$k\in\mathbb Z, \eta>0$ such that
\[
\mathbb P \bigl(I_{\varepsilon_n}(k) \geq j_{\eta}(k) \mid
Z_0=z \bigr) \leq\eta\qquad\forall n \in\mathbb N.
\]
Consequently, $\forall n \in\mathbb N$,
\begin{eqnarray*}
&& \mathbb P \bigl( \bigl\{ I_{\varepsilon_n}(0) \geq j_{\eta}(0) \bigr\}
\cup\bigl\{ \exists k \in\mathbb Z\setminus\{0\}\dvtx  I_{\varepsilon_n}(k)\geq j_{\eta/k^2}(k) \bigr\} \mid Z_0=z \bigr)
\\
&&\qquad \leq\bigl(1+2
\pi^2/6\bigr) \eta. %
\end{eqnarray*}
\end{lem}

Having in mind the construction of $C_{\infty}$, we define similarly
$I(k)$, $k \in\mathbb Z$ to be the largest integer $i \geq1$ such at
least one nonspine block present at time $R(k)$ attaches to the spine
at time $R(k-i)$ [and $I(k)=0$ if no such $i \geq1$ exists]. As a
direct consequence of Lemmas~\ref{lemfinite} and~\ref{lemtight1},
we have the following result, which is in the form we will use later
for the proof of Theorem~\ref{teomainabstract}.

%le6.4 #&#
\begin{lem}
\label{lemtight}
Let $z>0$ be fixed and such that the convergence of Lemma \ref
{cvpoints} holds. Consider a sequence $(\varepsilon_n)_{n \in\mathbb N}$
of strictly positive real numbers converging to 0. Then there exists a
family of positive integers $(i_{\eta}(k))$ indexed by $k\in\mathbb
Z, \eta>0$ such that
\[
\mathbb P \bigl( \exists k \in\mathbb Z\dvtx  I_{\varepsilon_n}(k) \geq
i_{\eta}(k) \mid Z_0=z \bigr) \leq\eta\qquad\forall n \in
\mathbb N
\]
and
\[
\mathbb P \bigl( \exists k \in\mathbb Z\dvtx  I(k) \geq i_{\eta}(k) \bigr)
\leq\eta. %
\]
\end{lem}

In order to prove Lemma~\ref{lemtight1}, we gather together some
technical results in the following lemma. They follow from Lemmas~\ref
{logZfiniteexpect},~\ref{mu},~\ref{lemLLN2} and~\ref{lemexptails}
in the \hyperref[appendix]{Appendix}.

%le6.5 #&#
\begin{lem} \label{lemsummary}
We have that for $p > 0$ and $\delta> 0$ sufficiently small,
\[
\mathbb{E} \bigl[\bigl\llvert\log\bigl(Z_0^{\mathrm{stat}}\bigr)
\bigr\rrvert^p \bigr] < \infty,\qquad\mathbb{E} \bigl[\bigl(\log
\bigl(Y_1^{\mathrm{stat}}\bigr)\bigr)^p \bigr] < \infty
\]
and
\[
\mathbb{E} \bigl[\bigl\llvert\log\bigl(\mathbb
{F}_{\zeta}\bigl(Z_0^{\mathrm{stat}} \bigl(Y_1^{\mathrm{stat}}
\bigr)^{\alpha
}\bigr)\bigr)\bigr\rrvert^{1+\delta} \bigr] < \infty.
\]
Moreover, there exist constants $A < \infty$ and $c_Y \in(0,1)$ such that
\[
\mathbb{E} \Biggl[\prod_{i=2}^n
\bigl(Y_i^{\mathrm{stat}}\bigr)^{\alpha} \Biggr] \le A
c_Y^n.
\]
\end{lem}

\begin{pf*}{Proof of Lemma~\ref{lemtight1}}
The proof is similar for all $k \in\mathbb Z$, and so, in order to
ease the notation, we will only write it out in the case where $k=0$.
In the following lines, $\eta>0$ is fixed, and $C$ denotes a finite
positive constant that may vary from line to line.

Our main goal is to prove the existence of $N_\eta\in\mathbb Z$ and
$\varepsilon_\eta>0$ such that
%
%e6.4 #&#
\begin{equation}
\label{eqobjectif} \mathbb P \bigl(I_{\varepsilon}(0) \geq N_{\eta} \mid
Z_0=z \bigr) \leq\eta\qquad\forall0 \leq\varepsilon\leq
\varepsilon_\eta.
\end{equation}
Since $(\varepsilon_n)_{n \in\mathbb N}$ is a sequence of strictly
positive real numbers converging to 0, this will imply the existence of
a positive integer $j_{\eta}(0)$ such that
\[
\mathbb P \bigl(I_{\varepsilon_n}(0) \geq j_{\eta}(0) \mid
Z_0=z \bigr) \leq\eta\qquad\forall n \in\mathbb N, %
\]
as expected.

Now, in order prove (\ref{eqobjectif}), note that for all integers $N
\geq1$, following the main lines of the proof of Lemma~\ref
{lemfinite}, we obtain that
\begin{eqnarray*}
&& \mathbb P \bigl(I_{\varepsilon}(0) \geq N \mid(Z,Y,\Delta) \bigr)
\\
&&\qquad \leq  C
\sum_{i=N}^{N_{\varepsilon}-1}\frac{R_{\varepsilon}(0)R_{\varepsilon
}(-i)^{-1}}{(1-R_{\varepsilon}(0)R_{\varepsilon}(-i)^{-1})^{1-1/\alpha}}
\frac{Z_{N_{\varepsilon}-i-1}^{1/\alpha} Y_{N_{\varepsilon}-i}}{\mathbb
{F}_{\zeta} (Z_{N_{\varepsilon}-i-1} Y_{N_{\varepsilon}-i}^{\alpha}
)}
\\
&&\qquad  \leq C A_{\varepsilon} B_{\varepsilon}(N),
\end{eqnarray*}
where
\begin{eqnarray*}
A_{\varepsilon}&=&\frac{R_{\varepsilon}(0)\exp(-\alpha S_{N_{\varepsilon
}}+\log(\varepsilon/Z_0))}{(1-R_{\varepsilon}(0)R_{\varepsilon
}(-1)^{-1})^{1-1/\alpha}},
\\
B_{\varepsilon}(N)&=&\sum_{i=N}^{N_{\varepsilon}-1}
\Biggl(\prod_{k=-i+1}^0 Y_{N_{\varepsilon}+k}^{\alpha}
\Biggr)\frac{Z_{N_{\varepsilon
}-i-1}^{1/\alpha} Y_{N_{\varepsilon}-i}}{\mathbb{F}_{\zeta}
(Z_{N_{\varepsilon}-i-1} Y_{N_{\varepsilon}-i}^{\alpha} )}. %
\end{eqnarray*}
Consequently, for every $A>0$,
\begin{eqnarray*}
\mathbb P \bigl(I_{\varepsilon}(0) \geq N \mid Z_0=z \bigr) &\leq&
\frac{\eta}{3} \mathbb P\bigl(A_{\varepsilon}B_{\varepsilon}(N) \leq
\eta/3C \mid Z_0=z\bigr)
\\
&&{}+ \mathbb P (A_{\varepsilon} \geq A/3C \mid Z_0=z ) +
\mathbb P \bigl(B_{\varepsilon}(N) \geq\eta/A \mid Z_0=z \bigr).
\end{eqnarray*}
But we know from Lemma~\ref{cvpoints} that when $\varepsilon\rightarrow
0$, conditional on $Z_0=z$,
\[
A_{\varepsilon} \stackrel{\mathrm{law}}\rightarrow\frac
{R(0)(Y_1^{\mathrm{bias}})^{\alpha U}}{(1-R(0)R(-1)^{-1})^{1-1/\alpha},}
\]
and the limit is almost surely finite.
Hence if we fix $A$ sufficiently large, then for all $\varepsilon$
sufficiently small, say $\varepsilon\leq\varepsilon_0$, and all $N \geq1$,
\[
\mathbb P \bigl(I_{\varepsilon}(0) \geq N \mid Z_0=z \bigr) \leq
\frac
{2\eta}{3} + \mathbb P \bigl(B_{\varepsilon}(N) \geq\eta/A \mid
Z_0=z \bigr). %
\]
Let us now deal with this last probability ($A$ is now fixed). We have
\begin{eqnarray*}
&& \mathbb P \bigl(B_{\varepsilon}(N) \geq\eta/A \mid Z_0=z \bigr)
\\
&&\qquad\leq\sum_{i=N}^{\infty} \mathbb P
\Biggl( \Biggl(\prod_{k=-i+1}^0
Y_{N_{\varepsilon}+k}^{\alpha} \Biggr)\frac{Z_{N_{\varepsilon
}-i-1}^{1/\alpha} Y_{N_{\varepsilon}-i}}{\mathbb{F}_{\zeta}
(Z_{N_{\varepsilon}-i-1} Y_{N_{\varepsilon}-i}^{\alpha} )} \mathbh
1_{\{i \leq N_{\varepsilon}-1 \}} \geq\frac{6\eta}{A \pi^2 i^2} \bigg| Z_0=z
\Biggr)
\end{eqnarray*}
(since\vspace*{1pt} $\sum_{i \geq N} i^{-2} \leq\pi^2/6$). Recall that, on the
event $\{Z_0 = z\}$, we have $\{N_{\varepsilon}=j\}=\{S_j \leq\alpha
^{-1} \log(\varepsilon/z)<S_{j+1}\}$. Hence
\begin{eqnarray*}
&&\mathbb P \Biggl( \Biggl(\prod_{k=-i+1}^0
Y_{N_{\varepsilon}+k}^{\alpha
} \Biggr)\frac{Z_{N_{\varepsilon}-i-1}^{1/\alpha} Y_{N_{\varepsilon
}-i}}{\mathbb{F}_{\zeta} (Z_{N_{\varepsilon}-i-1} Y_{N_{\varepsilon
}-i}^{\alpha} )} \mathbh1_{\{i \leq N_{\varepsilon}-1 \}}
\geq\frac{6\eta}{A \pi^2 i^2} \bigg|Z_0=z \Biggr)
\\
&&\qquad =\sum_{j=i+1}^{\infty} \mathbb P \Biggl(
\Biggl(\prod_{k=-i+1}^0 Y_{j+k}^{\alpha}
\Biggr)\frac{Z_{j-i-1}^{1/\alpha} Y_{j-i}}{\mathbb
{F}_{\zeta} (Z_{j-i-1} Y_{j-i}^{\alpha} )} \geq\frac{6\eta
}{A \pi^2 i^2},
\\
&&\hspace*{131pt} S_j \leq
\frac{\log(\varepsilon/z)}{\alpha} <S_{j+1} \bigg|  Z_0=z \Biggr)
\\
&&\qquad = \sum_{j=0}^{\infty} \mathbb P \Biggl(
\Biggl(\prod_{k=-i+1}^0 Y_{j+i+1+k}^{\alpha}
\Biggr)\frac{Z_{j}^{1/\alpha} Y_{j+1}}{\mathbb
{F}_{\zeta} (Z_{j} Y_{j+1}^{\alpha} )} \geq\frac{6\eta}{A
\pi^2 i^2},
\\
&&\hspace*{94pt} S_{j+i+1} \leq
\frac{ \log(\varepsilon/z)}{\alpha} <S_{j+i+2} \bigg| Z_0=z \Biggr)
\\
&&\qquad = \sum_{j=0}^{\infty} \mathbb E \biggl[
g_i \biggl(Z_j, \frac{\log
(\varepsilon/z)}{\alpha} -S_j
\biggr) \Big| Z_0=z \biggr],
\end{eqnarray*}
where for $x>0$, $y \in\mathbb R$ and $i \ge1$,
\[
g_i(x,y)=\mathbh1_{\{y \geq0\}}\mathbb P \Biggl( \Biggl( \prod
_{k=2}^{i+1} Y_k^{\alpha}
\Biggr) \frac{x^{1/\alpha}Y_1}{\mathbb
{F}_{\zeta}(x Y_1^{\alpha})}\geq\frac{6\eta}{A \pi^2 i^2}, S_{i+1} \leq y
<S_{i+2} \bigg|\zeta=x \Biggr). %
\]
So, finally,
\[
\mathbb P \bigl(B_{\varepsilon}(N) \geq\eta/A \mid Z_0=z \bigr)
\leq\sum_{j=0}^{\infty} \mathbb E \bigl[ g
\bigl(Z_j, \alpha^{-1} \log(\varepsilon/z)
-S_j \bigr) \mid Z_0=z \bigr], %
\]
where
$
g(x,y)=\sum_{i \geq N}g_i(x,y)$.
Assume for the moment that this function $g$ satisfies conditions (a)
and (b) of Theorem~\ref{teoAlsmeyer}. Then, as a consequence of that theorem,
\[
\limsup_{\varepsilon\rightarrow0} \mathbb P \bigl(B_{\varepsilon}(N) \geq\eta/A
\mid Z_0=z \bigr)\leq\frac{1}{\mu} \int_{\mathbb R_+}
\int_{\mathbb R_+} \sum_{i \geq N}g_i(x,y)
\pi_{\mathrm
{stat}}(\mathrm d x)<\infty. %
\]
We then use the monotone convergence theorem to conclude that there
exists some $N_{\eta}$ and then some $\varepsilon_{\eta}$ $(\leq
\varepsilon_0)$ such that
\[
\mathbb P \bigl(B_{\varepsilon}(N_{\eta}) \geq\eta/A \mid
Z_0=z \bigr) \leq\frac{\eta}{3}\qquad\forall\varepsilon\leq
\varepsilon_{\eta}, %
\]
which was the missing piece we needed to get (\ref{eqobjectif}).

It remains to check that $g$ satisfies conditions (a) and (b) of
Theorem~\ref{teoAlsmeyer}. Note that we do not even know yet that
$g(x,y)<\infty$ for Lebesgue a.e. $x,y$. We start with (b). For this,
note that if $y \in[n,n+1)$ and $y \in[S_{i+1},S_{i+2})$, then
$S_{i+2} > n$ and $S_{i+1} <n+1$. Moreover, the number of integers $n
\in(S_{i+1}-1,S_{i+2})$ is smaller than $S_{i+2}-(S_{i+1}-1)+1=\log
(Y_{i+2})+2$. Thus
\begin{eqnarray*}
&& \int_{0}^{\infty} \sum_{n \in\mathbb Z_+}
\sup_{y \in[n,n+1)} g(x,y)\pi_{\mathrm{stat}}(\mathrm dx)
\\
&&\qquad \leq\sum
_{i \geq N}\mathbb E \bigl[\bigl(\log\bigl(Y^{\mathrm{stat}}_{i+2}
\bigr)+2\bigr)
\\
&&\hspace*{59pt}{}\times \mathbh1_{ \{ (
\prod_{k=2}^{i+1} (Y^{\mathrm{stat}}_k)^{\alpha} )
{(Z_0^{\mathrm{stat}})^{1/\alpha}Y^{\mathrm{stat}}_1}/
(\mathbb
{F}_{\zeta}(Z_0^{\mathrm{stat}} (Y^{\mathrm{stat}}_1)^{\alpha
}))\geq {6\eta}/({A \pi^2 i^2}) \} } \bigr]. %
\end{eqnarray*}
Fix $\delta\in(0,1)$. By H\"older's inequality, and for any $c \in(0,1)$,
\begin{eqnarray*}
&& \mathbb E \bigl[\bigl(\log\bigl(Y^{\mathrm{stat}}_{i+2}\bigr)+2\bigr)
\mathbh1_{\{ ( \prod_{k=2}^{i+1} (Y^{\mathrm{stat}}_k)^{\alpha} )
{(Z_0^{\mathrm{stat}})^{1/\alpha}Y^{\mathrm{stat}}_1}/
(\mathbb
{F}_{\zeta}(Z_0^{\mathrm{stat}} (Y^{\mathrm{stat}}_1)^{\alpha
}))\geq {6\eta}/({A \pi^2 i^2}) \} } \bigr]
\\
&&\qquad \leq\mathbb E \bigl[\bigl(\log\bigl(Y^{\mathrm{stat}}_{1}\bigr)+2
\bigr)^{1/\delta
} \bigr]^\delta\mathbb P \Biggl( \Biggl( \prod
_{k=2}^{i+1} \bigl(Y^{\mathrm{stat}}_k
\bigr)^{\alpha} \Biggr) \frac{(Z_0^{\mathrm
{stat}})^{1/\alpha}Y^{\mathrm{stat}}_1}{\mathbb{F}_{\zeta
}(Z_0^{\mathrm{stat}} (Y^{\mathrm{stat}}_1)^{\alpha})}\geq\frac
{6\eta}{A \pi^2 i^2}
\Biggr)^{1-\delta}
\\
&&\qquad \leq C \Biggl(\mathbb P \Biggl( c^{-i} \Biggl(\prod
_{k=2}^{i+1} \bigl(Y^{\mathrm{stat}}_k
\bigr)^{\alpha} \Biggr) \geq\frac{6\eta}{A \pi^2
i^2} \Biggr)^{1-\delta}
\\
&&\hspace*{54pt}{} + \mathbb P \biggl(c^i\frac{(Z_0^{\mathrm
{stat}})^{1/\alpha}Y^{\mathrm{stat}}_1}{\mathbb{F}_{\zeta
}(Z_0^{\mathrm{stat}} (Y^{\mathrm{stat}}_1)^{\alpha})} \geq1 \biggr
)^{1-\delta}
\Biggr),
\end{eqnarray*}
where, in the last inequality, we have used the finiteness of the
expectation $\mathbb E [(\log(Y^{\mathrm{stat}}_{1}))^{1/\delta
} ]$; see Lemma~\ref{lemsummary}. By Markov's inequality, the
first probability on the right-hand side above is smaller than
$Ci^2(c_Y/c)^i$, where $c_Y$ is defined in Lemma~\ref{lemsummary}.
For the second term on the right-hand side, first take the logarithm
inside the probability, and then use Markov's inequality to bound it
from above by
\begin{eqnarray*}
&& C i^{-(1+2 \delta)} \bigl( \mathbb E \bigl[ \bigl\llvert\log
\bigl(Z_0^{\mathrm
{stat}}\bigr) \bigr\rrvert^{1+2 \delta} \bigr]+
\mathbb E \bigl[ \bigl\llvert\log\bigl(Y_1^{\mathrm
{stat}}\bigr) \bigr
\rrvert^{1+2 \delta} \bigr]
\\
&&\hspace*{56pt}\qquad {} +\mathbb E \bigl[ \bigl\llvert\log\bigl
(\mathbb
{F}_{\zeta}\bigl(Z_0^{\mathrm{stat}}\bigl(Y^{\mathrm{stat}}_1
\bigr)^{\alpha}\bigr)\bigr) \bigr\rrvert^{1+2 \delta} \bigr] \bigr).
\end{eqnarray*}
By Lemma~\ref{lemsummary}, this sum of three expectations is finite
for $\delta>0$ small enough.
Consequently, for $c \in(c_Y,1)$,
\[
\int_{0}^{\infty} \sum_{n \in\mathbb Z_+}
\sup_{y \in[n,n+1)} g(x,y)\pi_{\mathrm{stat}}(\mathrm dx) \leq\sum
_{i \geq N} C \biggl( i^2 \biggl(
\frac{c_Y}{c} \biggr)^{i(1-\delta)} + i^{-(1+2 \delta
)(1-\delta)} \biggr),
\]
and this sum on $i \geq N$ is finite as soon as $(1+2 \delta)(1-\delta
)>1$, that is, as soon as $\delta<1/2$. Hence, condition (b) of
Theorem~\ref{teoAlsmeyer} is satisfied.

To get condition (a), note that we have shown in the last paragraph
that for all $c\in(0,1)$ and all $\delta>0$ small enough,
\[
\mathbb P \Biggl( \Biggl( \prod_{k=2}^{i+1}
\bigl(Y^{\mathrm
{stat}}_k\bigr)^{\alpha} \Biggr)
\frac{(Z_0^{\mathrm{stat}})^{1/\alpha
}Y^{\mathrm{stat}}_1}{\mathbb{F}_{\zeta}(Z_0^{\mathrm{stat}}
(Y^{\mathrm{stat}}_1)^{\alpha})}\geq\frac{6\eta}{A \pi^2
i^2} \Biggr) \leq C \bigl(i^2(c_Y/c)^i+
i^{-1-2\delta} \bigr),
\]
where $C$ depends both on $c$ and $\delta$, but not on $i \geq N$.
Consequently, considering $c\in(c_Y,1)$, we get
\[
\int_0^{\infty} \sum_{i \geq N}
\mathbb P \Biggl( \Biggl( \prod_{k=2}^{i+1}
Y_k^{\alpha} \Biggr) \frac{x^{1/\alpha}Y_1}{\mathbb
{F}_{\zeta}(x Y_1^{\alpha})}\geq\frac{6\eta}{A \pi^2 i^2}
\bigg| \zeta=x \Biggr) \pi_{\mathrm{stat}}(\mathrm d x)<\infty, %
\]
hence for Lebesgue a.e. $x>0$, $\sum_{i \geq N}\mathbb P (
( \prod_{k=2}^{i+1} Y_k^{\alpha} ) \frac{x^{1/\alpha
}Y_1}{\mathbb{F}_{\zeta}(x Y_1^{\alpha})}\geq\frac{6\eta}{A \pi
^2 i^2} \mid\zeta=x )$ is finite. For those $x$, $g(x,y)<\infty$
for all $y\geq0$ and we can apply the dominated convergence theorem to
deduce that $g(x, \cdot)$ is continuous at each point which is not an
atom of one of the $S_i, i \geq1$. The Lebesgue measure of this set of
atoms is~0; hence condition (a) is also satisfied.
\end{pf*}

%s6.2 #&#
\subsection{Proof of Theorem \texorpdfstring{\protect\ref{teomainabstract}}{1.1}}
Consider a sequence $(\varepsilon_n)$ of strictly positive real numbers
converging to 0, and recall from Corollary~\ref{corspine} the spine
construction of
\[
\bigl(\varepsilon_n^{1/\alpha}F\bigl((\zeta-\varepsilon_n
t)-\bigr),0\leq t \leq\zeta/\varepsilon_n\bigr)
\]
in terms of the Markov chain $(Z_k,Y_k,\bolds\Delta_k)_{k \geq0}$
and the time-reversed fragmentations
\[
\bar F_{i,m}^{(Z_{N_{\varepsilon_n}+i-1} Y_{N_{\varepsilon_n}+i}^{\alpha}
\Delta^{\alpha}_{N_{\varepsilon_n}+i,m} )}, \qquad i \in\mathbb Z, m \geq1,
\]
where these fragmentations are conditionally independent given
$(Z_k,Y_k,\break \bolds\Delta_k)_{k \geq0}$.
For the rest of this proof, we fix $z>0$ such that the conditional
convergence of Lemma~\ref{cvpoints} holds.

\begin{longlist}[\textit{Step} 2.]
\item[\textit{Step} 1.] As we have already mentioned, an important
technical issue is the possibility that, among the blocks present at
time $t$, there are some which will persist in the system for a very
long time before coalescing with spine. In other words, we would like
to be able to say that
$H_i^{\varepsilon_n,\downarrow}$
does not contribute to the state for large negative $i$ (uniformly in $n$).
For this reason, we introduce, for all $\eta>0$ and $n \in\mathbb N$,
the modified process
\[
\bigl(\varepsilon_n^{1/\alpha}F^{(\eta)}\bigl((\zeta-
\varepsilon_n t)-\bigr),0\leq t \leq\zeta/\varepsilon_n\bigr)
\]
whose spine decomposition is constructed from $(Z_k,Y_k,\bolds\Delta
_k)_{k \geq0}$ in a way very similar to $(\varepsilon
_n^{1/\alpha}F((\zeta-\varepsilon_n \cdot)-))$ except that some terms
are omitted: for $t \in[R_{\varepsilon_n}(k+1), R_{\varepsilon
_n}(k) )$, $k \geq-N_{\varepsilon_n}$, we take $\varepsilon
_n^{1/\alpha} F^{(\eta)}((\zeta- \varepsilon_n t)-)$ to be the
decreasing rearrangement of the terms involved in:
\begin{itemize}
\item  $Z_{N_{\varepsilon_n} + k}^{1/\alpha} (R_{\varepsilon
_n}(k))^{-1/\alpha}$,

\item  $H^{\varepsilon_n,\downarrow}_{i}(t)$,
$k-i_{\eta}(k) \le i \le k$,
\end{itemize}
where the (deterministic) integers $i_{\eta}(k)$ are those introduced
in Lemma~\ref{lemtight}. If $t>\zeta/\varepsilon_n$, we set $F^{(\eta
)}((\zeta-\varepsilon_n t)-)=\mathbf1$. By Lemma~\ref{lemtight}, the
processes $F^{(\eta)}((\zeta-\varepsilon_n \cdot)-)$ and $F((\zeta
-\varepsilon_n \cdot)-)$ are identical with a high probability
independently of $n$, namely
\[
\mathbb P \bigl(\bigl(\varepsilon_n^{1/\alpha}F^{(\eta)}
\bigl((\zeta-\varepsilon_n t)-\bigr),t \geq0\bigr)\neq\bigl(
\varepsilon_n^{1/\alpha}F\bigl((\zeta-\varepsilon_n t)-
\bigr), t \geq0\bigr) \mid\zeta=z \bigr) \leq\eta. %
\]
Consequently, for every bounded continuous test function $f\dvtx \mathcal S
\rightarrow\mathbb R$,
\[
\bigl\llvert\mathbb E \bigl[f \bigl(\varepsilon_n^{1/\alpha}F
\bigl((\zeta-\varepsilon_n \cdot)-\bigr) \bigr) \mid\zeta=z \bigr]-
\mathbb E \bigl[f \bigl(\varepsilon_n^{1/\alpha}F^{(\eta)}
\bigl((\zeta-\varepsilon_n \cdot)-\bigr) \bigr) \mid\zeta=z \bigr]\bigr
\rrvert\leq C\eta,
\]
where $C$ is independent of $n$ and $\eta$. Similarly, again by Lemma
\ref{lemtight}, $ \llvert \mathbb E [f (C_\infty)
]- \mathbb E[f( C_\infty^{(\eta)})]\rrvert \leq C\eta$, where
$C_{\infty
}^{(\eta)}(t)$ is defined for $t \in[R(k+1), R(k) )$ to
be the decreasing rearrangement of the terms involved in:
\begin{itemize}
\item  $(Z_k^{\mathrm{bias}}) ^{1/\alpha}
(R(k))^{-1/\alpha}$,\vspace*{1pt}

\item  $H^{\downarrow}_{i}(t)$, $k - i_{\eta}(k) \le i
\le k$.
\end{itemize}
Therefore, the expected convergence in distribution will be proved if
we show that the process $(\varepsilon_n^{1/\alpha}F^{(\eta)}((\zeta
-\varepsilon_n \cdot)-))$ converges in distribution (conditional on
$\zeta=z$) to $C_{\infty}^{(\eta)}$, for each $\eta>0$.

\item[\textit{Step} 2.] Fix $\eta>0$. Our goal is to prove that
conditionally on $\zeta=z$, there exist versions of $(\varepsilon
_n^{1/\alpha}F^{(\eta)}((\zeta-\varepsilon_n t)-),0\leq t \leq\zeta
/\varepsilon_n)$, $n \in\mathbb N$, that converge to a version of
$C^{(\eta)}_{\infty}$, almost surely as $\varepsilon_n \rightarrow0$.
With step~1 above, this will clearly entail Theorem~\ref{teomainabstract}.

By Lemma~\ref{cvpoints} and the Skorokhod representation theorem,
conditionally on \mbox{$\zeta=z$}, there exist versions of
%
%e6.5 #&#
\begin{equation}
\label{eqcvps} \biggl( (Z_{N_{\varepsilon_n}+k},Y_{N_{\varepsilon_n}+k},{\bolds{
\Delta
}}_{N_{\varepsilon_n}+k} )_{k \in\mathbb Z},\frac{1}{\alpha}\log(
\varepsilon_n/\zeta)-S_{N_{\varepsilon_n}} \biggr)
\end{equation}
that converge almost surely as $\varepsilon_n \rightarrow0$ to a version
of $ ( (Z^{\mathrm{bias}},Y^{\mathrm{bias}},{\bolds{\Delta
}}^{\mathrm{bias}} ),\break U\log(Y^{\mathrm{bias}}_1) )$.
From now on, we always consider these versions. Using Lemma \ref
{cvSkobarF}, we get the joint Skorokhod convergence in distribution,
conditional on $\zeta=z$, of the c\`adl\`ag processes
\[
H^{\varepsilon_n}_{i,m} \rightarrow H_{i,m}\qquad\mbox{as }
\varepsilon_n \rightarrow0,  i \in\Z,m \geq1.
\]
By the Skorokhod representation theorem, we may again assume that these
convergences hold almost surely. Without changing notation, we work
with these versions for the rest of this proof. In fact, we will
implicitly work on the event of probability one where the convergence
of (\ref{eqcvps}) to $ ( (Z^{\mathrm{bias}},Y^{\mathrm
{bias}},{\bolds{\Delta}}^{\mathrm{bias}} ),\break U\log(Y^{\mathrm
{bias}}_1) )$ holds, as well as all convergences of processes
$H^{\varepsilon_n}_{i,m}$ to $H_{i,m}$, $i \in\mathbb Z, m\geq1$.

\item[\textit{Step} 2(a).] We then claim that for each $i \in
\mathbb Z$,
\[
H^{\varepsilon_n,\downarrow}_{i} \rightarrow H^{\downarrow}_{i}
\qquad\mbox{as }\varepsilon_n \rightarrow0, %
\]
in the Skorokhod sense (for the distance $d$ on $\mathcal S$). To see
this, we use Proposition~\ref{PropEK} and Lemma~\ref{cvSkorearr} from
the \hyperref[appendix]{Appendix}. For this, fix a time $t \geq0$ and a sequence
$(t_{\varepsilon_n})$ converging to $t$. The integer $i$ being fixed, our
goal is to check that the functions $H^{\varepsilon_n,\downarrow}_{i}$
and $H^{\downarrow}_{i}$ satisfy assertions (a), (b) and (c) of
Proposition~\ref{PropEK} for the sequence of times $(t_{\varepsilon
_n})$. In order to do this, we distinguish three cases: $t \in
[0,\infty)\setminus\{R(i),0\}$, $t=0$ and $t=R(i)$.

First assume that $t \neq R(i)$ and $t>0$. Since a reversed
fragmentation process $\bar F^{(x)}$ almost surely does not jump at any
given fixed time except $x$, the processes $H_{i,m},m\geq1$ cannot
jump simultaneously on $\mathbb R_{+}\setminus\{R(i)\}$. So at most
one process among $H_{i,m}, m \geq1$ jumps at time $t$ (almost
surely). Let $m_t$ be the index of this process if it exists. For $m
\neq m_t$, $H^{\varepsilon_n}_{i,m}(t_{\varepsilon_n}) \rightarrow
H_{i,m}(t)$ and this leads to the convergence in $\mathcal S$ of the
decreasing rearrangement of all terms involved in at least one sequence
$H^{\varepsilon_n}_{i,m}(t_{\varepsilon_n})$ for some $m \neq m_t$, to the
decreasing rearrangement of all terms involved in at least one sequence
$H_{i,m}(t)$ for some $m \neq m_t$, although the number of $m$ involved
may be infinite. Indeed, this is due to the continuity property for
finite decreasing rearrangements (Lemma~\ref{lemcontinuity}) and to
the fact that
\begin{eqnarray*}
\sum_{m \geq M} \bigl\llVert H^{\varepsilon_n}_{i,m}(t_{\varepsilon_n})
\bigr\rrVert_1 &\leq& Z_{N_{\varepsilon_n} + i-1}^{1/\alpha}
\bigl(R_{\varepsilon_n}(i-1)\bigr)^{-1/\alpha
} \sum
_{m \geq M} \Delta_{N_{\varepsilon_n}+ i,m}
\\
&\displaystyle\mathop{\rightarrow}_{\varepsilon_n \rightarrow0}&
\bigl(Z^{\mathrm{bias}}_{i-1}\bigr)^{1/\alpha} \bigl(R(i-1)
\bigr)^{-1/\alpha}\sum_{m
\geq M}\Delta^{\mathrm{bias}}_{i,m},
\end{eqnarray*}
which implies that for all $\delta>0$ there exists $M_{\delta} \in
\mathbb N$ such that for all $\varepsilon_n$ small enough,
%
%e6.6 #&#
\begin{equation}
\label{restespetits} \sum_{m \geq M_{\delta}} \bigl\llVert
H^{\varepsilon_n}_{i,m}(t_{\varepsilon_n})\bigr\rrVert_1
\leq\delta\quad\mbox{and}\quad\sum_{m \geq M_{\delta}} \bigl\llVert
H_{i,m}(t)\bigr\rrVert_1 \leq\delta.
\end{equation}
Hence, $\sum_{m\geq1,m\neq m_t}d(H^{\varepsilon_n}_{i,m}(t_{\varepsilon
_n}),H_{i,m}(t)) \to0$, and so, by Lemma~\ref{lemdecroissance}, the
decreasing rearrangement $\{H^{\varepsilon_n}_{i,m}(t_{\varepsilon_n}),m
\neq m_t\}^{\downarrow}$ converges in $\mathcal S$ to $\{H_{i,m}(t),\break m
\neq m_t\}^{\downarrow}$. Now, we also have that $H^{\varepsilon
_n}_{i,m_t}$ converges in the Skorokhod sense to $H_{i,m_t}$. It
follows, using Lemma~\ref{cvSkorearr}(i), that $H^{\varepsilon
_n,\downarrow}_{i}$ and $H^{\downarrow}_i$ satisfy assertions (a),
(b) and (c) of Proposition~\ref{PropEK} for the sequence of times
$(t_{\varepsilon_n})$.

Next assume that $t=0$. Let $(s_k)_{k \in\mathbb N}$ be a decreasing
sequence of strictly positive times that are not jump times of
$H^{\downarrow}_i$, and that converge to 0. Then, since $s_k \neq
R(i)$ and $s_k>0$, as we have just seen,
\[
\bigl\llVert H^{\varepsilon_n,\downarrow}_{i}(s_k) \bigr\rrVert
_1\mathop\rightarrow_{n \rightarrow\infty}\bigl\llVert
H^{\downarrow}_{i}(s_k) \bigr\rrVert_1
\qquad\forall k \in\mathbb N.
\]
We conclude by using a monotonicity argument: for all $k$ and then all
$\varepsilon_n$ sufficiently small, we have $t_{\varepsilon_n}\leq s_k$,
and so
\[
\bigl\llVert H^{\varepsilon_n,\downarrow}_{i}(t_{\varepsilon_n}) \bigr\rrVert
_1 \leq\bigl\llVert H^{\varepsilon_n,\downarrow}_{i}(s_k)
\bigr\rrVert_1,
\]
and then
\[
\limsup_{\varepsilon_n \rightarrow0} \bigl\llVert H^{\varepsilon_n,\downarrow
}_{i}(t_{\varepsilon_n})
\bigr\rrVert_1 \leq\bigl\llVert H^{\downarrow}_{i}(s_k)
\bigr\rrVert_1 \leq\bigl\llVert C_{\infty}(s_k)
\bigr\rrVert_1\qquad\forall k \in\mathbb N.
\]
Now let $k \rightarrow\infty$, so that $\llVert C_{\infty}(s_k)
\rrVert _1
\rightarrow0$, by the right-continuity of $C_{\infty}$ at 0. Hence,
$H^{\varepsilon_n,\downarrow}_{i}(t_{\varepsilon_n}) \rightarrow\mathbf
0=H^{\downarrow}_i(0)$ as $\varepsilon_n \rightarrow0$.

Finally, for $t=R(i)$, consider the subsequences $(t_{\varepsilon_{\phi
(n)}})$ and $(t_{\varepsilon_{\psi(n)}})$ of $(t_{\varepsilon_n})$
characterized by
\begin{eqnarray*}
R_{\varepsilon_{\phi(n)}}(i) &\leq& t_{\varepsilon_{\phi(n)}}<R_{\varepsilon
_{\phi(n)}} (i-1),
\\
R_{\varepsilon_{\psi(n)}}(i+1) &\leq& t_{\varepsilon_{\psi(n)}} < R_{\varepsilon
_{\psi(n)}} (i).
\end{eqnarray*}
For\vspace*{1pt} $N$ large enough, there always exists a $n$ such that either
$N=\phi(n)$ or $N=\psi(n)$. Since $H^{\varepsilon_n,\downarrow
}_{i}(s)=\mathbf0$ for all $s \geq R_{\varepsilon_n}(i)$, we clearly
have that
\[
H^{\varepsilon_n,\downarrow}_{i} (t_{\varepsilon_{\phi(n)}}) \rightarrow
\mathbf0=
H^{\downarrow}_{i} (t). %
\]
Next,\vspace*{1pt} note that $H^{\varepsilon_n}_{i,m}(t_{\varepsilon_{\psi(n)}})
\rightarrow H_{i,m}(t-)$ for all $m \geq1$. Moreover, similar to~(\ref{restespetits}), for all $\delta> 0$, there exists an integer
$M_{\delta}$ such that for all $\varepsilon_n$ small enough,
\[
\sum_{m \geq M_{\delta}} \bigl\llVert H^{\varepsilon_n}_{i,m}(t_{\varepsilon
_{\psi
(n)}})
\bigr\rrVert\leq\delta\quad\mbox{and}\quad\sum_{m \geq M_{\delta}}
\bigl\llVert H_{i,m}(t-)\bigr\rrVert\leq\delta. %
\]
From this and Lemma~\ref{lemdecroissance} we deduce that
\[
H^{\varepsilon_n,\downarrow}_{i} (t_{\varepsilon_{\psi(n)}}) \rightarrow
H^{\downarrow}_{i} (t-). %
\]
Assertion (a) of Proposition~\ref{PropEK} follows. To get assertion
(b), note that if
\[
H^{\varepsilon_n,\downarrow}_{i} (t_{\varepsilon_{n}}) \rightarrow\mathbf
0=H^{\downarrow}_{i} (t),
\]
then necessarily $R_{\varepsilon_{n}}(i)\leq t_{\varepsilon
_{n}}<R_{\varepsilon_{n}} (i-1)$ for $n$ large enough [since
$H^{\downarrow}_{i} (t-) \neq \mathbf{0}$]. Hence if $(s_{\varepsilon_{n}})$ is a
sequence converging to $t$ with $s_{\varepsilon_{n}} \geq t_{\varepsilon
_n}$, one has $R_{\varepsilon_{n}}(i) \leq s_{\varepsilon_{n}}<R_{\varepsilon
_{n}} (i-1)$ for $n$ large enough and then $
H^{\varepsilon_n,\downarrow}_{i} (s_{\varepsilon_{n}})=\mathbf0=
H^{\downarrow}_{i} (t)$. We obtain assertion (c) similarly.

\item[\textit{Step} 2(b).]
Conditionally on $\zeta=z$, we consider for all $n$ the version of
%
%e6.7 #&#
\begin{equation}
\label{eqFeta} \bigl(\varepsilon_n^{1/\alpha}F^{(\eta)}
\bigl((\zeta-\varepsilon_n t)-\bigr),0\leq t \leq\zeta/
\varepsilon_n \bigr)
\end{equation}
built from the chain $ (Z_{N_{\varepsilon_n}+k},Y_{N_{\varepsilon
_n}+k},{\bolds{ \Delta}}_{N_{\varepsilon_n}+k} )_{k \in\mathbb Z}$,
the real number\break $\frac{1}{\alpha}\log(\varepsilon_n/\zeta
)-S_{N_{\varepsilon_n}}$ and the processes $H^{\varepsilon_n,\downarrow
}_{i}, i \in\mathbb Z$. We know that (almost surely) these quantities
converge, respectively, to
$ (Z^{\mathrm{bias}},Y^{\mathrm{bias}},{\bolds{\Delta}}^{\mathrm
{bias}} )$,\break $U\log(Y^{\mathrm{bias}}_1)$ and \mbox{$H^{\downarrow
}_{i}, i \in\mathbb Z$}. To prove that this version of (\ref{eqFeta})
converges for the Skorokhod topology as $\varepsilon_n \rightarrow0$ to
a version of $C^{(\eta)}_{\infty}$ [indeed, the version constructed
from $ (Z^{\mathrm{bias}},Y^{\mathrm{bias}},{\bolds{\Delta
}}^{\mathrm{bias}} ), U\log(Y^{\mathrm{bias}}_1)$ and
$H^{\downarrow}_{i}, i \in\mathbb Z$], we will again use Proposition
\ref{PropEK} and Lemma~\ref{cvSkorearr}.

We start by proving the Skorokhod convergence on any compact set $[a,b]
\subseteq(0,\infty)$. Let $R(k_a)$ be the largest $R(k)$ strictly
smaller than $a$ and similarly $R(k_b)$ be the smallest $R(k)$ strictly
larger than $b$. For all $\varepsilon_n$ small enough, $R_{\varepsilon
_n}(k_a)<a$ and $R_{\varepsilon_n}(k_b)>b$. This implies that the
processes\break $(\varepsilon_n^{1/\alpha}F^{(\eta)}((\zeta-{\varepsilon_n} t
)-), t \in[a,b])$ and $(C^{(\eta)}_{\infty}(t),t \in[a,b])$ are
constructed from the sequences $H^{\varepsilon_n,\downarrow}_{i}$ and
$H^{\downarrow}_i$, respectively, with $k_b-i_{\eta}(k_b) \le i \le
k_a-1$ [together with the terms $Z_{N_{\varepsilon_n} + k}^{1/\alpha}
(R_{\varepsilon_n}(k))^{-1/\alpha}$, $(Z_k^{\mathrm{bias}}) ^{1/\alpha
} (R(k))^{-1/\alpha}$, for $k_b \le k \le k_a-1$]. Crucially, the
number of processes $H^{\varepsilon_n,\downarrow}_{i}, H^{\downarrow
}_{i}$ involved in these constructions is finite, independently of $n$.
Moreover, the processes $H^{\downarrow}_{i}, i \in\mathbb Z$ do not
jump simultaneously (almost surely). We can therefore apply Lemma~\ref
{cvSkorearr}(ii) to obtain the Skorokhod convergence of $\varepsilon
_n^{1/\alpha}F^{(\eta)}((\zeta-{\varepsilon_n} \cdot)-)$ to $C^{(\eta
)}_{\infty}$ on any compact set $[a,b] \subseteq(0,\infty)$.

It remains to check that for any sequence $(t_{\varepsilon_n})$
converging to 0, $\varepsilon_n^{1/\alpha}F^{(\eta)}((\zeta-{\varepsilon
_n} t_{\varepsilon_n} )-)$ converges to $\mathbf0=C^{(\eta)}_{\infty
}(0)$. This can be done via a monotonicity argument, exactly as in the
case $t=0$ of step~2(a).
\end{longlist}

%%%%%%%%%%%%%%%%%%%%%%%%%%%%%%
%s7 #&#
\section{An invariant measure for the fragmentation process}\label{secinvariant}
%%%%%%%%%%%%%%%%%%%%%%%%%%%%%%

This section is devoted to the proof of Theorem~\ref{teoinvariant}.
Throughout, we will assume that the assumptions of Theorem~\ref
{teomainabstract} are satisfied.
Recall that the occupation measure $\lambda$ on $(\mathcal{S},
\mathcal{B}(\mathcal{S}))$ is defined by
\[
\lambda(A) = \int_0^{\infty} \mathbb{P}
\bigl(C_{\infty}(t) \in A \bigr)\, \mathrm{d} t\quad\mbox{for all } A
\in
\mathcal{B}(\mathcal{S}).
\]
By definition of the process $C_{\infty}$, it is clear that $\lambda
(\{\mathbf0\})=0$ and also, using its self-similarity, that
\[
\lambda\bigl(\{\mathbf s \in\mathcal S\dvtx  s_i \leq a_i
x, \forall i \geq1\} \bigr)=x^{-\alpha}\lambda\bigl( \{\mathbf s \in
\mathcal S\dvtx  s_i \leq a_i, \forall i \geq1\} \bigr)
\]
for all $a_i \geq0$ and all $x>0$.

Recall the notation $\llVert \mathbf s\rrVert _1=\sum_{i \geq1} s_i$
for $\mathbf
s \in\mathcal S$.
Our goal in this section is to prove first that
\[
\lambda\bigl(\bigl\{ \mathbf s \in\mathcal S\dvtx  \llVert\mathbf s\rrVert
_1 \leq x\bigr\}\bigr)<\infty\qquad\forall x>0
\]
(which implies that $\lambda$ is $\sigma$-finite)
and second that
\[
\int_{\mathcal S} f(\mathbf s)\lambda(\mathrm d \mathbf s) = \int
_{\mathcal{S}} \mathbb E_{\mathbf s} \bigl[f\bigl(F(u)\bigr)
\bigr] \lambda(\mathrm{d} \mathbf{s})
\]
for all $u> 0$ and all continuous functions $f\dvtx \mathcal S \rightarrow
\mathbb R_+$ such that
$f(\mathbf s) \leq\mathbh1_{\{0< \llVert \mathbf s\rrVert _1 \leq c\}
}$ for some $c>0$.

%le7.1 #&#
\begin{lem}
\label{lemcveps}
For all continuous functions $f\dvtx \mathcal S \rightarrow\mathbb R_+$
such that
$f(\mathbf s) \leq\mathbh1_{\{\llVert \mathbf s\rrVert _1 \leq c\}}$
for some $c>0$,
\[
\int_0^{\infty} \mathbb E \bigl[ f \bigl(
\varepsilon^{1/\alpha} F(\zeta-\varepsilon t) \bigr) \bigr]\,\mathrm dt
\mathop\rightarrow_{\varepsilon\rightarrow0}\int_{\mathcal S}f(\mathbf s)\lambda
(\mathrm d \mathbf s) \in[0,\infty). %
\]
\end{lem}

\begin{pf} To simplify the notation, we assume that $c=1$; a similar
argument works for a general $c>0$.
By Theorem~\ref{teomainabstract}, for all $t>0$, $\mathbb E [ f
( \varepsilon^{1/\alpha} F(\zeta-\varepsilon t) ) ]
\rightarrow\mathbb E [ f (C_{\infty}(t) ) ]$. It
remains to check that we can apply the dominated convergence theorem.
For this, we introduce for every $a>0$ the stopping time $\tau
_{a}=\inf{\{u \geq0\dvtx  \llVert F(u) \rrVert _1 \leq a\}}$. By
Proposition~\ref
{strongMarkov}, we may write
\[
\zeta-\tau_a=\sup_{i \geq1} \bigl\{F_i(
\tau_a)^{-\alpha}\zeta^{(i)} \bigr\}, %
\]
where the $\zeta^{(i)}$s are i.i.d. distributed as $\zeta$ and
independent of $F(\tau_a)$. Hence, for all $\beta\geq1$,
\begin{eqnarray*}
\mathbb E \bigl[ f \bigl( \varepsilon^{1/\alpha} F(\zeta-\varepsilon t) \bigr)
\bigr] &\leq& \mathbb P (\zeta-\varepsilon t \geq\tau_{\varepsilon^{-1/\alpha}})
\\
&\leq&\mathbb P
\bigl((\zeta-\tau_{\varepsilon^{-1/\alpha}})^{-\beta/\alpha} \geq(\varepsilon
t)^{-\beta
/\alpha} \bigr)
\\
&\leq&\mathbb P \biggl( \sum_{i \geq1} F_i(
\tau_{\varepsilon^{-1/\alpha
}})^{\beta}\bigl(\zeta^{(i)}
\bigr)^{-\beta/\alpha} \geq(\varepsilon t)^{-\beta
/\alpha} \biggr)
\\
&\leq&\frac{\mathbb E [\zeta^{-\beta/\alpha} ] \mathbb
E [ \sum_{i \geq1} F_i(\tau_{\varepsilon^{-1/\alpha}})^{\beta
} ]}{(\varepsilon t)^{-\beta/\alpha}}
\\
&\leq&\frac{\mathbb E [\zeta^{-\beta/\alpha}
]}{t^{-\beta/\alpha}},
\end{eqnarray*}
by definition of $\tau_{\varepsilon^{-1/\alpha}}$ and the fact that
$\beta\geq1$. Taking $\beta$ larger if necessary so that $-\beta
/\alpha>1$ and recalling that $\mathbb E [\zeta^{-\beta/\alpha
} ]<\infty$, we obtain
\[
\mathbb E \bigl[ f \bigl( \varepsilon^{1/\alpha} F(\zeta-\varepsilon t) \bigr)
\bigr] \leq\min\bigl(1,Ct^{\beta/\alpha}\bigr)\qquad\forall t \geq0
\]
for some finite constant $C$, independently of $\varepsilon$. The result follows.
\end{pf}

\begin{pf*}{Proof of Theorem~\ref{teoinvariant}}
Consider the potential measure
\[
\lambda_{\varepsilon}(A):= \int_0^{\infty}
\mathbb P_{ \varepsilon
^{1/\alpha} \mathbf{1}} \bigl(F(t) \in A \bigr)\, \mathrm{d} t.
\]
Equivalently, $\lambda_{\varepsilon}(A)$ is the expected time spent in
$A$ by a fragmentation process started from $\varepsilon^{1/\alpha}
\mathbf{1}$.
Suppose now that $f\dvtx \mathcal S \rightarrow\mathbb R_+$ is continuous
and such that
$f(\mathbf s) \leq\mathbh1_{\{0<\llVert \mathbf s\rrVert _1 \leq c\}
}$ for some
$c>0$. By the self-similarity of the fragmentation process,
%
%e7.1 #&#
\begin{eqnarray}
\label{cvborne} && \int_{\mathcal S}f(\mathbf s )\lambda_{\varepsilon}(
\mathrm d \mathbf s)  =  \int_0^{\infty} \mathbb E
\bigl[ f \bigl(\varepsilon^{1/\alpha
} F(\varepsilon t) \bigr) \bigr]\, \mathrm{d} t
\nonumber\\[-8pt]\\[-8pt]\nonumber
&&\hspace*{49pt} \displaystyle\mathop{=}_{(\mathrm{if}~\varepsilon^{1/\alpha}>c)} \int
_0^{\infty} \mathbb E \bigl[f \bigl(
\varepsilon^{1/\alpha} F(\zeta- \varepsilon t) \bigr) \bigr]\, \mathrm{d} t
\to\int
_{\mathcal S}f(\mathbf s)\lambda(\mathrm d \mathbf s)
\end{eqnarray}
as $\varepsilon\to0$, by Lemma~\ref{lemcveps}.

From now on, fix $u>0$, $c>0$ and a continuous function $f\dvtx \mathcal S
\rightarrow\mathbb R_+$ such that
$f(\mathbf s) \leq\mathbh1_{\{0<\llVert \mathbf s\rrVert _1 \leq c\}
}$. Our goal
is to check, on the one hand, that
%
%e7.2 #&#
\begin{equation}
\label{cvpotential1} \int_{\mathcal S} \mathbb E_{\mathbf{s} } \bigl[ f
\bigl(F(u)\bigr) \bigr] \lambda_{\varepsilon}(\mathrm d \mathbf s)
\rightarrow
\int_{\mathcal
S}f(\mathbf s)\lambda(\mathrm d \mathbf s)
\end{equation}
and, on the other, that
%
%e7.3 #&#
\begin{equation}
\label{cvpotential2} \int_{\mathcal S} \mathbb E_{\mathbf{s}} \bigl[ f
\bigl(F(u)\bigr) \bigr] \lambda_{\varepsilon}(\mathrm d \mathbf s)
\rightarrow
\int_{\mathcal
S} \mathbb E_{\mathbf s} \bigl[ f\bigl(F(u)\bigr)
\bigr] \lambda(\mathrm d \mathbf s).
\end{equation}
Together, these will yield the invariance of $\lambda$.

We start with (\ref{cvpotential1}). By the definition of $\lambda
_{\varepsilon}$,
\begin{eqnarray*}
&& \int_{\mathcal S} \mathbb E_{\mathbf{s} } \bigl[ f\bigl(F(u)
\bigr) \bigr] \lambda_{\varepsilon}(\mathrm d \mathbf s)
\\
&&\qquad  = \int
_0^{\infty} \mathbb E \bigl[ f\bigl(
\varepsilon^{1/\alpha
}F\bigl(\varepsilon(u+t)\bigr)\bigr) \bigr] \,\mathrm d t
\\
&&\qquad = \int_0^{\infty} \mathbb E \bigl[ f\bigl(
\varepsilon^{1/\alpha
}F(\varepsilon t)\bigr) \bigr] \,\mathrm d t-\int
_0^u \mathbb E \bigl[ f\bigl(
\varepsilon^{1/\alpha}F(\varepsilon t)\bigr) \bigr] \,\mathrm d t.
\end{eqnarray*}
The first integral in the last line converges to $\int_{\mathcal
S}f(\mathbf s)\lambda(\mathrm d \mathbf s)$, by (\ref{cvborne}). The
second converges to 0, since $\mathbb E [ f(\varepsilon^{1/\alpha
}F(\varepsilon t)) ] \rightarrow0$ for all $t>0$ [as $\varepsilon
^{1/\alpha}\llVert F(\varepsilon t)\rrVert _1>c$ for $\varepsilon$ small
enough, a.s.].
The convergence in (\ref{cvpotential1}) follows.

To get (\ref{cvpotential2}), set $g(\mathbf s)=\mathbb E_{ \mathbf{s}
} [ f(F(u)) ] $. The function $g$ is continuous, bounded
and $\mathbb R_+$-valued, but is not supported by a set of the form
$0<\llVert \mathbf s\rrVert _1\leq c'$ for some~$c'$, so we cannot
conclude the
desired result directly from the convergence of $\lambda_{\varepsilon}$
to $\lambda$. Note that, for all $c'>0$,
\begin{eqnarray*}
&& \int_{\mathcal S} g(\mathbf s) \mathbh{1}_{ \{ \llVert \mathbf s\rrVert
_1>c' \} }
\lambda_{\varepsilon} (\mathrm d \mathbf s)
\\
&&\qquad = \int_0^{\infty}
\mathbb{E} \bigl[f \bigl( \varepsilon^{1/\alpha} F\bigl(\varepsilon(u+t)\bigr)
\bigr) \mathbh{1}_{ \{ \llVert \varepsilon^{1/\alpha
}F(\varepsilon t)\rrVert _1 > c' \} } \bigr] \,\mathrm d t
\\
&&\qquad \leq \int_0^{\infty} \mathbb P \bigl(\bigl\llVert
\varepsilon^{1/\alpha
}F(\varepsilon t)\bigr\rrVert_1 >c',
0<\bigl\llVert\varepsilon^{1/\alpha}F\bigl(\varepsilon(t+u)\bigr)\bigr\rrVert
_1 \leq c \bigr) \,\mathrm dt
\\
&&\qquad \leq\int_0^{\infty} \mathbb P \bigl( \bigl\llVert
\varepsilon^{1/\alpha}F(\zeta-\varepsilon t-\varepsilon u) \bigr\rrVert
_1>c',0< \bigl\llVert\varepsilon^{1/\alpha}F(\zeta-
\varepsilon t) \bigr\rrVert_1\leq c \bigr) \,\mathrm dt.
\end{eqnarray*}

Using the dominated convergence theorem (and the same argument as in
the proof of Lemma~\ref{lemcveps}), we see that the right-hand side
converges to $\int_0^{\infty} \mathbb P (\llVert C_{\infty}
(t+u)\rrVert
_1>c', \llVert C_{\infty}(t) \rrVert _1\leq c ) \,\mathrm dt$, which
is finite.
Hence, for all $\eta>0$ and then all $c'>0$ large enough, say $c' \geq
c'_{\eta}$,
\[
\limsup_{\varepsilon\rightarrow0} \int_{\mathcal S} g(\mathbf{s} )
\mathbh{1}_{ \{ \llVert \mathbf s\rrVert _1>c' \} } \lambda_{\varepsilon
} (\mathrm d \mathbf s)\leq\eta.
\]
Now,
\begin{eqnarray*}
&& \int_{\mathcal S} g(\mathbf s) \mathbh{1}_{ \{ \llVert \mathbf s\rrVert
_1>c' \} }\lambda_{\varepsilon} (\mathrm d \mathbf s)
\\
&&\qquad =\int_0^{\infty}
\mathbb E \bigl[ g \bigl(\varepsilon^{1/\alpha
}F(\varepsilon t) \bigr)
\mathbh{1}_{ \{ \llVert \varepsilon^{1/\alpha
}F(\varepsilon t)\rrVert _1 > c' \} } \bigr] \,\mathrm dt
\\
&&\qquad = \int_0^{\infty} \mathbb E \bigl[ g \bigl(
\varepsilon^{1/\alpha
}F(\zeta-\varepsilon t) \bigr) \mathbh{1}_{ \{ \llVert \varepsilon
^{1/\alpha}F(\zeta-\varepsilon t)\rrVert _1 > c' \} }
\mathbh{1}_{
\{ \zeta\ge\varepsilon t \} } \bigr] \,\mathrm dt.
\end{eqnarray*}
Since the function $\mathbf{s} \mapsto g(\mathbf{s}) \mathbh{1}_{
\{ \llVert \mathbf{s}\rrVert _1 > c' \} }$ is lower semi-continuous,
by the Portmanteau theorem,
\begin{eqnarray*}
&& \liminf_{\varepsilon\to0} \mathbb{E} \bigl[g \bigl(\varepsilon^{1/\alpha
}F(
\zeta-\varepsilon t) \bigr) \mathbh{1}_{ \{ \llVert \varepsilon
^{1/\alpha}F(\zeta-\varepsilon t)\rrVert _1 > c' \} } \mathbh{1}_{
\{ \zeta\ge\varepsilon t \} }
\bigr]
\\
&&\qquad \ge\mathbb E \bigl[ g \bigl(C_{\infty}(t)\bigr)\mathbh{1}_{ \{
\llVert
C_{\infty}(t)\rrVert _1 > c' \} }
\bigr].
\end{eqnarray*}
Hence, by Fatou's lemma,
\[
\int_0^{\infty} g(\mathbf s) \mathbh{1}_{ \{ \llVert \mathbf s\rrVert
_1>c' \} }
\lambda(\mathrm d \mathbf s) \leq\liminf_{\varepsilon
\rightarrow0} \int
_{\mathcal S} g(\mathbf{s} ) \mathbh{1}_{
\{ \llVert \mathbf s\rrVert _1>c' \} }
\lambda_{\varepsilon}(\mathrm d \mathbf s)\leq\eta%
\]
for all $c'\geq c'_{\eta}$.
Finally, fix $\eta>0$ and then $c' \geq c'_{\eta}$. Consider then
$c'' \in(c',\infty)$, and let $h\dvtx \mathcal S \rightarrow\mathbb
[0,1]$ be a continuous function such that $h(\mathbf{s}) = 1$ when $
\llVert
\mathbf s\rrVert _1 \leq c'$ and $h(\mathbf{s}) = 0$ when $\llVert
\mathbf s\rrVert _1
\geq c''$. Then
\begin{eqnarray*}
&& \biggl\llvert\int_{\mathcal S} g(\mathbf s) (
\lambda_{\varepsilon}- \lambda) (\mathrm d \mathbf s)\biggr\rrvert
\\
&&\qquad \leq\biggl\llvert\int_{\mathcal S} g(\mathbf s)h(\mathbf s) (
\lambda_{\varepsilon}- \lambda) (\mathrm d \mathbf s)\biggr\rrvert+\biggl
\llvert\int_{\mathcal S} g(\mathbf s) \bigl(1-h(\mathbf s)\bigr)
\lambda_{\varepsilon}(\mathrm d \mathbf s)\biggr\rrvert
\\
&&\quad\qquad{} +\biggl\llvert
\int
_{\mathcal S} g(\mathbf s) \bigl(1-h(\mathbf s)\bigr) \lambda(\mathrm
d \mathbf s)\biggr\rrvert.
\end{eqnarray*}
We have chosen $c'$ and $h$ so that the second and third terms are each
smaller than $\eta$ for small enough $\varepsilon$. By (\ref{cvborne}),
the first term converges to 0 as $\varepsilon\to0$. The convergence in
(\ref{cvpotential2}) follows.
\end{pf*}

%%%%%%%%%%%%%%%%%%%%%%%%%%%%%%
%s8 #&#
\section{Discussion of geometric fragmentations} \label{geofrag}
%%%%%%%%%%%%%%%%%%%%%%%%%%%%%%

In this section, we consider geometric fragmentations; that is, we
assume that the set of $r \in(0,1)$ such that
%
%e8.1 #&#
\begin{equation}
\label{eqr} \nu\bigl(s_i \in r^{\mathbb N}\cup\{0\}, i \ge1
\bigr)=1
\end{equation}
is nonempty, and we let $r_{\mathrm{min}}$ denote its unique minimal element.
It is easy to see that $r_{\mathrm{min}}$ exists and is characterized
by the fact that $\nu$-a.e. $s_i=r_{\mathrm{min}}^{n_i}, \forall i $
where the nonzero integers $n_i$ have 1 as highest common factor.
Moreover, for every $r \in(0,1)$ satisfying (\ref{eqr}), there is a
$q \in\mathbb N$ such that $r_{\mathrm{min}}=r^q$.

This case has some interesting connections to other parts of the
probability literature, which we will
briefly describe below.
We will then see that $\varepsilon^{1/\alpha}F(\zeta-\varepsilon)$ cannot
converge in distribution in this case. However, it does converge along
appropriate subsequences.
Finally, we will restrict attention to the simple case of $k$-ary
fragmentations, when each fragmentation of a block produces $k$ blocks
with identical masses, and describe all possible limit distributions of
the rescaled last fragment $\varepsilon^{1/\alpha}F_*(\zeta-\varepsilon)$
in these simple $k$-ary fragmentations.

%s8.1 #&#
\subsection{Related models}

Specialize, for the moment, to the case where the fragmentation has
dislocation measure
\[
\nu(\mathrm d \mathbf s) = \delta_{ (1/k, 1/k,
\ldots,
1/k,0,\ldots)}(\mathrm d \mathbf s), \qquad
\mathbf s \in\Sfl.
\]
This fragmentation process has been studied in various different guises
in the probability literature.

In \cite{Athreya}, Athreya considers a model which he
calls the \emph{discounted branching random walk}. Start with a
single particle situated at a distance to the right of the origin
which is distributed as $\Exp(1)$. At each epoch, every particle
present gives birth to two particles. At epoch $n$, these new
particles have a displacement rightwards from the parent with
distribution $\Exp(2^{-n \alpha})$, independently for different
particles. It is easy to see that the positions of the $2^n$ particles
at generation $n$ correspond to
the times at which the blocks of size $2^{-n}$ appear in the simple
binary fragmentation (when $k=2$). Athreya concerns himself
particularly with a recursive
equation for the distribution of the right-hand end of the support of
the particle distribution at time $\infty$. This, of course, has the
same distribution as $\zeta$, and the recursive distributional
equation is $\zeta= T_1 + 2^{n \alpha} \max\{\zeta^{(1)}, \zeta
^{(2)}\}$ in our notation. This equation and others like it are
discussed in more detail in Aldous and Bandyopadhyay~\cite
{AldousBandy}. The convergence of the last fragment in Theorem~\ref
{Theorem1} (which is valid for geometric fragmentations) entails that
the distance between the ancestor of generation
$n$ of the winning particle and the winning particle itself, rescaled
by $2^{-n \alpha}$ converges in distribution as $n \to\infty$. Of
course, this construction is easily extended to the case where each
individual gives birth to $k$ offspring.

Barlow, Pemantle and Perkins~\cite{BPP} consider a model of
randomly-growing $k$-ary trees which has also been studied, in various
versions, in \cite{AldousShields,DeanMajumdar,Devroye,Pittel}.
Suppose we grow the complete $k$-ary tree as follows. [For
definiteness, label vertices in the tree by $k$-ary strings, so
that the root is $\varnothing$, its neighbors are $0,1,\ldots,k-1$ and,
in general, the descendants of a vertex labeled $x$ are $x0, x1,
\ldots,x(k-1)$.] We start with the empty tree and wait an $\Exp(1)$
amount of time; then the root gets filled in. Let $A(0) =
\{\varnothing\}$. In general, let $A(t)$ be the set of vertices in the
$k$-ary tree which have not yet been filled in themselves, but whose
parents in the tree have been filled in. A vertex in $A(t)$ at height
$n$ (where
the root has height 0) becomes filled in at a rate $k^{-\alpha n}$.
The vertices in $A(t)$ correspond exactly to blocks in our
fragmentation at time $t$. In particular, a vertex at height $n$
corresponds to a
block of size $k^{-n}$.
This model can be thought of as a sort of \emph{first-passage
percolation} or as \emph{diffusion-limited aggregation} on a tree. In
particular, Barlow, Pemantle and Perkins study the structure of the
cluster at the first time
that it hits the boundary of the tree. This corresponds to the time at
which mass first disappears in the fragmentation. They show that at
that time the cluster consists of a unique
infinite backbone with small finite trees hanging off it. We are
instead interested in what happens near the time at which the \emph
{last} point on the boundary of the tree is reached. Theorem~\ref
{Theorem1} tells us that the time taken to reach this last point on the
boundary from its ancestor in generation $n$, suitably rescaled, has a
limit in distribution as $n \to\infty$.

We now turn to a more general context and prove some results which
apply to these special cases.

%s8.2 #&#
\subsection{Absence of limit in distribution}
We return to the general case of a geometric fragmentation $F$,
assuming solely that $\int_{\mathcal S_1} s_1^{-1} \nu(\mathrm d
\mathbf s)$ is finite.
Recall that
$T_n$ is the $n$th jump time of the last fragment process $F_*$. From
Theorem~\ref{Theorem1}, we know that
\[
Z_n^{1/\alpha}=(\zeta-T_n)^{1/\alpha}F_*(T_n)
\]
converges in distribution to a law which is fully supported by
$(0,\infty)$. However, we do not have convergence in distribution of
the rescaled sequence $\varepsilon^{1/\alpha}F(\zeta-\varepsilon)$ as
$\varepsilon\rightarrow0$.

%pr8.1 #&#
\begin{prop}
\label{propgeo}
In the geometric cases,
$\varepsilon^{1/\alpha}F(\zeta-\varepsilon)$ and $\varepsilon^{1/\alpha
}F_*(\zeta-\varepsilon)$ do not converge in distribution as $\varepsilon
\rightarrow0$. However, for each $x \in[0,1)$, the sequence
$r_{\mathrm{min}}^{-n-x}F_*(\zeta-r_{\mathrm{min}}^{-\alpha(n+x)})$
has\vspace*{1pt} a nonzero limit in distribution as $n \rightarrow\infty$, which
depends on $x$.
\end{prop}

In the next section, we specify this limit and its dependence on $x$
for the simple $k$-ary fragmentations.

\begin{pf*}{Proof of Proposition \ref{propgeo}}
Suppose (for a contradiction) that\break $\varepsilon^{1/\alpha
}F(\zeta-\varepsilon)$ converges in distribution in $\mathcal S$. Then
$\varepsilon^{1/\alpha}F_1(\zeta-\varepsilon)$ has a limit in
distribution, say $L \in[0,\infty)$. Consider the sequence $\varepsilon
_n=ar_{\mathrm{min}}^{-\alpha n}$, $n \geq1$, where $a \in(0,\infty
)$ is fixed. Then the random variables $\varepsilon_n^{1/\alpha
}F_1(\zeta-\varepsilon_n)$ almost surely all belong to the set
$a^{1/\alpha}r_{\mathrm{min}}^{\mathbb Z}$, and so $L \in a^{1/\alpha
}r_{\mathrm{min}}^{\mathbb Z}\cup\{0 \}$ a.s. But this
assertion holds for all $a \in(0,\infty)$, hence $L=0$ a.s. In
particular, this implies that $\varepsilon^{1/\alpha}F_*(\zeta-\varepsilon
)$ converges in distribution to 0. Similarly, supposing first that
$\varepsilon^{1/\alpha}F_*(\zeta-\varepsilon)$ has a limit in
distribution, we conclude that this limit is necessarily 0.

But a zero limit is not possible, because $r_{\mathrm
{min}}^{-n}F_*(\zeta-r_{\mathrm{min}}^{-\alpha n})$ has a nonzero
limit in distribution as $n \rightarrow\infty$, provided that $\int
_{\mathcal S_1} s_1^{-1} \nu(\mathrm d \mathbf s)<\infty$. To see
this, we use Corollary~2.2(b) of Alsmeyer \cite{Alsmeyer2}, on Markov
renewal theory in the geometric cases. Given this corollary, it is
possible to check that the rescaled sequence $r_{\mathrm
{min}}^{-n}F_*(\zeta-r_{\mathrm{min}}^{-\alpha n})$ has a nontrivial
limit in distribution as $n \rightarrow\infty$, in exactly the same
way as we proved the one-dimensional convergence in Section~\ref
{Onedim}. Using arguments from Section~\ref{Wholepro} giving an
expression for $N_{\varepsilon t}$ in terms of $N_{\varepsilon}$, it is then
easy to deduce the convergence in distribution of $r_{\mathrm
{min}}^{-n-x}F_*(\zeta-r_{\mathrm{min}}^{-\alpha n+\alpha x})$ to a
nontrivial limit. We leave these extensions to the reader.
\end{pf*}

%re8.2 #&#
\begin{rem}
This result then certainly leads to the convergence of $r_{\mathrm
{min}}^{-n-x}F(\zeta-r_{\mathrm{min}}^{-\alpha(n+x)})$ to a
nontrivial limit and more generally of the whole process $r_{\mathrm
{min}}^{-n-x} ( F((\zeta-r_{\mathrm{min}}^{-\alpha(n+x)}t)-),t
\geq0 )$, at\vspace*{2pt} least when $\int_{\mathcal S_1} s_1^{-1-\rho} \nu
(\mathrm d \mathbf s)<\infty$ for some $\rho>0$. In order to see
this, one should mimic the proofs of Sections~\ref{MarkovRW} and~\ref
{fullfrag}. However, for ease and brevity of exposition, we omit this
part and leave it to the motivated reader. We emphasize that the limit
process depends on $x$ and cannot be self-similar. Moreover, the proofs
of Lemma~\ref{lemcveps} and Theorem~\ref{teoinvariant} in
Section~\ref{secinvariant} are still valid when replacing $\varepsilon$
by $\varepsilon_n(x)=r_{\mathrm{min}}^{-\alpha(n+x)}$ and letting $n
\rightarrow0$. Hence, we may deduce the existence of invariant
measures for these geometric fragmentations. Note that the invariant
measure constructed from the sequence $(\varepsilon_n(x))_{n \geq0}$ is
supported by elements $\mathbf s$ of $\mathcal S$ such that $s_i \in
r_{\mathrm{min}}^{-x+\mathbb Z}$ for all $i$. We have, therefore, a
continuum set of distinct invariant measures, indexed by $x \in[0,1)$.
\end{rem}

%s8.3 #&#
\subsection{Simple $k$-ary fragmentations}

From now on, we assume that the fragmentation has dislocation measure
\[
\nu(\mathrm d \mathbf s) = \delta_{ (1/k, 1/k,
\ldots,
1/k,0,\ldots)}(\mathrm d \mathbf s), \qquad
\mathbf s \in\Sfl.
\]
By adapting the method of proof of Theorems 5.1 and 5.2 of \cite{BPP},
we can obtain a stronger version of
Theorem~\ref{Theorem1}. Note that here $T_n= \inf\{t \geq0\dvtx  F_*(t) =
k^{-n}\}$ and $Z_n = k^{-n \alpha}(\zeta- T_n)$.

%pr8.3 #&#
\begin{prop} \label{propstochincr}
The sequence $(Z_n)_{n\geq0}$ is stochastically increasing.
As a consequence,
\[
Z_n \stackrel{\mathit{law}} {\rightarrow}Z_{\infty}
\]
as $n \to\infty$, where $Z_{\infty} \sim\pi_{\mathrm{stat}}$ and
$Z_{\infty} \geq_{\mathrm{st}} \zeta$.
\end{prop}

\begin{pf}
We\vspace*{2pt} argue by induction, using the notation of Section~\ref{cvZn}.
Recall that $Z_0 = \zeta$ and that $Z_1 = \zeta^{(I)} = \max_{1 \leq i
\leq k} \zeta^{(i)}$.
%Let $F_0(t) = \Prob{Z_0 \leq t}$ and $F_1(t) = \Prob{Z_1 \leq t}$.
%Then
%
%\[
%F_1(t) = F_0(t)^k \leq F_0(t).
%\]
%
It follows that $Z_0 \leq_{\mathrm{st}}Z_1$.
Let next
\[
p(t,x) = \mathbb P \bigl(\zeta^{(I)} \geq t \mid\zeta= x \bigr)
\]
in the sense of a regular conditional probability. Since $(Z_n)_{n
\geq0}$ is a Markov chain,
\[
\mathbb{P} (Z_{n+1} \geq t ) = \mathbb{E} \bigl[p(t,Z_n)
\bigr].
\]
Suppose for the moment that, for fixed $t$, $p(t,x)$ is increasing in
$x$. Our induction hypothesis is that $Z_{n-1} \leq_{\mathrm
{st}}Z_n$. Then
\[
\mathbb{P} (Z_{n+1} \geq t ) = \mathbb{E} \bigl[p(t,Z_n)
\bigr] \geq\mathbb{E} \bigl[p(t,Z_{n-1}) \bigr] = \mathbb{P}
(Z_n \geq t ).
\]
So it remains to show that $p(t,x)$ is increasing in $x$.

We have $\zeta= T_1 + k^{\alpha} \max_{1 \leq i \leq k}
\zeta^{(i)}=T_1+k^{\alpha}\zeta^{(I)}$ with $T_1$ independent of
$\zeta^{(I)}$. From this, it is easy to see that $(\zeta,\zeta
^{(I)})$ has a density which may be written as
\[
(x,y) \in\mathbb R_+^2 \mapsto f_{\zeta^{(I)}}(y)e^{k^{\alpha
}y-x}
\mathbh{1}_{ \{ x\geq k^{\alpha}y \} }. %
\]
Then for $t \leq k^{-\alpha}x$,
\[
p(t,x)=\frac{\int_t^{k^{-\alpha}x} f_{\zeta^{(I)}}(y)e^{k^{\alpha
}y-x}\, \mathrm dy}{\int_0^{k^{-\alpha}x} f_{\zeta
^{(I)}}(y)e^{k^{\alpha}y-x}\, \mathrm dy}=1-\frac{\int_0^t f_{\zeta
^{(I)}}(y)e^{k^{\alpha}y}\, \mathrm dy}{\int_0^{k^{-\alpha}x} f_{\zeta
^{(I)}}(y)e^{k^{\alpha}y}\, \mathrm dy} %
\]
and so $p(t,x)$ is, indeed, increasing in $x$.
\end{pf}

Now, for $t \ge0$, let
\[
x(t) = \frac{1}{\alpha} \log_k t - \biggl[ \frac{1}{\alpha} \log
_k t \biggr].
\]
We will now specify the asymptotics of the last fragment $F_*(\zeta
-\varepsilon_n)$, according to the behavior of the sequence $(\varepsilon
_n)$ under the action of the function $x$.

%pr8.4 #&#
\begin{prop} \label{teoconv} Let $(\varepsilon_n)_{n \ge0}$ be any
sequence of times converging to 0 such that $x(\varepsilon_n) \to x$ for
some fixed $x \in[0,1)$. Then we have as $n \rightarrow\infty$
\[
\varepsilon_{n}^{1/\alpha} F_{*}(\zeta-
\varepsilon_{n}) \stackrel{\mathrm{law}} {\rightarrow}k^{x-N(x)},
\]
where $N(x)= \sup\{n \in\Z\dvtx  Z^{\mathrm{stat}}_{n} \geq
k^{(x-n)\alpha} \}$.
\end{prop}

We note that $N(x) > -\infty$ almost surely, a statement which we will
justify during the course of the proof. It is also the case that $N(x)
< \infty$. As an example of an application of this proposition, for
all $x \in[0,1)$, we have
\[
k^{x+n}F_{*}\bigl(\zeta- k^{\alpha(x+n)}\bigr) \stackrel{
\mathit{law}} {\rightarrow}k^{x-N(x)}\qquad\mbox{as }n \to\infty.
\]

\begin{pf*}{Proof of Proposition \ref{teoconv}}
For any $\varepsilon\ge0$, let $N_{\varepsilon} = \sup\{n \ge0\dvtx  \zeta-
\varepsilon\ge T_n\} = \sup\{n \ge0\dvtx  \varepsilon\le\zeta- T_n\}$. Then,
\[
\varepsilon^{1/\alpha} F_*(\zeta- \varepsilon) = \varepsilon^{1/\alpha}
F_*(T_{N_\varepsilon}) = \varepsilon^{1/\alpha} k^{-N_\varepsilon}.
\]
Using $Z_n = k^{-n\alpha} (\zeta- T_n)$, we have
\[
N_\varepsilon= \sup\bigl\{n \ge0\dvtx  Z_n \ge k^{-n \alpha}
\varepsilon\bigr\}.
\]
Write $m(\varepsilon) = [(\log_k \varepsilon)/\alpha]$
so that
$
m(\varepsilon) + x(\varepsilon)=(\log_k \varepsilon)/\alpha$.
Then
\begin{eqnarray*}
N_\varepsilon- m(\varepsilon) & =& \sup\bigl\{n \ge-m(\varepsilon)\dvtx
Z_{m(\varepsilon) + n} \ge k^{-(m(\varepsilon) + n) \alpha} \varepsilon\bigr\}
\\
& =& \sup\bigl\{n \ge-m(\varepsilon)\dvtx  Z_{m(\varepsilon) + n} \ge k^{-n
\alpha} \cdot
k^{\alpha x(\varepsilon) } \bigr\}.
\end{eqnarray*}
Now take $\varepsilon= \varepsilon_n$ so that $\varepsilon_n\rightarrow0$
and $x(\varepsilon_n) \to x$ as $n \to\infty$. Then for all $p \in
\mathbb Z$ and all $n$ such that $p>-m(\varepsilon_n)$,
\[
\mathbb P \bigl(N_{\varepsilon_n}-m(\varepsilon_n)\geq p \bigr)=\mathbb
P \bigl(Z_{m(\varepsilon_n) + p} \ge k^{-p \alpha} \cdot k^{\alpha
x(\varepsilon_n)} \bigr)
\]
since the sequence $(Z_nk^{n\alpha})$ is nonincreasing in $n$ a.s.
(indeed, $Z_n k^{n \alpha} = \zeta-T_n$). Similarly,
\[
\mathbb P \bigl(N(x) \geq p \bigr)= \mathbb P \bigl(Z_p^{\mathrm
{stat}}
\geq k^{-p\alpha}k^{\alpha x} \bigr)=\pi_{\mathrm
{stat}} \bigl(\bigl[k^{-p\alpha}k^{\alpha x},\infty\bigr) \bigr). %
\]
Then, since $x(\varepsilon_n)\rightarrow x$, $Z_{m(\varepsilon_n) + p}$
converges in law to $\pi_{\mathrm{stat}}$ as $n\rightarrow\infty$
and as $\pi_{\mathrm{stat}}$ is nonatomic, we get that
\[
\mathbb P \bigl(N_{\varepsilon_n}-m(\varepsilon_n)\geq p \bigr)
\rightarrow\mathbb P \bigl(N(x) \geq p \bigr), %
\]
for all $p \in\mathbb Z$. In other words,
$
N_{\varepsilon_n} - m(\varepsilon_n)$ converges in law to $N(x)$
as $n \to\infty$. So
$
N_{\varepsilon_n} - (\log_k \varepsilon_n)/\alpha$ converges in law to
$N(x) - x$,
which entails that
\[
\varepsilon_n^{1/\alpha} k^{-N_{\varepsilon_n}} \stackrel{\mathrm{law}}
{\rightarrow}k^{x - N(x)},
\]
as $n \to\infty$, as required.
\end{pf*}

\setcounter{teo}{0}
\setcounter{equation}{0}
\begin{appendix}
%s9 #&#
\section*{Appendix} \label{appendix}

%s9.1 #&#
\subsection{Convergence criteria}

In this section, we record various technical lemmas concerning criteria
for convergence in $(\mathcal{S},d)$ and in the Skorohod topology on
c\`adl\`ag processes taking values in $(\mathcal{S}, d)$. The proofs
of the first two lemmas are straightforward, and so we omit them.

%le9.1 #&#
\begin{lem}
\label{lemdecroissance}
Let $(\mathbf s^{(n)}, n \geq1)$ be a sequence of nonnegative
elements of $\ell_1$ converging to $\mathbf s^{(\infty)} \in\ell_1$
for the $\ell_1$-topology. For every integer $n \in\mathbb N \cup\{
\infty\}$, let $\mathbf s^{(n),\downarrow}$ denote the decreasing
rearrangement of the terms of $\mathbf s^{(n)}$. Then $\mathbf
s^{(n),\downarrow} \rightarrow\mathbf s^{(\infty),\downarrow}$ in
$(\mathcal S,d)$.
\end{lem}

%le9.2 #&#
\begin{lem}
\label{lemcontinuity}
Let $n \in\mathbb N$. The two following functions are continuous:
\begin{longlist}[(ii)]
\item[(i)] $(\mathbf s^{(1)}, \ldots,\mathbf s^{(n)}) \in\mathcal
S^n \mapsto\{s_j^{(i)}, 1 \le i \le n, j \geq1 \}
^{\downarrow} \in\mathcal S$,
where
$\mathcal S^n$ is endowed with the product topology;

\item[(ii)] $(x,\mathbf s) \in\mathbb R_+ \times\mathcal S \mapsto
\{xs_j, j \geq1 \} \in\mathcal S$.
\end{longlist}
\end{lem}

We next recall a classical
result on Skorokhod convergence (see Proposition~3.6.5 of Ethier and
Kurtz~\cite{EK}) which we will use repeatedly.

%pr9.3 #&#
\begin{prop}
\label{PropEK}
Consider a metric space $(E,d_{E})$, and let $f_{n},f$ be c%
\`{a}dl\`{a}g paths with values in $E$. Then $f_{n}\rightarrow f$ with
respect to the Skorokhod topology if and only if the three following
assertions are satisfied for all sequences $t_{n}\rightarrow t$, $%
t_{n},t\geq0$:

\begin{longlist}[(a)]
\item[(a)]
$\min
(d_{E}(f_{n}(t_{n}),f(t)),d_{E}(f_{n}(t_{n}),f(t-)))\rightarrow0$;

\item[(b)]  $d_{E}(f_{n}(t_{n}),f(t))\rightarrow0$ $\Rightarrow$ $
d_{E}(f_{n}(s_{n}),f(t))\rightarrow0$ for all sequences
$s_{n}\rightarrow
t$, $s_{n}\geq t_{n}$;

\item[(c)] $d_{E}(f_{n}(t_{n}),f(t-))\rightarrow0$ $\Rightarrow$
$%
d_{E}(f_{n}(s_{n}),f(t-))\rightarrow0$ for all sequences
$s_{n}\rightarrow
t$, $s_{n}\leq t_{n}$.
\end{longlist}
\end{prop}

Of course, if $t$ is not a jump time of $f$, then (a), (b), (c) are
equivalent to $d_{E}(f_{n}(t_{n}),f(t))\rightarrow0$.

We now establish three lemmas on Skorokhod convergence, which are used
in the main body of the paper.

%le9.4 #&#
\begin{lem}
\label{Skocv}
Consider $(c_n)_{n \in\mathbb Z_+ \cup\{\infty\}}$, a sequence of
real-valued nondecreasing piecewise constant c\`adl\`ag functions
defined on $\mathbb R_+$ by $c_n(0)=0$ and, for $t>0$,
\[
c_n(t)=b_n(k)\qquad\mbox{if } r_n(k) > t \ge
r_n(k+1), %
\]
where $(r_n(k))_{k \in\mathbb Z}$ is strictly decreasing in $k$ and
such that $r_n(k)\rightarrow0$ as $k \rightarrow\infty$ and
$r_n(k)\rightarrow\infty$ as $k \rightarrow-\infty$. Suppose that
for all $k\in\mathbb Z$, $r_n(k)\rightarrow r_{\infty}(k)$ and
$b_n(k)\rightarrow b_{\infty}(k)$ as $n \rightarrow\infty$. Then
$c_n \rightarrow c_{\infty}$ for the Skorokhod topology on the set of
real-valued c\`adl\`ag functions on $\mathbb R_+$.
\end{lem}

\begin{pf}
This is nearly obvious from the definition of the Skorokhod topology.
To prove it carefully, we use Proposition~\ref{PropEK}. It is easy to
see that for a fixed $t>0$ and all sequences $t_n\rightarrow t$,
conditions (a), (b) and (c) of this proposition are satisfied for the
sequence $(c_n)_{n \in\mathbb Z_+}$, with $c_{\infty}$ at the limit.
It remains to check them for $t=0$, which consists then in checking
that $c_n(t_n)\rightarrow c_{\infty}(0)=0$. This is immediate, using
monotonicity. Indeed, let $\varepsilon>0$; for large $n$, $t_n \leq
\varepsilon$, and so $c_n(t_n)\leq c_n(\varepsilon)$. The sequence
$(c_n(\varepsilon))$ might not converge, but clearly $\limsup_n
c_n(\varepsilon) \leq c_{\infty}(\varepsilon)$. Since $c_{\infty}$ is
right-continuous, we get, letting $\varepsilon\rightarrow0$, that
$\limsup_n c_n(\varepsilon)=0$.
\end{pf}

The next lemma concerns the time-reversed conditioned fragmentation
process $\bar{F}^{(x)}$ introduced in Section~\ref{Spine}.

%le9.5 #&#
\begin{lem}
\label{cvSkobarF}
Let $(a_n), (b_n), (c_n), a_{\infty},b_{\infty}, c_{\infty}$ be
nonnegative numbers such that $a_n \rightarrow a_{\infty}$, $b_n
\rightarrow b_{\infty}$ and $c_n \rightarrow c_{\infty}$. Then
\[
\bigl(c_n\bar F^{(a_n)}(b_nt+),t \geq0 \bigr)
\stackrel{\mathit{law}}\rightarrow\bigl(c_{\infty}\bar
F^{(a_{\infty})}(b_{\infty
}t+),t \geq0 \bigr) %
\]
in sense of the Skorokhod topology on c\`adl\`ag processes taking
values in $(\mathcal{S},d)$.
\end{lem}

\begin{pf}
Let $F$ be a fragmentation process and, for all $n \in\mathbb N \cup\{
\infty\}$, let $G^{(n)}$ be defined by
\[
G^{(n)}(t)=\cases{ c_nF(a_n-b_nt),
&\quad if $0 \leq b_nt \leq a_n$,
\cr
\mathbf0, &\quad if
$b_nt>a_n$.}
\]
Then observe that for all $u \geq0$:
\begin{itemize}
\item if $(u_n)$ is a sequence converging to $u$, with $u_n>u$ for all
$n$, then $F(u_n-) \to F(u)$;
\item if $(u_n)$ is a sequence converging to $u$, with $u_n\leq u$ for
all $n$, then $F(u_n-) \to F(u-)$.
\end{itemize}
We can deduce from this [together with Lemma~\ref{lemcontinuity}(ii)] that for all $t \geq0$:
\begin{itemize}
\item$G^{(n)}(t_n+) \rightarrow G^{(\infty)}(t)$ when $t_n\rightarrow
t$ and $a_n-b_nt_n>a_{\infty}-b_{\infty}t$ for all $n$ large enough;

\item$G^{(n)}(t_n+) \rightarrow G^{(\infty)}(t+)$ when $t_n
\rightarrow t$ and $a_n-b_nt_n \leq a_{\infty}-b_{\infty}t$ for all
$n$ large enough.
\end{itemize}
From these observations and Proposition~\ref{PropEK}, we get that
$(G^{(n)}(t+),t\geq0)$ converges to $(G^{({\infty})}(t+),t\geq0)$ as
$n \rightarrow\infty$ for the Skorokhod topology on $\mathcal S$,
almost surely. Since the extinction time $\zeta$ of $F$ has a
continuous cumulative distribution function, we also have $ \mathbh
{1}_{ \{ \zeta<a_n \} } \rightarrow\mathbh{1}_{
\{ \zeta<a_{\infty} \} }$, almost surely.
Hence, for all bounded continuous test functions $f\dvtx \mathcal S
\rightarrow\mathbb R$,
\begin{eqnarray*}
&& \mathbb E \bigl[f \bigl( \bigl(c_n\bar F^{(a_n)}(b_nt+),t
\geq0\bigr) \bigr) \bigr]= \frac{ \mathbb E [f ( (G^{(n)}(t+),t \geq0)
) \mathbh{1}_{ \{ \zeta<a_n \} } ]}{\mathbb
P(\zeta<a_n)}
\\
&&\hspace*{129pt} \displaystyle\mathop\rightarrow_{n\rightarrow\infty}
\frac{ \mathbb E
[f ( (G^{(\infty)}(t+),t \geq0) ) \mathbh{1}_{ \{
\zeta<a_{\infty} \} } ]}{\mathbb P(\zeta<a_{\infty})}
\\
&&\hspace*{134pt} =\mathbb E \bigl[f \bigl( \bigl(c_{\infty}\bar F^{(a_{\infty
})}(b_{\infty}t+),t
\geq0\bigr) \bigr) \bigr].
\end{eqnarray*}
\upqed
\end{pf}

Finally, the following lemma is an easy consequence of Proposition~\ref
{PropEK} and the continuity property for the decreasing rearrangement
of a finite number of elements of $\mathcal S$ [Lemma~\ref{lemcontinuity}(i)]. Its proof is omitted.

%le9.6 #&#
\begin{lem}
\label{cvSkorearr}
\textup{(i)} Consider c\`adl\`ag functions $u^{(n)}, u\dvtx [0,\infty)
\rightarrow\mathcal S$ such that $u^{(n)} \rightarrow u$ as $n
\rightarrow\infty$ with
respect to the Skorokhod topology. Let $(t_n)$ be a sequence of
nonnegative numbers converging to $t \geq0$, and consider another
family of c\`adl\`ag functions $v^{(n)}, v\dvtx [0,\infty) \rightarrow
\mathcal S$ such that $v^{(n)}(t_n) \rightarrow v(t)$. For $s \geq0$,
set $f_n(s)=\{u_j^{(n)}(s),v_k^{(n)}(s), j \geq1,k\geq1\}^{\downarrow
}$ and similarly $f(s)=\{u_j(s),v_k(s), j \geq1,k\geq1\}^{\downarrow
}$. Then the functions $f_n,f$ are c\`adl\`ag and satisfy assertions
\textup{(a)}, \textup{(b)} and~\textup{(c)} of Proposition~\ref{PropEK} for the sequence $(t_n)$.

\textup{(ii)} Let $u^{(n,i)}, u^{(i)}, n \in\mathbb N, i \in I$ be
c\`adl\`ag functions from $[0,\infty)$ to $\mathcal S$, with $I$ a
finite set. For $t \geq0$, set $g_n(t)=\{u^{(n,i)}_j(t),j \geq1,i \in
I\}^{\downarrow}$ and $g(t)=\{u^{(n,i)}_j(t),j \geq1,i \in I\}
^{\downarrow}$. These functions are c\`adl\`ag. Moreover, if
$u^{(n,i)} \rightarrow u^{(i)}$ as $n \rightarrow\infty$ in the
Skorokhod sense for all $i \in I$ and if the functions $u^{(i)}, i \in
I$ do not jump simultaneously on $[0,\infty)$, then $g_n$ converges in
the Skorokhod sense to $g$ as $n \rightarrow\infty$.
\end{lem}

%s9.2 #&#
\subsection{Properties of the stationary and biased Markov chains}
\label{secappendix2}

We collect here various technical results about the stationary and
biased Markov chains (introduced in Section~\ref{secStatPro}) which
are used in the body of the paper.

%le9.7 #&#
\begin{lem}
\label{lemqualitatif}
If $\int_{\mathcal S_1} s_1^{-1} \nu(\mathrm d \mathbf s)<\infty$,
then for all $c>0$,
\label{LemmaI}
\[
\int_0^{\infty}\frac{\exp{(-cx)}}{f_{\zeta}(x)}\pi_{\mathrm
{stat}}(
\mathrm d x)<\infty.
\]
\end{lem}

\begin{pf}
It suffices to prove the result for small values of $c>0$.
As in the proof of Lemma~\ref{lemFosterLyap}, let $V(x)=\exp
(-cx)/f_{\zeta}(x), x>0$, with $c \in(0,1/2)$ small enough so that
$\exp(cx)f_{\zeta}(x)\rightarrow0$ as $x \rightarrow\infty$.
Then, as a direct consequence of (\ref{eqdrift}) and
Theorem 14.0.1 of \cite{MeynTweedie}, we have that $\int_{0}^{\infty
}V(x) \pi_{\mathrm{stat}}(\mathrm dx)<\infty$. The result follows.
\end{pf}

%le9.8 #&#
\begin{lem}
\label{logZfiniteexpect}
If $\int_{\mathcal S_1} s_1^{-1} \nu(\mathrm d \mathbf s)<\infty$,
then for $a>0$ sufficiently small and all $b<1+1/\llvert \alpha
\rrvert $,
\[
\int_1^{\infty} \exp(ax) \pi_{\mathrm{stat}}(\mathrm
d x) < \infty\quad\mbox{and}\quad\int_0^{1}
x^{-b} \pi_{\mathrm{stat}}(\mathrm d x) < \infty.
\]
In particular, for all $p > 0$,
\[
\mathbb{E} \bigl[\bigl\llvert\log\bigl(Z_0^{\mathrm{stat}}\bigr)
\bigr\rrvert^p \bigr] < \infty. %
\]
\end{lem}

\begin{pf}
To see the first assertion, note that by Lemma~\ref{LemmaDensity},
there exist constants $C_1> 0$ and $c > 0$ such that
\[
f_{\zeta}(x) \le C_1 \exp(-cx)
\]
for all $x > 0$. Hence, for all $a<c$, by Lemma~\ref{lemqualitatif},
\[
\int_0^{\infty} \exp(ax) \pi_{\mathrm{stat}}(\mathrm
d x) \le C_1 \int_0^{\infty}
\frac{\exp(-(c-a)x)}{f_{\zeta}(x)} \pi_{\mathrm
{stat}}(\mathrm d x) < \infty.
\]
Next, from (\ref{Eqstat}), we have that
\[
\frac{\pi_{\mathrm{stat}}(x)}{f_{\zeta}(x)} = \int_{{\mathcal
S}_1} \Biggl(\sum
_{i=1}^{\infty}e^{s_{i}^{-\alpha}x}\prod
_{j\neq
i}\mathbb{F}_{\zeta}\bigl(s_{j}^{\alpha}s_{i}^{-\alpha}x
\bigr) \biggl(\int_{s_{i}^{-\alpha
}x}^{\infty}\frac{e^{-y}\pi_{\mathrm{stat}}(y)}{f_{\zeta}(y)}\,
\mathrm d y \biggr) \Biggr)\nu(\mathrm d \mathbf s).
\]
Recall\vspace*{-2pt} the definition of $\zeta^{(I)}$ from just below equation (\ref
{eqzeta}). Since $\int_0^{\infty} \frac{\exp(-x)}{f_{\zeta}(x)}\*
\pi_{\mathrm{stat}}(\mathrm d x) < \infty$ and $e^{s_i^{-\alpha}x}
\le e^x$, there exists a constant $C$ such that
\begin{eqnarray*}
\frac{\pi_{\mathrm{stat}}(x)}{f_{\zeta}(x)} & \le&\frac{Ce^x}{1 -
\mathbb{F}_{\zeta}(x)} \int_{{\mathcal
S}_1}
\Biggl(\sum_{i=1}^{\infty}\bigl(1 -
\mathbb{F}_{\zeta}(x)\bigr) \prod_{j\neq i}
\mathbb{F}_{\zeta}\bigl(s_{j}^{\alpha}s_{i}^{-\alpha}x
\bigr) \Biggr)\nu(\mathrm d \mathbf s)
\\
& \le&\frac{Ce^x}{1 - \mathbb{F}_{\zeta}(x)} \mathbb P \bigl(\zeta
^{(I)} > x \bigr) \le
\frac{Ce^x}{1 - \mathbb{F}_{\zeta}(x)}.
\end{eqnarray*}
Then, for $x \in(0,1]$, $\pi_{\mathrm{stat}}(x)/f_{\zeta}(x)$ is
bounded by some constant $C_2$. It follows that
\[
\int_0^{1} x^{-b}
\pi_{\mathrm{stat}}(\mathrm{d} x) \le C_2 \int_0^{1}
x^{-b} f_{\zeta}(x)\, \mathrm{d} x,
\]
and this upper bound is finite by Lemma~\ref{LemmaDensity}(iii) when
$b<1+1/\llvert \alpha\rrvert $.
\end{pf}

%le9.9 #&#
\begin{lem}\label{mu} If $\int_{\Sfl}s_1^{-1} \nu(\mathrm d \mathbf
s)<\infty$, then $\mathbb E [(\log({Y}^{\mathrm
{stat}}_1))^p ]<\infty$ for all $p>0$.
\end{lem}

\begin{pf} Fix $p>0$.
By definition,
\begin{eqnarray*}
&& \mathbb E \bigl[\bigl(\log\bigl({Y}^{\mathrm{stat}}_1\bigr)
\bigr)^p \bigr]
\\
&&\qquad = \frac
{1}{\llvert \alpha\rrvert ^p}\int_{0}^{\infty}
\mathbb E \biggl[ \biggl(\log\biggl( \frac{\zeta}{\zeta-T_1} \biggr)
\biggr)^p \Big|\zeta=x \biggr] \pi_{\mathrm{stat}} (\mathrm dx)
\\
&&\qquad  \leq\frac{C_p}{\llvert \alpha\rrvert ^p} \int_{0}^{\infty} \bigl(
\bigl\llvert\log(x)\bigr\rrvert^p + \mathbb E \bigl[ \bigl\llvert
\log(\zeta-T_1)\bigr\rrvert^p \mathbh1_{\{\zeta
-T_1 \leq1\}}
\mid\zeta=x \bigr] \bigr)\pi_{\mathrm{stat}} (\mathrm dx),
\end{eqnarray*}
for some constant $C_p$. By Lemma~\ref{logZfiniteexpect},
$
\int_{0}^{\infty} \llvert \log(x)\rrvert ^p \pi_{\mathrm{stat}}
(\mathrm dx)
<\infty$.
Next, using the notation introduced in Lemma~\ref{LemmaDensity}, we
write $\zeta=T_1+\xi$ where $\xi=\max_{i\geq1}\{F_i^{-\alpha
}(T_1)\zeta^{(i)}\}$. Since $T_1$ is independent of $\xi$ and is
exponentially distributed with mean 1, the joint distribution of $(\xi
,\zeta)$ is $\exp(-x+y)\times   \mathbh1_{0 \leq y \leq x} f_{\xi
}(y)\,\mathrm dy \,\mathrm dx$,
where we recall that $f_{\xi}$ denotes the density of $\xi$. Hence
\begin{eqnarray*}
&& \int_{0}^{\infty} \mathbb E \bigl[ \bigl\llvert
\log(\zeta-T_1)\bigr\rrvert^p \mathbh1_{\{ \zeta-T_1 \leq1\}}
\mid\zeta=x \bigr]\pi_{\mathrm{stat}} (\mathrm dx)
\\
&&\qquad  = \int_0^{\infty} \frac{\exp(-x)}{f_{\zeta}(x)} \biggl(\int
_0^{\min(x,1)} \exp(y) \llvert\log y\rrvert
^p f_{\xi}(y)\,\mathrm dy \biggr) \pi_{\mathrm{stat}} (
\mathrm dx)
\\
&&\qquad \leq e\int_0^{1} \llvert\log y\rrvert
^p f_{\xi}(y)\,\mathrm dy \int_0^{\infty}
\frac{\exp(-x)}{f_{\zeta}(x)} \pi_{\mathrm{stat}} (\mathrm dx).
\end{eqnarray*}
The integral $\int_0^{\infty} \exp(-x)/f_{\zeta}(x) \pi_{\mathrm
{stat}} (\mathrm dx)$ is finite, by Lemma~\ref{lemqualitatif}.
Finally, note that $\xi\geq F_1(T_1)^{-\alpha} \zeta^{(1)}$ and so
\[
\mathbb E \bigl[ \bigl\llvert\log(\xi)\bigr\rrvert^p
\mathbh1_{\{ \xi\leq
1\}} \bigr] \leq C_p \bigl(\llvert\alpha\rrvert
^p\mathbb E \bigl[ \bigl\llvert\log\bigl(F_1(T_1)
\bigr) \bigr\rrvert^p \bigr] + \mathbb E \bigl[\bigl\llvert\log
\bigl(\zeta^{(1)} \bigr) \bigr\rrvert^p \bigr] \bigr).
\]
The first expectation on the right-hand side is equal to $\int
_{\mathcal S_1} \llvert \log(s_1) \rrvert ^p \nu(\mathrm d \mathbf s
)$ and is finite since $\int_{\mathcal S_1} s_1^{-1} \nu(\mathrm d
\mathbf{s}) < \infty$. The\vspace*{2pt} second expectation is also finite, by
assertions (i) and (ii) of Lemma~\ref{LemmaDensity}.
\end{pf}

The following result is the only place that we need the extra condition
$\int_{\mathcal S_1} s_1^{-1-\rho}\nu(\mathrm d \mathbf s)<\infty$
for some $\rho>0$.

%le9.10 #&#
\begin{lem}
\label{lemLLN2}
Assume\vspace*{2pt} that $\int_{\mathcal S_1} s_1^{-1-\rho}\nu(\mathrm d \mathbf
s)<\infty$ for some $\rho>0$. Then there exits $\delta_{\rho}>0$
such that for all $\delta\in[0,\delta_{\rho})$,
\[
\mathbb E \bigl[ \bigl\llvert\log\bigl(\mathbb{F}_{\zeta}
\bigl(Z_0^{\mathrm
{stat}} \bigl(Y_1^{\mathrm{stat}}
\bigr)^{\alpha}\bigr) \bigr) \bigr\rrvert^{1+\delta
} \bigr]<\infty.
\]
\end{lem}

\begin{pf}
Again we let $\xi=\sup_{i\geq1}\{F_i(T_1)^{-\alpha}\zeta^{(i)}\}$.
The first step of our proof is to show that, for $0<x \leq1$,
%
%e9.1 #&#
\begin{equation}
\label{eqtech1} \bigl\llvert\log\bigl(\mathbb{F}_{\zeta}(x)\bigr)\bigr
\rrvert\leq C \bigl(x^{1/\alpha} \llvert\log x\rrvert+ x^{1/\alpha}
\bigl\llvert\log\bigl(\mathbb{F}_{\xi}(x)\bigr)\bigr\rrvert\bigr).
\end{equation}
For $x>0$, let $K(x)=\sup\{k \geq1\dvtx
F_k(T_1)>x^{-1/\alpha}\}$, and let $C_1>1$ be
such that $1-t \geq\exp(-C_1 t)$ for all $t \in[0,\mathbb P(\zeta
>1))$. Then
\begin{eqnarray*}
\prod_{i\geq K(x)+1}\mathbb{F}_{\zeta}\bigl(x
F_i(T_1)^{\alpha}\bigr) &\geq&\exp
\biggl(-C_1 \sum_{i\geq K(x)+1} \mathbb P\bigl(
\zeta>xF_i(T_1)^{\alpha}\bigr) \biggr)
\\
&\geq&\exp\biggl(-C_1 \mathbb E\bigl[\zeta^{-1/\alpha}\bigr]
x^{1/\alpha}\sum_{i\geq K(x)+1}F_i(T_1)
\biggr)
\\
&\geq&\exp\bigl(-C_2 x^{1/\alpha} \bigr),
\end{eqnarray*}
where we have used Markov's inequality to get the second inequality and
the fact that $\sum_{i\geq K(x)+1}F_i(T_1) \leq1$ to get\vspace*{2pt}
the third. Now note that $K(x) \le x^{1/\alpha}$ since $F_k(T_1) \le
1/k$ for all $k \ge1$. So, for $c \in(0,1)$ such that $\nu(s_1\leq c)>0$,
%
%e9.2 #&#
\begin{eqnarray}\label{xiineq}
&& \mathbb{F}_{\xi}(x)=\mathbb E \biggl[ \prod
_{i\geq1}\mathbb{F}_{\zeta}\bigl(x F_i(T_1)^{\alpha}
\bigr) \biggr]\nonumber
\\
&&\qquad \geq \mathbb{E} \Biggl[\prod_{i=1}^{K(x)}
\mathbb{F}_{\zeta}\bigl(x F_1(T_i)^{\alpha}
\bigr) \Biggr] \exp\bigl(-C_2 x^{1/\alpha} \bigr)
\nonumber\\[-8pt]\\[-8pt]\nonumber
&&\qquad \geq \mathbb E \bigl[ \mathbb{F}_{\zeta}\bigl(x c^{\alpha
}\bigr)^{K(x)}\mathbh1_{\{F_1(T_1) \leq c\}} \bigr]\exp\bigl(-C_2
x^{1/\alpha} \bigr)
\nonumber
\\
&&\qquad \geq \nu(s_1 \leq c) \mathbb{F}_{\zeta}\bigl(xc^{\alpha}
\bigr)^{x^{1/\alpha
}}\exp\bigl(-C_2 x^{1/\alpha} \bigr).\nonumber
\end{eqnarray}
Next, since $\mathbb{F}_{\zeta}(x)=\exp(-x)\int_0^x \exp(y)\mathbb
{F}_{\xi}(y)\,\mathrm dy$, we have that for all $0<x \leq1$,
%
%e9.3 #&#
\begin{equation}
\label{zetaineq} \mathbb{F}_{\zeta}(x) \geq\exp(-1) \bigl(1-c^{-\alpha/2}
\bigr) x \mathbb{F}_{\xi}\bigl(c^{-\alpha/2} x\bigr).
\end{equation}
Using (\ref{xiineq}), we get
\[
\mathbb{F}_{\zeta}(x) \geq C_3 x \mathbb{F}_{\zeta}
\bigl(c^{-\alpha/2} x c^{\alpha}\bigr)^{c^{-1/2} x^{1/\alpha}} \exp
\bigl(-C_2 c^{-1/2} x^{1/\alpha}\bigr),
\]
and another application of (\ref{zetaineq}) yields
\[
\mathbb{F}_{\zeta}(x) \geq C_3 x \bigl(C_4 x
\mathbb{F}_{\xi
}\bigl(c^{-\alpha} x c^{\alpha}\bigr)
\bigr)^{c^{-1/2} x^{1/\alpha}} \exp\bigl(-C_2 c^{-1/2}
x^{1/\alpha}\bigr).
\]
All of the constants here are strictly positive, and so (\ref
{eqtech1}) follows.

From (\ref{eqtech1}) and H\"older's inequality, we deduce that for
all $\delta>0$,
\begin{eqnarray*}
&& \mathbb E \bigl[ \bigl\llvert\log\bigl(\mathbb{F}_{\zeta}(\xi)\bigr)
\mathbh1_{\{
\xi\leq1\}}\bigr\rrvert^{1+\delta} \bigr]
\\
&&\qquad  \leq C'
\bigl(\mathbb E \bigl[ \xi^{{(1+\delta)^2}/{\alpha}} \bigr]
+\mathbb E \bigl[ \xi
^{{(1+\delta)^2}/{\alpha}} \bigr]^{1/{(1+\delta)}} \mathbb E
\bigl[\bigl\llvert\log\bigl(
\mathbb{F}_{\xi}(\xi)\bigr)\bigr\rrvert^{(1+\delta )^2/{\delta}} \bigr]
^{\delta/{(1+\delta)}} \bigr).
\end{eqnarray*}
Since $\mathbb{F}_{\xi}(\xi)$ has a uniform distribution, $\llvert \log
(\mathbb{F}_{\xi}(\xi))\rrvert \sim\operatorname{Exp}(1)$ and so has finite
positive moments of all orders. Moreover, since $\xi\geq
F_1(T_1)^{-\alpha} \zeta^{(1)}$, we have
\[
\mathbb E \bigl[ \xi^{{(1+\delta)^2}/{\alpha}} \bigr] \leq\mathbb
E \bigl[
\zeta^{{(1+\delta)^2}/{\alpha}} \bigr] \int_{\mathcal S_1}
s_1^{-(1+\delta)^2}
\nu(\mathrm d \mathbf s). %
\]
Let $\rho>0$ be such that $\int_{\mathcal S_1} s_1^{-1-\rho}\nu
(\mathrm d \mathbf s)<\infty$. By Lemma~\ref{LemmaDensity}(iii),
$\mathbb E[\zeta^{-a}]<\infty$ for all $a<1+(1+\rho)/\llvert \alpha
\rrvert $. So
for all $\delta\geq0$ such that $(1+\delta)^2 \leq1+\rho$, the
expectation $\mathbb E [ \xi^{(1+\delta)^2/\alpha} ]$ is
finite and thus
\[
\mathbb E \bigl[ \bigl\llvert\log\bigl(\mathbb{F}_{\zeta}(\xi)\bigr)
\bigr\rrvert^{1+\delta} \bigr]<\infty%
\]
[since $\llvert \log(\mathbb{F}_{\zeta}(\xi)) \rrvert \leq\llvert
\log(\mathbb{F}_{\zeta}(1)) \rrvert $ when $\xi\geq1$].

In particular, we can deduce that $\mathbb E [ | \log(\mathbb
{F}_{\zeta}(\xi)) | ^{1+\delta} | \zeta=x_0 ]<\infty
$ for some $x_0>0$.
Our goal now is to check that
\[
\mathbb E \bigl[ \bigl\llvert\log\bigl(\mathbb{F}_{\zeta}
\bigl(Z_0^{\mathrm
{stat}} \bigl(Y_1^{\mathrm{stat}}
\bigr)^{\alpha}\bigr)\bigr) \bigr\rrvert^{1+\delta}\bigr]<\infty.
\]
Recall that $\xi=\zeta-T_1$ and so $Z_0 Y_1^{\alpha} = Z_1 \Theta
_1^{-\alpha} = \xi$. Hence,
\[
\mathbb E \bigl[ \bigl\llvert\log\bigl(\mathbb{F}_{\zeta}
\bigl(Z_0^{\mathrm
{stat}} \bigl(Y_1^{\mathrm{stat}}
\bigr)^{\alpha}\bigr) \bigr) \bigr\rrvert^{1+\delta
} \bigr]=\int
_0^{\infty} \mathbb E \bigl[ \bigl\llvert\log\bigl(
\mathbb{F}_{\zeta} (\xi) \bigr) \bigr\rrvert^{1+\delta} \mid\zeta
=x \bigr] \pi_{\mathrm{stat}}(\mathrm d x).
\]
Write $\int_0^{\infty}=\int_0^{x_0}+\int_{x_0}^{\infty}$, where
$x_0$ is chosen so that $\mathbb E [ \llvert \log(\mathbb
{F}_{\zeta}(\xi)) \rrvert ^{1+\delta} \mid\zeta=x_0 ]<\infty
$. As seen in the proof of Lemma~\ref{logZfiniteexpect}, $\pi
_{\mathrm{stat}}(x) \leq C_{x_0} f_{\zeta} (x)$ on $(0,x_0)$. Hence,
\[
\int_0^{x_0} \mathbb E \bigl[ \bigl\llvert\log
\bigl(\mathbb{F}_{\zeta}(\xi)\bigr) \bigr\rrvert^{1+\delta} \mid
\zeta=x \bigr] \pi_{\mathrm
{stat}}(\mathrm d x) \leq C_{x_0} \mathbb E
\bigl[ \bigl\llvert\log\bigl(\mathbb{F}_{\zeta}(\xi)\bigr) \bigr
\rrvert
^{1+\delta} \bigr]<\infty. %
\]
Next, for $x >x_0$, we use the fact that the joint distribution of
$(\xi,\zeta)$ is $\exp(-z+y) \mathbh1_{0 \leq y \leq z} f_{\xi
}(y)\,\mathrm dy \,\mathrm dz$,
to obtain that
\begin{eqnarray*}
&& \mathbb E \bigl[ \bigl\llvert\log\bigl(\mathbb{F}_{\zeta}(\xi)\bigr)
\bigr\rrvert^{1+\delta} \mid\zeta=x \bigr]
\\
&&\qquad = \frac{e^{-x}}{f_{\zeta
}(x)} \int
_0^x e^{y} \bigl\llvert\log\bigl(
\mathbb{F}_{\zeta}(y)\bigr) \bigr\rrvert^{1+\delta}
f_{\xi}(y)\,\mathrm dy
\\
&&\qquad \leq \frac{e^{-x}}{f_{\zeta}(x)} \int_0^{x_0}
e^{y} \bigl\llvert\log\bigl(\mathbb{F}_{\zeta}(y)\bigr)
\bigr\rrvert^{1+\delta} f_{\xi}(y)\,\mathrm dy
\\
&&\quad\qquad{}+ \bigl\llvert\log
\bigl(\mathbb{F}_{\zeta}(x_0)\bigr) \bigr\rrvert
^{1+\delta} \frac
{e^{-x}}{f_{\zeta}(x)} \int_{x_0}^x
e^{y} f_{\xi}(y)\,\mathrm dy
\\
&&\qquad \leq \frac{e^{-x}f_{\zeta}(x_0)}{f_{\zeta}(x) e^{-x_0}} \mathbb E \bigl
[ \bigl\llvert\log\bigl(
\mathbb{F}_{\zeta}(\xi)\bigr) \bigr\rrvert^{1+\delta} \mid
\zeta=x_0 \bigr] + \bigl\llvert\log\bigl(\mathbb{F}_{\zeta}(x_0)
\bigr) \bigr\rrvert^{1+\delta}.
\end{eqnarray*}
The integral of this upper bound with respect to $\pi_{\mathrm
{stat}}(\mathrm d x)$ on $(x_0,\infty)$ is finite, by Lemma~\ref
{lemqualitatif}.
\end{pf}

We now prove some almost sure limits for the biased chain.

%le9.11 #&#
\begin{lem} \label{lemLLN}
As $n \to\infty$, the following limits hold almost surely:
\begin{eqnarray*}
%\begin{array} {ll}
\frac{1}{n} \sum _{j=1}^n \log\bigl(Y_j^{\mathrm{bias}}
\bigr) &\to&\mu,\qquad\frac{1}{n} \sum_{j=-n+1}^{0}
\log\bigl(Y_j^{\mathrm{bias}}\bigr) \to\mu,
\\
\frac{1}{n} \log\bigl(Y_{-n}^{\mathrm{bias}}
\bigr) &\to& 0, \qquad\frac{1}{n} \log\bigl(Z_{-n}^{\mathrm{bias}}\bigr) \to0,
\\
\frac{1}{n} \log\bigl(\mathbb{F}_{\zeta}
\bigl(Z_{-n-1}^{\mathrm{bias}} \bigl(Y_{-n}^{\mathrm{bias}}
\bigr)^{\alpha}\bigr)\bigr) &\to&0.
%\end{array}
\end{eqnarray*}
\end{lem}

\begin{pf}
Suppose that $(X_k)_{k \ge0}$ is any positive Harris chain possessing
an invariant distribution. Then Theorem 17.0.1 of Meyn and
Tweedie~\cite{MeynTweedie} gives the following law of large numbers:
for any function $g$ such that $\mathbb{E} [\llvert g(X_0^{\mathrm
{stat}})\rrvert ] < \infty$,
\[
\frac{1}{n} \sum_{j=1}^n
g(X_j) \to\mathbb{E} \bigl[g\bigl(X_0^{\mathrm
{stat}}
\bigr) \bigr]
\]
almost surely, as $n \to\infty$, irrespective of the distribution of
$X_0$. Moreover, it follows straightforwardly from\vspace*{2pt} this that
$
n^{-1} g(X_n) \to0
$
almost surely, as $n \to\infty$.

Now note that $(Z_k^{\mathrm{bias}}, Y_k^{\mathrm{bias}})_{k \ge1}$
is a realization of the Markov chain\break $(Z_k, Y_k)_{k \ge1}$ with
initial state $(Z_1,Y_1)$ having the distribution specified (for
suitable test functions $\phi$) by
\[
\mathbb{E} \bigl[\phi(Z_1, Y_1) \bigr] =
\frac{1}{\mu} \mathbb{E} \bigl[\log\bigl(Y_1^{\mathrm{stat}}\bigr)
\phi\bigl(Z_1^{\mathrm{stat}}, Y_1^{\mathrm{stat}}\bigr)
\bigr].
\]
Since $\mathbb{E} [\log(Y_1^{\mathrm{stat}}) ] = \mu<
\infty$, we get that a.s.
\[
\frac{1}{n} \sum_{j=1}^n \log
\bigl(Y_j^{\mathrm{bias}}\bigr) \to\mu.
\]
Observe next that $(Z_{-k}^{\mathrm{bias}}, Y_{-k}^{\mathrm
{bias}})_{k \ge0}$ is a realization of the (backward) Markov chain
$(Z_{-k}, Y_{-k})_{k \ge0}$ with initial distribution for $(Z_0, Y_0)$
specified (for suitable test functions $\phi$) by
\[
\mathbb{E} \bigl[\phi(Z_0,Y_0) \bigr] =
\frac{1}{\mu} \mathbb{E} \bigl[\log\bigl(Y_1^{\mathrm{stat}}\bigr)
\phi\bigl(Z_0^{\mathrm{stat}}, Y_0^{\mathrm{stat}}\bigr)
\bigr].
\]
The chain $(Z_{-k}, Y_{-k})_{k \ge0}$ is also a positive Harris chain
possessing the same invariant distribution as $(Z_k, Y_k)_{k \ge1}$. Hence,
\[
\frac{1}{n} \sum_{j=-n+1}^{0} \log
\bigl(Y_j^{\mathrm{bias}}\bigr) \to\mu\quad\mbox{and}\quad
\frac{1}{n} \log\bigl(Y_{-n}^{\mathrm{bias}}\bigr) \to0
\]
almost surely, as before. By Lemma~\ref{logZfiniteexpect}, $\mathbb
{E} [\llvert \log(Z_1^{\mathrm{stat}})\rrvert ] < \infty$
and, by the $\delta= 0$ case of Lemma~\ref{lemLLN2}, $\mathbb{E}
[\llvert \log(\mathbb{F}_{\zeta}(Z_0^{\mathrm{stat}} (Y_1^{\mathrm
{stat}})^{\alpha}))\rrvert ] < \infty$, and so we also have the
almost sure convergences
\[
\frac{1}{n} \log\bigl(Z_{-n}^{\mathrm{bias}}\bigr) \to0 \quad
\mbox{and}\quad\frac{1}{n} \bigl\llvert\log\bigl(\mathbb{F}_{\zeta}
\bigl(Z_{-n-1}^{\mathrm
{bias}}\bigl(Y_{-n}^{\mathrm{stat}}
\bigr)^{\alpha}\bigr)\bigr)\bigr\rrvert\to0.
\]
\upqed
\end{pf}

Finally, we show that $\mathbb{E} [\prod_{i=1}^n (Y^{\mathrm
{stat}}_i)^{\alpha} ]$ decays exponentially in $n$.

%le9.12 #&#
\begin{lem} \label{lemexptails}
For any $x > 0$, we have
\[
\limsup_{n \to\infty} \frac{1}{n} \log\mathbb{E} \Biggl[\prod
_{i=1}^n \bigl(Y^{\mathrm{stat}}_i
\bigr)^{\alpha} \Biggr] < 0.
\]
\end{lem}

In order to prove Lemma~\ref{lemexptails}, we use a renewal process
derived from the biased Markov chain $(Z_n^{\mathrm{bias}})_{n \in\Z
}$. We therefore begin with a result about general renewal processes.

Suppose that $(N(n))_{n \ge0}$ is a delayed renewal process. Write
$\tau_0$ for the delay and $\tau_1, \tau_2, \ldots$ for the
subsequent arrival times, so that $\tau_{k+1} - \tau_k$ are i.i.d.
random variables for $k \ge0$, independent of $\tau_0$, and $N(n) = \#
\{k \ge1\dvtx  \tau_k \le n\}$. We will say that a random variable $X$ has
\emph{exponential tails} if there exists $r > 1$ such that $\mathbb
{E} [r^X ] < \infty$.

%le9.13 #&#
\begin{lem} \label{lemrenewal}
Suppose that $\tau_0$ and $\tau_1 - \tau_0$ both have exponential
tails. Then for any $s \in(0,1)$,
\[
\limsup_{n \to\infty} \frac{1}{n} \log\mathbb{E}
\bigl[s^{N(n)} \bigr] < 0.
\]
\end{lem}

\begin{pf} The proof is elementary, and so we sketch it. Let $\chi=
\mathbb{E} [\tau_1 - \tau_0 ]$ be the mean of the
standard inter-arrival distribution and take $\varepsilon> 0$. Then
\begin{eqnarray*}
\mathbb{E} \bigl[s^{N(n)} \bigr] & \le&\mathbb{P} \bigl(N(n) < \bigl(
\chi^{-1} - \varepsilon\bigr)n \bigr) + s^{(\chi^{-1} - \varepsilon)n}
\\
& \le&\mathbb{P} (\tau_{k_n} \ge n ) + s^{(\chi^{-1} -
\varepsilon)n},
\end{eqnarray*}
where $k_n = \lfloor(\chi^{-1} - \varepsilon)n \rfloor$. But a simple
application of the G\"artner--Ellis theorem then implies that
\[
\mathbb{P} (\tau_{k_n} \ge n ) \le\mathbb{P} \bigl(
\tau_{k_n} \ge k_n \chi/(1 - \chi\varepsilon) \bigr)
\]
is exponentially small in $n$. The result follows.
\end{pf}

Suppose now that we mark the $k$th inter-arrival interval with some
probability which depends, in general, on its length $\tau_k - \tau
_{k-1}$, but independently for different inter-arrival intervals. Let
$I_k$ be the indicator that the $k$th inter-arrival interval is marked,
so that $I_1, I_2, \ldots$ are independent Bernoulli random variables
such that $I_k$ depends on $\tau_i, i \geq0$ only through $\tau_k -
\tau_{k-1}$. Let
%
%e9.4 #&#
\begin{equation}
\label{eqnmarks} M(n) = \#\{k \ge1\dvtx  \tau_k \le n, I_k = 1
\}.
\end{equation}
$(M(n))_{n \ge0}$ is again a delayed renewal process.

%le9.14 #&#
\begin{lem} \label{lemmarkedrenewal}
Suppose that $\tau_0$ and $\tau_1 - \tau_0$ have exponential tails
and that $q:= \mathbb{P} (I_1 = 1 ) > 0$. Then the delay
and inter-arrival distributions of $(M(n))_{n \ge0}$ have exponential
tails. Hence, for any $s \in(0,1)$,
\[
\limsup_{n \to\infty} \frac{1}{n} \log\mathbb{E}
\bigl[s^{M(n)} \bigr] < 0.
\]
\end{lem}

\begin{pf} The case $q = 1$ follows immediately from Lemma~\ref
{lemrenewal}, and so we henceforth assume that $q < 1$.
Let $\sigma_1, \sigma_2, \ldots, \tilde{\sigma}$ and $G$ be
mutually independent random variables, independent of $\tau_0$. Let
$\sigma_1, \sigma_2, \ldots$ have common \mbox{distribution} given by
$\mathbb{P} (\sigma_1 = i ) = \mathbb{P} (\tau_1
- \tau_0 = i \mid I_1 = 0 )$, $i \ge1$. Let $\tilde{\sigma}$
have distribution $\mathbb{P} (\tilde{\sigma} = i ) =
\mathbb{P} (\tau_1 - \tau_0 = i \mid I_1 = 1 )$, $i \ge
1$. Finally, let $G$ be such that $\mathbb{P} (G = i ) =
q (1 - q)^{i}$ for $i \ge0$. Then the delay has the same distribution as
\[
\tau_0 + \sum_{i=1}^G
\sigma_i + \tilde{\sigma}
\]
and the inter-arrival intervals have the same distribution as
\[
\sum_{i=1}^G \sigma_i +
\tilde{\sigma}.
\]
By Lemma~\ref{lemrenewal}, it will be sufficient to prove that $\sum
_{i=1}^G \sigma_i$ and $\tilde{\sigma}$ are random variables with
exponential tails. For $r \ge0$,
\[
\mathbb{E} \bigl[r^{\sigma_1} \bigr] = \mathbb{E} \bigl[r^{\tau
_1 - \tau_0}
\mid I_1 = 0 \bigr] \le\frac{\mathbb{E} [r^{\tau
_1 - \tau_0} ]}{1 -q}
\]
and, similarly,
\[
\mathbb{E} \bigl[r^{\tilde{\sigma}} \bigr] = \mathbb{E} \bigl[r^{\tau_1
- \tau_0}
\mid I_1 = 1 \bigr] \le\frac{\mathbb{E}
[r^{\tau_1 - \tau_0} ]}{q}.
\]
By assumption, there exists $r > 1$ such that $\mathbb{E}
[r^{\tau_1 - \tau_0} ] < \infty$. Hence, there exists $r > 1$
such that $\mathbb{E} [r^{\sigma_1} ] < \infty$ and
$\mathbb{E} [r^{\tilde{\sigma}} ] < \infty$. Moreover,
\[
\mathbb{E} \bigl[r^{\sum_{i=1}^G \sigma_i} \bigr] = \frac{r q}{1
- (1 - q) \mathbb{E} [r^{\sigma_1} ]}.
\]
Now $\mathbb{E} [r^{\sigma_1} ] \to1$ as $r \downarrow
1$, and so we can find $r > 1$ sufficiently small that $\mathbb{E}
[r^{\sigma_1} ] < (1-q)^{-1}$. Hence, for such a value of $r$,
\[
\mathbb{E} \bigl[r^{\sum_{i=1}^G \sigma_i} \bigr] < \infty.
\]
The result follows.
\end{pf}

Recall from Lemma~\ref{lemFosterLyap} the Foster--Lyapunov criterion
for the Markov chain $(Z_k)_{k \ge0}$: there exist a function $V\dvtx
(0,\infty) \to[1, \infty)$, a small set $C$ and constants $\beta\in
(0,1)$ and $b > 0$ such that
\[
\mathbb{E} \bigl[V(Z_1) \mid Z_0 = x \bigr] \le(1 -
\beta) V(x) + b \mathbh{1}_{ \{ x \in C \} }.
\]
Since $C$ is small, there exist $p \in(0,1)$ and a probability measure
$\tilde{\mu}_C$ [which is a version of the measure $\mu_C$ given
explicitly at (\ref{eqnboundingmeasure}) normalized to have total
mass~1] such that
\[
P(x, B) = \mathbb{P} (Z_1 \in B \mid Z_0 \in x ) \ge p
\tilde{\mu}_C(B)
\]
for all $x \in C$ and any $B$ any Borel subset of $(0,\infty)$.
Consider now constructing the process $(Z_k)_{k \ge0}$ via the
standard \emph{split chain construction}: whenever $Z_k \in C$, we
flip a coin with probability $p \in(0,1)$. If the coin comes up heads,
we sample $Z_{k+1}$ from the measure $\tilde{\mu}_C$. Otherwise,
sample $Z_{k+1}$ from the probability measure $(P(Z_k, \cdot) - p
\tilde{\mu}_C(\cdot))/(1 - p)$. If $Z_k \notin C$, we simply sample
$Z_{k+1}$ from $P(Z_k, \cdot)$. If $Z_k \in C$ and the coin comes up
heads, we say that there is a \emph{regeneration} at time~$k$. (In
particular, a regeneration can only occur at $k$ if $Z_k \in C$.) Let
\[
\tau_0 = \inf\{i \ge0\dvtx  \mbox{there is a regeneration at $i$}\}
\]
and for $k \ge0$,
\[
\tau_{k+1} = \inf\{i > \tau_k\dvtx  \mbox{there is a
regeneration at $i$}\}.
\]
Then $\tau_0$ and $\{\tau_{k+1} - \tau_k\dvtx  k \ge0\}$ are all
independent, and $\{\tau_{k+1} - \tau_k\dvtx  k \ge0\}$ are identically
distributed. Hence, $N(n):= \#\{k \ge1\dvtx  \tau_k \le n\}$ is a delayed
renewal process.

The following lemma is a standard consequence of geometric ergodicity;
see, for example, equation (22) of Roberts and Rosenthal~\cite
{RobertsRosenthal} for the precise formulation given here.

%le9.15 #&#
\begin{lem} \label{lemregen}
There exists $\theta> 1$ such that
\[
\int_0^{\infty} \mathbb{E} \bigl[
\theta^{\tau_0} \mid Z_0 = x \bigr] \pi_{\mathrm{stat}}(\mathrm
d x) < \infty\quad\mbox{and}\quad\mathbb{E} \bigl[\theta^{\tau_1 -
\tau_0} \bigr]
< \infty.
\]
\end{lem}

Hence, if the chain is begun in stationarity, $(N(n))_{n \ge0}$ is a
delayed renewal process such that both delay and inter-arrival
distributions have exponential tails.

\begin{pf*}{Proof of Lemma~\ref{lemexptails}}
Let $f\dvtx  (0,\infty)^2 \to(0,1)$ be defined by
\[
f(x,y)= \mathbb{E} \bigl[Y_1^{\alpha} \mid Z_0 =
x, Z_1 = y \bigr].
\]
Using the fact that $(Z_n)_{n \ge0}$ acts a driving chain for $(Z_n,
Y_n)_{n \ge0}$, we have that $Y_1, Y_2, \ldots, Y_n$ are
conditionally independent given $Z_0, Z_1, \ldots, Z_n$ and, for $1
\le i \le n$, the distribution of $Y_i$ depends only on the values of
$Z_{i-1}$ and $Z_{i}$. Hence, for all $x>0$,
\[
\mathbb{E} \Biggl[\prod_{i=1}^n
Y_i^{\alpha} \bigg| Z_0 = x \Biggr] = \mathbb{E}
\Biggl[\prod_{i=1}^n f(Z_{i-1},
Z_i) \bigg| Z_0 = x \Biggr]
\]
and, therefore,
\[
\mathbb{E} \Biggl[\prod_{i=1}^n
\bigl(Y^{\mathrm{stat}}_i\bigr)^{\alpha} \Biggr] = \mathbb{E}
\Biggl[\prod_{i=1}^n f
\bigl(Z^{\mathrm
{stat}}_{i-1}, Z^{\mathrm{stat}}_i\bigr)
\Biggr]. %
\]
The function $f$ takes values in $(0,1)$ and is continuous, so for any
compact set $K \subseteq(0,\infty)^2$ we can find a constant $\gamma
\in(0,1)$ such that $f(x,y) \le\gamma$ on $K$. Take $K = K_1 \times
K_2$, where $K_1, K_2 \subseteq(0,\infty)$ are compact and have
strictly positive Lebesgue measure. Let $\tilde{N}(n) = \#\{1 \le i
\le n\dvtx  (Z^{\mathrm{stat}}_{i-1}, Z^{\mathrm{stat}}_i) \in K\}$. Then
\[
\mathbb{E} \Biggl[\prod_{i=1}^n
\bigl(Y_i^{\mathrm{stat}}\bigr)^{\alpha} \Biggr] \le\mathbb{E}
\bigl[\gamma^{\tilde{N}(n)} \bigr].
\]
We will bound $\tilde{N}(n)$ below by the number of renewals between
which there is a visit to $K$, that is,
\[
M(n) = \#\bigl\{k \ge1\dvtx  \tau_k \le n, \bigl(Z^{\mathrm{stat}}_{i-1},
Z^{\mathrm{stat}}_i\bigr) \in K\mbox{ for some } \tau_{k-1}
+ 1 < i \le\tau_k\bigr\}.
\]
This clearly has the effect of independently marking the renewal
intervals, as at~(\ref{eqnmarks}). Note that since $P(x, B) > 0$ for
any $x \in(0,\infty)$ and any Borel set $B \subseteq(0,\infty)$ of
positive Lebesgue measure, there is positive probability of visiting
$K$ between any two renewals. The result then follows from Lemmas~\ref
{lemmarkedrenewal} and~\ref{lemregen}.
\end{pf*}
\end{appendix}

\section*{Acknowledgments}
We would like to thank Jon Warren for very stimulating discussions
about invariant measures and fragmentations. We are grateful to the
referees for their careful reading of the paper and for their comments
which led to several improvements.

%\section*{Acknowledgments}

%\begin{supplement}[id=suppA]
%\slink[doi]{10.1214/00-AOPXXXXSUPP} %[doi,text={...}] - jei reikia
%suskaldyti doi
%\sfilename{aopXXXX\_supp.pdf}
%\end{supplement}
% imsref loaded by linak, 2015-01-05 10:56:07
%

\printaddresses
\end{document}